\documentclass[a4paper,reqno]{amsart}

\usepackage{euscript}
\usepackage{amsthm}
\usepackage{amssymb,amscd}
\usepackage{amsmath}
\usepackage{mathbbol}
\usepackage{graphicx}
\usepackage{color}
\usepackage{marginnote}
\usepackage{{hyperref}}

\newtheorem{theorem}{Theorem}
\newtheorem*{maintheorem}{Main Theorem}
\newtheorem{proposition}[theorem]{Proposition}

\newtheorem{corollary}[theorem]{Corollary}
\newtheorem{lemma}[theorem]{Lemma}

\newtheorem{definition}[theorem]{Definition}

 \theoremstyle{remark}
 \newtheorem{remark}[theorem]{\bf Remark}

\def\bdef{\begin{definition}}
\def\endef{\end{definition}}
\def\bthm{\begin{theorem}}
\def\ethm{\end{theorem}}
\def\blm{\begin{lemma}}
\def\elm{\end{lemma}}
\def\brm{\begin{remark}}
\def\erm{\end{remark}}
\def\bprop{\begin{proposition}}
\def\eprop{\end{proposition}}
\def\bcor{\begin{corollary}}
\def\ecor{\end{corollary}}
\def\be{\begin{eqnarray}}
\def\ee{\end{eqnarray}}
\def\beal{\begin{aligned}}
\def\enal{\end{aligned}}

\def\A{\mathbb A}
\def\C{\mathbb{C}}

\def\N{\mathbb{N}}

\def\R{\mathbb{R}}

\def\T{\mathbb{T}}
\def\Z{\mathbb{Z}}
\def\bbe {\mathbb{e}}

\def\cD{\mathcal{D}}
\def\E{\mathcal{E}}
\def\Ell {{\mathcal E}\!\ell\!\,\ell}

\def\cL{\mathcal{L}}
\def\cM{\mathcal M}

\def \V{\mathcal{V}}

\def \ie {{\it i.e.}, }
\def \ta {\tilde a}
\def \arcosh{{\rm arcosh}}
\def \arsinh{{\rm arsinh}}

\def\a{\alpha}

\def\d{\delta}
\def\e{\varepsilon}
\def \f {\varphi}

\def\g {\gamma}
\def\Gm{\Gamma}
\def\l{\lambda}
\def\Lb{\Lambda}

\def\Om{\Omega}
\def\om{\omega}
\def\th{\theta}
\def \z {\zeta}

\def \sn {{\rm sn}\,}
\def \cn {{\rm cn}\,}

\def \am {{\rm am}\,}

\def \Im {{\rm Im}\,}

\def\textb{\textcolor{blue}}

\allowdisplaybreaks

\title{On the local Birkhoff Conjecture for convex billiards}

\author{Vadim Kaloshin}
\address{Department of Mathematics, University of Maryland, College Park, MD, USA, \&
ETH Zurich, Institute for Theoretical Studies, Zurich
Switzerland}
\email{vadim.kaloshin@gmail.com}

\author{Alfonso Sorrentino}
\address{Dipartimento di Matematica, Universit\`a degli Studi di Roma ``Tor Vergata'', Rome, Italy.}
\email{sorrentino@mat.uniroma2.it}

\date{\today}
\begin{document}
\begin{abstract}
The classical Birkhoff conjecture claims that the boundary of a strictly
convex integrable billiard table is necessarily an ellipse (or a circle as a
special case). In this article we prove a complete local version of this conjecture:
a small integrable perturbation of an ellipse must be an ellipse. This extends and completes
the result in \cite{ADK}, where nearly circular domains 
were considered. {One of the crucial ideas in the proof is to extend  action-angle 
coordinates for elliptic billiards into complex domains (with respect to the angle), 
and to thoroughly analyze the nature of their complex singularities.} 
As an application, we are able to prove some spectral rigidity results for elliptic domains.
\end{abstract}

\maketitle 

\begin{center}
{\it \quad \quad \quad 
Dedicated to the memory of our thesis advisor John N. Mather: 
\newline a great  mathematician  and a remarkable person }
\end{center}

%%%%%%%%%%%%%%%%%%%%%%%%%%%%%%

\section{Introduction}

A {\it mathematical billiard} is a system describing the inertial motion of a point mass inside a domain, with elastic reflections at the boundary (which is assumed to have infinite mass). This simple model has been first proposed by G.D. Birkhoff as a mathematical playground where ``{\it the formal side, usually so formidable in dynamics, almost completely disappears and only the interesting qualitative questions need to be considered}'', \cite[pp. 155-156]{Birkhoff}.\\

Since then billiards have captured much attention in many different contexts, becoming a very popular subject of investigation. Not only is their law of motion very physical and intuitive, but billiard-type dynamics is ubiquitous. Mathematically, they offer models in every subclass of dynamical systems (integrable, regular, chaotic, etc.); more importantly, techniques initially devised for billiards have often been applied and adapted to other systems, becoming standard tools and having ripple effects beyond the field. \\

Let us first recall some properties of the billiard map. We refer to \cite{Siburg, Tabach1, Tabach} for a more comprehensive introduction to the study of billiards.\\

Let $\Omega$ be a strictly convex domain in $\R^2$ with $C^r$ boundary $\partial \Omega$,
with $r\geq 3$. The phase space $M$ of the billiard map consists of unit vectors
$(x,v)$ whose foot points $x$ are on $\partial \Omega$ and which have inward directions.
The billiard ball map $f:M \longrightarrow M$ takes $(x,v)$ to $(x',v')$, where $x'$
represents the point where the trajectory starting at $x$ with velocity $v$ hits the boundary
$\partial \Omega$ again, and $v'$ is the {\it reflected velocity}, according to
the standard reflection law: angle of incidence is equal to the angle of reflection (figure \ref{billiard}).

\begin{remark}
Observe that if $\Omega$ is not convex, then the billiard map is not continuous;
in this article we will be interested only in strictly convex domains (see Remark \ref{Matherglancing}).
Moreover, as pointed out by Halpern \cite{Halpern}, if the boundary is not at
least $C^3$, then the flow might not be complete.
\end{remark}

Let us introduce coordinates on $M$.
We suppose that $\partial \Omega$ is parametrized  by  arc-length $s$ and
let $\g:  [0, |\partial \Omega|] \longrightarrow \R^2$ denote such a parametrization,
where $|\partial \Omega|$ denotes the length of $\partial \Omega$. Let $\phi$
be the angle between $v$ and the positive tangent to $\partial \Omega$ at $x$.
Hence, $ M$ can be identified with the annulus $\A = [0,|\partial \Omega|] \times (0,\pi)$
and the billiard map $f$ can be described as

\begin{eqnarray*}
f: [0,|\partial \Omega|) \times (0,\pi) &\longrightarrow& [0,|\partial \Omega|) \times (0,\pi)\\
(s,\phi) &\longmapsto & (s',\phi').
\end{eqnarray*}

\begin{figure} [h!]
\begin{center}
\includegraphics[scale=0.23]{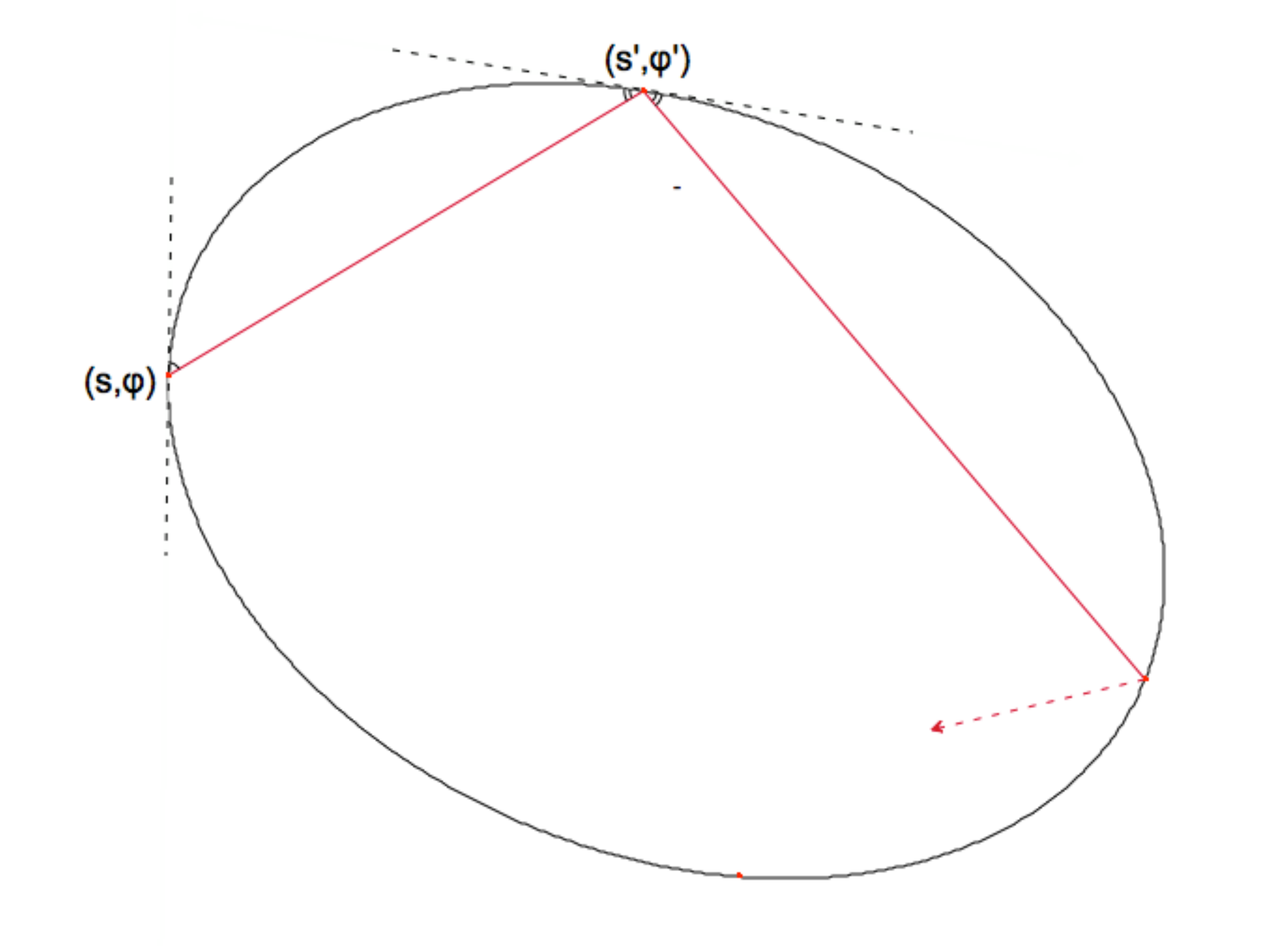}
\caption{}
\label{billiard}
\end{center}
\end{figure}

In particular $f$ can be extended to $\bar{\A}=[0,|\partial \Omega|] \times [0,\pi]$ by fixing
$f(s,0)=(s,0)$ and $f(s,\pi)= (s,\pi)$ for all $s$. Let us denote by
$$
\ell(s,s') := \|\g(s) - \g(s')\|
$$
the Euclidean distance between two points on $\partial \Omega$. It is easy to prove that
\begin{equation}\label{genfunctbill}
\left\{ \begin{array}{l}
\dfrac{\partial \ell}{\partial s}(s,s') = - \cos \phi \\
\\
\dfrac{\partial \ell}{\partial s'}(s,s') = \cos \phi'\,.\\
\end{array}\right.\\
\end{equation}

\begin{remark}
If we lift everything to the universal cover and introduce new coordinates
$(x,y)=(s, -\cos \phi) \in \R \times (-1,1)$, then the billiard map is a twist map
with $\ell$ as generating function and it preserves the area form $dx \wedge dy$. See \cite{Siburg, Tabach1, Tabach}.
\medskip 
\end{remark}

Despite the apparently simple (local) dynamics, the qualitative dynamical properties of billiard maps are extremely {non-local}. This global influence
on the dynamics translates into several intriguing {\it rigidity  phenomena}, which
are at the basis of several unanswered questions and conjectures (see, for example, \cite{ADK, Bialy, DKW, Gutkin, HKS, HKS2, Po, PT, RR, Siburg, SorDCDS,Tabach1, Tabach, Tre}). Amongst many, in this article we will address the question of classifying  {\it integrable billiards}, also known as {\it Birkhoff conjecture}. As an application of our main result, in subsection \ref{spectralellipses} we will also discuss certain spectral rigidity properties of ellipses.

%%%%%%%%%%%%%%%%%%%%%%%%%%%%%%%%%%%%%%

\subsection{Integrable billiards and Birkhoff conjecture}

The easiest example of billiard is given by a billiard in a disc $\cD$ (for example of radius $R$). It is easy to check in this case that the angle of reflection remains constant at each reflection (see also  \cite[Chapter 2]{Tabach}). 
If we denote by $s$  the arc-length parameter ({\it i.e.}, $s\in {\R}/  {\tiny 2\pi R} \Z$) and by $\theta \in (0,\pi/2]$ the angle of reflection, then
the billiard map  has a very simple form:
$$
f(s,\theta) = (s + 2R\, \theta,\; \theta).
$$
In particular, $\theta$ stays constant along the orbit and it represents an {\it integral of motion} for the map.
Moreover, this billiard enjoys the peculiar property of
having  the phase space -- which is topologically a cylinder --  completely foliated by homotopically non-trivial invariant curves ${\mathcal C}_{\theta_0}=\{\theta\equiv \theta_0\}$. These curves correspond to concentric circles of radii 
$\rho_0= R\cos \theta_0$ and are examples of what are called {\it caustics},  {\it i.e.}, (smooth and convex) curves with the property
that if a trajectory is tangent to one of them, then it will remain tangent after each reflection (see figure \ref{circle-billiard}).

\begin{figure} [h!]
\begin{center}
\includegraphics[scale=0.3]{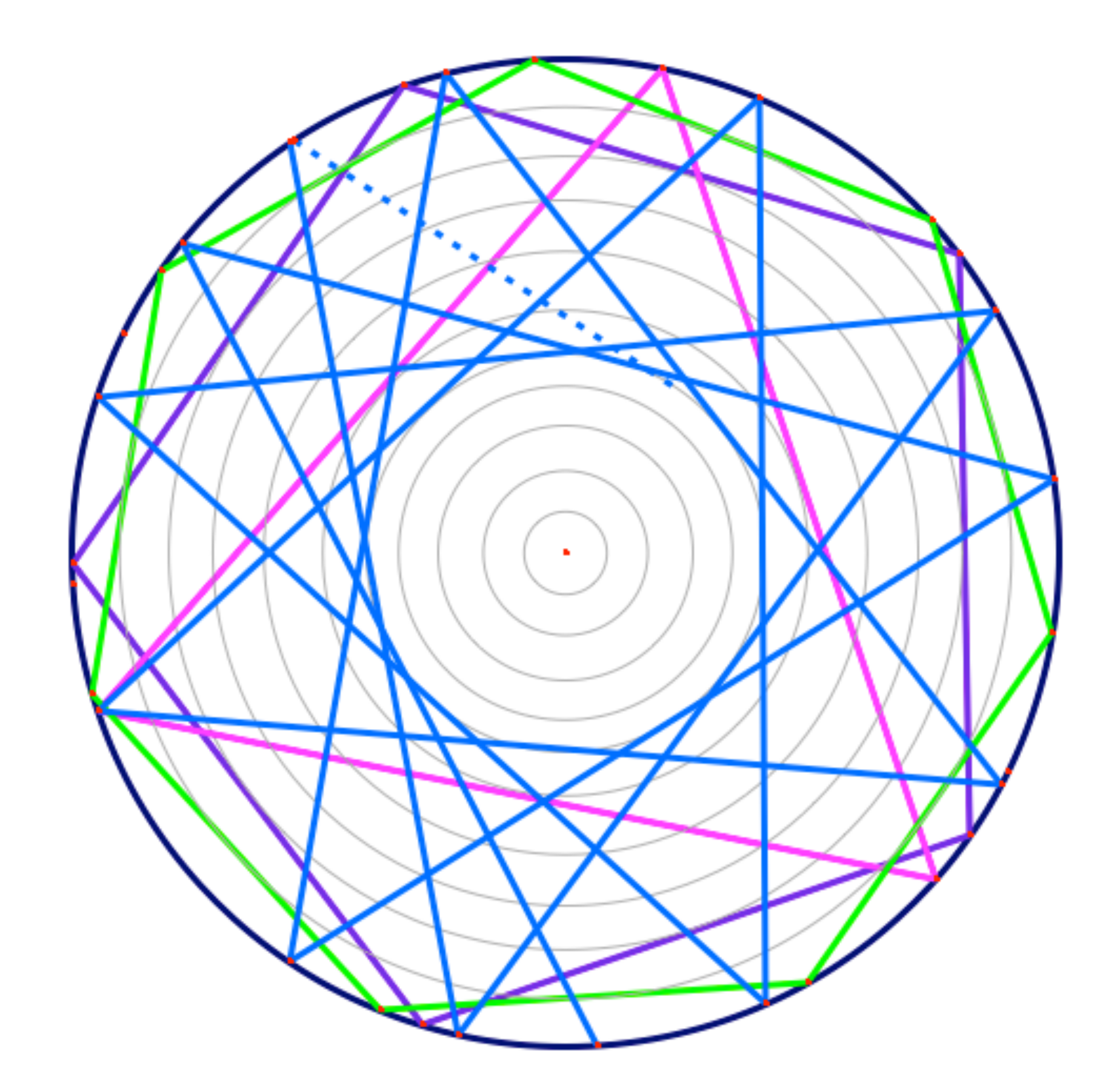}
\caption{Billiard in a disc}
\label{circle-billiard}
\end{center}
\end{figure}

A billiard in a disc is an example of an  {\it integrable billiard}. There are different  ways to define global/local integrability for billiards (the equivalence
of these notions is an interesting problem itself):
\begin{itemize}
\item[-] either through the existence of an integral of motion, globally or locally near the boundary (in the circular case  an integral of motion is given by $I(s,\theta)=\theta$),
\item[-] or through the existence of a (smooth) foliation of the whole phase space (or locally in a neighborhood of the boundary $\{\theta=0\}$), consisting of invariant curves of the billiard map; for example, in the circular case these are given by ${\mathcal C}_{\theta}$. This property translates (under suitable assumptions) into the existence of a (smooth) family of caustics, globally or locally near the boundary (in the circular case, the concentric circles of radii $R\cos \theta$).\\
\end{itemize}

In \cite{Bialy},  Misha Bialy proved the following  result concerning global integrability (see also \cite{Woi}):\\

\noindent {\bf Theorem (Bialy).}
\noindent {\it If the phase space of the billiard ball map is globally foliated by continuous
invariant curves which are not null-homotopic, then  it is a circular billiard.}\\

However, while circular billiards are the only examples of global integrable billiards,  integrability itself is still an intriguing open question.
One could  consider a  billiard in an ellipse: this is in fact  integrable (see Section \ref{sec:dynamicellipse}). Yet, the dynamical picture
is very distinct from the circular case: as it is showed in figure \ref{ellipse-billiard}, each trajectory which does not pass through
a focal point, is always tangent to precisely one confocal conic section, either
a confocal ellipse or the two branches of a confocal hyperbola (see for example
\cite[Chapter 4]{Tabach}). Thus, the confocal ellipses inside an elliptical billiards
are convex caustics, but they do not foliate the whole domain: the segment between
the two foci is left out (describing the dynamics explicitly is much more complicated: see for example \cite{Taba} and Section \ref{sec:dynamicellipse}). \\  

\begin{figure} [h!]
\begin{center}
\includegraphics[scale=0.22]{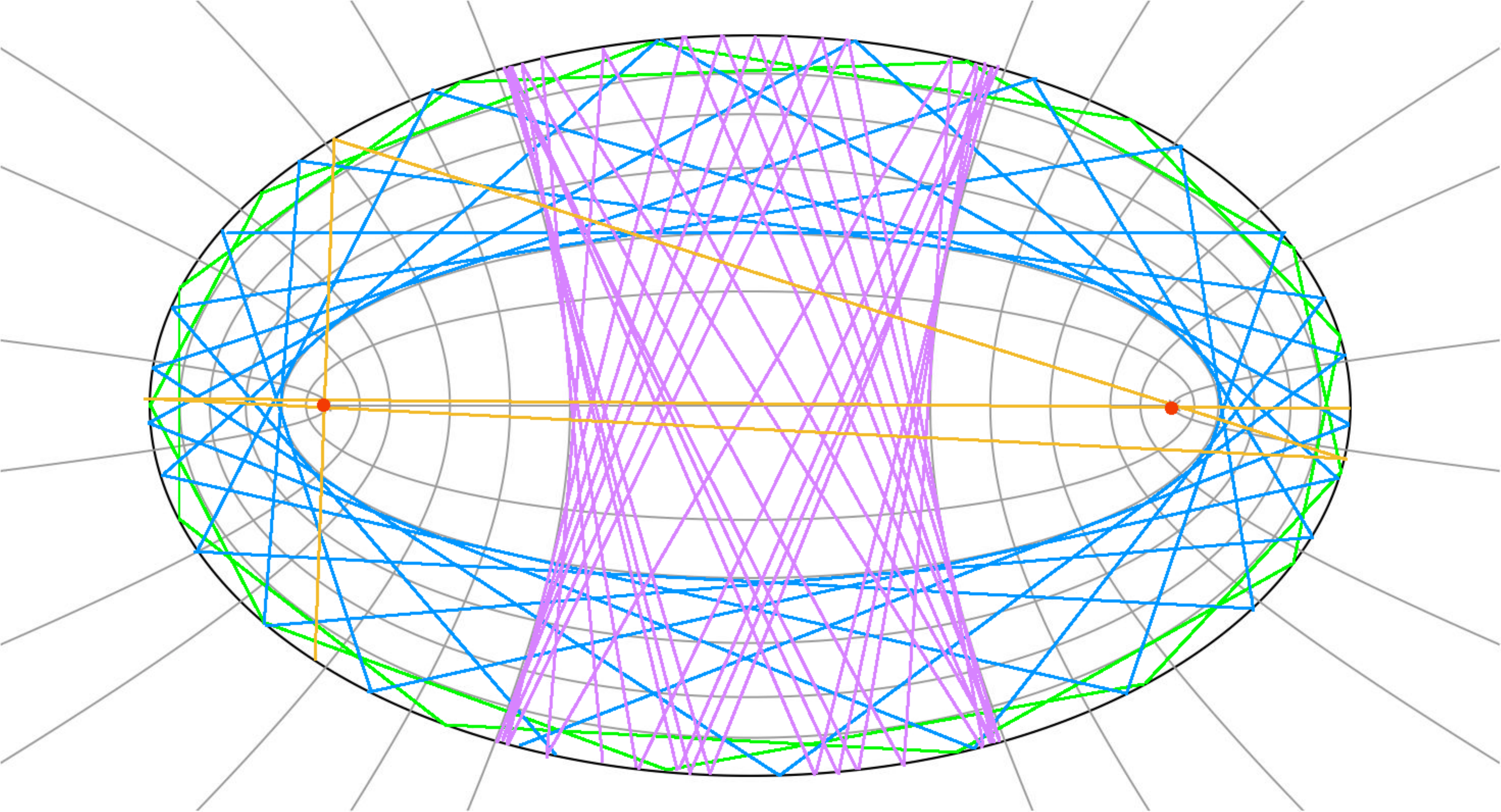}
\caption{Billiard in an ellipse}
\label{ellipse-billiard}
\end{center}
\end{figure}

\noindent{\bf Question (Birkhoff).} {\it Are there other examples 
of integrable billiards?}\\

\begin{remark}
Although some vague indications of this question can be found in \cite{Birkhoff}, to the best of our  knowledge, 
its first appearance as a conjecture was in a paper by Poritsky \cite{Po}, where the author attributes it to 
Birkhoff himself\footnote{Poritsky was Birkhoff's doctoral student and \cite{Po} was published several years 
after Birkhoff's death.}. Thereafter, references to this conjecture (either as {\it Birkhoff conjecture} or 
{\it Birkhoff-Poritsky conjecture}) repeatedly appeared in the literature: see, for example, 
Gutkin \cite[Section 1]{Gutkin}, Moser \cite[Appendix A]{Mo}, Tabachnikov \cite[Section 2.4]{Tabach1}, etc.
\end{remark}

\begin{remark}\label{Matherglancing}
In \cite{Mather82} Mather  proved the non-existence of caustics (hence, 
the non-integrability) if the curvature of the boundary vanishes at one point.
This observation justifies the restriction of our attention  to strictly convex domains. 
\end{remark}

\begin{remark}
{\it i}) Interestingly, Treschev in \cite{Tre} gives indication that there might exist analytic  
billiards, different from ellipses, for which the dynamics in a  neighborhood of the 
elliptic period-$2$ orbit is conjugate to a rigid rotation. These billiards can be seen as 
an instance of {\it local integrability}; however, this regime is somehow complementary 
to the one conjectured by Birkhoff. {Here one has local integrabilility in 
a neighborhood of an elliptic periodic orbit of period $2$, while Birkhoff  conjecture is 
related to  integrability in a neighborhood of the boundary. This gives an indication 
that these two notions of integrability do differ. }\\
{\it ii}) An algebraic version of this conjecture states that
the only billiards admitting polynomial (in the velocity) integrals are circles and ellipses.
For recent results in this direction, see \cite{BialyMironov}.
\end{remark}

Despite its long history and the amount of attention that this conjecture has captured, 
it remains still open. As far as our understanding of integrable billiards is concerned, 
the  most important related results are the above--mentioned theorem by Bialy 
\cite{Bialy} (see also \cite{Woi}), 
{a result by Innami\footnote{We are grateful to M. Bialy for pointing out this reference.} 
\cite{Innami}, in which he shows that the existence of caustics with rotation numbers accumulating to $1/2$ implies that 
the billiard must be an ellipse}\footnote{This regime of integrability is somehow diametrically opposed to ours, since we are interested in integrability near the boundary of the billiard domain.}, a result by Delshams and Ram\'irez-Ros \cite{DRR} 
in which they study entire perturbations of elliptic billiards and  prove that any 
nontrivial symmetric perturbation of the elliptic billiard is not integrable, 
{near homoclinic solutions}, and a very recent result by Avila, De Simoi 
and Kaloshin \cite{ADK} in which they show a perturbative version of 
this conjecture for  ellipses of small eccentricity.
\medskip

Let us introduce an important notion for this paper. 
\medskip

\begin{definition}\label{defrationalint}
{\rm (i)} We say $\Gamma$  is an {\rm{integrable} rational caustic} for the billiard 
map in $\Omega$, if the corresponding (non-contractible) invariant curve $\Gamma$ 
consists of periodic points; in particular, the corresponding rotation number is rational. 
\\
{\rm (ii)} If the billiard map inside $\Omega$ admits integrable rational caustics of rotation number $1/q$ for all $q > 2$, we say that 
$\Omega$ is {\rm  rationally integrable. }
\\
\end{definition}

\begin{remark}
A simple sufficient condition for rational integrability is the following (see \cite[Lemma 1]{ADK}).
Let ${\mathcal C}_{\Omega}$ denote the union of all smooth convex caustics of the billiard in  $\Omega$;
if the interior of ${\mathcal C}_{\Omega}$ contains caustics of rotation number 
$1/q$ for any $q > 2$, then $\Om$ is rationally integrable. \medskip 
\end{remark}

Our main result is the following.

\medskip 

\begin{maintheorem}[{\bf Local Birkhoff Conjecture}] \label{thm:main}  
Let $\E_{0}$ be an ellipse of eccentricity $0\leq e_0 <1$ and semi-focal distance $c$; {let $k\ge 39$}.  For every $K>0$, 
there exists $\e=\e(e_0,c,K)$ such that the following holds:
if $\Omega$ is a rationally integrable $C^{k}$-smooth domain so that $\partial\Omega$ is $C^{k}$-$K$-close and 
$C^1$-$\e$-close to $\E_0$, then $\Omega$ is an ellipse.\\
\end{maintheorem}

\begin{remark}
One could replace the smallness condition in the $C^1$-norm with a smallness condition with respect to the $C^0$-topology (this can be showed by using interpolation inequalities and the convexity of the domains)\footnote{This remark was suggested to the authors by Camillo De Lellis.}.\\
\end{remark}

\begin{remark}
In \cite{HKS} we  prove a similar rigidity statement for 
a different type of rational integrability. 
{Namely, {we describe an algorithm to prove that for any given }{$q_0\geq 3$} there exists $e_0=e(q_0)>0$ {such that every sufficiently smooth 
perturbation of $\E_e$, with $0 < e<e_0$},} having integrable rational caustics 
of rotation numbers $p/q$, for all $0 < p/q < 1/q_0,$
must be an ellipse. {This algorithm is conditional on checking the invertibility of finitely many explicit matrices, which we   prove  in the cases $q_0=3,4,5$.}
 Observe that  the analysis in \cite{HKS} 
only applies to ellipses of small eccentricity as in \cite{ADK}, {since Taylor expansions with respect to $e$ are needed in order to get higher order (integrability) conditions.}\\
\end{remark}

One of the crucial ideas to extend the analysis beyond the almost circular case in \cite{ADK},  is to consider analytic extensions of   
the action-angle coordinates of the elliptic billiard (more specifically, of the boundary parametrizations induced by each integrable caustic) 
and to study their singularities (see Section \ref{sec:adapted-basis}). These functions  can be explicited expressed in terms of elliptic 
integrals and Jacobi elliptic functions (see subsection \ref{subsec:elliptcinteg}). {This analysis
will be exploited to define a {\it dynamically-adapted basis} for 
$L^2(\R/2\pi\Z)$, which will provide  the main framework to carry out our analysis. See subsection \ref{schemeourproof} for a more detailed 
description of the scheme of the proof.

In addition to this, in Appendix \ref{affineflowideas} we propose a possible strategy to use the {\it affine length shortening (ALS) flow} 
(see, for instance, \cite{SaTa}) as a potential approach to prove the {\it global} Birkhoff conjecture. Our proposal is based on the fact 
that the ALS flow evolves any convex domain with  smooth boundary into an ellipse in finite time.

\medskip

\subsection{Applications for spectral rigidity of ellipses} \label{spectralellipses}
In this subsection we describe an interesting application of our Main Theorem 
to spectral rigidity properties of ellipses\footnote{This was suggested to the authors 
by Hamid Hezari.}.

Let $\Omega$ be a smooth strictly convex (planar) domain.
While the dependence of the dynamics on the geometry of the domain is well perceptible, an intriguing challenge is to understand to  which extent dynamical 
information can be used to reconstruct the shape of the domain.  A particular 
interesting problem in this direction is to unravel  which information on the geometry 
of the billiard domain, the set of periodic orbits does encode. More ambitiously, 
one could wonder whether a complete knowledge of this set allows one to 
reconstruct the shape of the billiard and hence the whole of its dynamics. Several 
results in this direction (and in related ones) are contained, for instance, in 
\cite{Bangert, DKW, GM, HZ, HKS, MM, MM2, Popov, PT, Siburg, SorDCDS, Zelditch}.\\

Let us start by introducing the {\it Length Spectrum} of a domain $\Om$.

\begin{definition}[{\bf Length Spectrum}] 
Given a domain $\Omega$, the length spectrum of $\Omega$ is given by the set  of 
lengths of its periodic orbits, counted with multiplicity:
\[
\mathcal  L(\Omega) := \mathbb N 
\{\text{ lengths of closed geodesics in }\Omega\} \cup 
\mathbb N \, |\partial \Omega|,
\]
where $|\partial \Omega|$ denotes the length of the boundary of $\Omega$.\\
\end{definition}

A remarkable relation exists between the length spectrum of a billiard in 
a convex domain $\Omega$  and the spectrum of the Laplace operator in $\Omega$ 
with Dirichlet boundary conditions (similarly for Neumann boundary conditions):
\begin{equation}\label{DirichletProblem}
\left\{
\begin{array}{l}
\Delta f = \lambda f \quad \text{in}\; \Om \\
f|_{\partial \Om} = 0.\end{array}\right.\\
\end{equation}
From the physical point of view, the eigenvalues $\lambda$'s
are the eigenfrequencies of the membrane $\Om$ with
a fixed boundary. K. Andersson and R. Melrose \cite{AM}
proved the following relation between the Laplace spectrum 
and the length spectrum. Call the function 
\[
w(t):=\sum_{\lambda_i \in spec \Delta}
\cos (t \sqrt{-\lambda_i}),
\]
the wave trace.
Then, the wave trace $w(t)$  is a well-defined 
generalized function (distribution) of $t$, smooth 
away from the length spectrum, namely, 
\begin{equation}\label{AndersonMelroseformula}
\mbox{sing. \!\!\!\! supp.} \big( w(t) \big)\subseteq  \pm \cL(\Om) \cup \{0\}.
\end{equation}
So if $l > 0$ belongs to the singular support of this distribution, then there exists 
either a closed billiard trajectory of length $l$, or a closed geodesic of length 
$l$ in the boundary of the billiard table.

Generically, equality holds in \eqref{AndersonMelroseformula}. 
More precisely, if no two distinct orbits have the same length and 
the Poincar\'e map of any periodic orbit is non-degenerate, then the singular 
support of the wave trace coincides with $\pm \cL(\Om) \cup \{0\}$ (see e.g. 
\cite{Popov}). This theorem implies that,  at least for generic domains, one can 
recover the length spectrum from the Laplace one. \\

This relation between periodic orbits and spectral properties of the domain, immediately recalls a more famous spectral problem (probably the most famous): 
{\it Can one hear the shape of a drum?}, as formulated in a very suggestive way 
by Mark Kac \cite{Kac} (although the problem had been already stated by Hermann 
Weyl). More precisely: is it possible to infer information about the shape of a drumhead 
({\it i.e.}, a domain) from the sound it makes ({\it i.e.}, the list of basic harmonics/ 
eigenvalues of the Laplace operator with Dirichlet or Neumann boundary conditions)?
This question has not been completely solved yet: there are several negative 
answers (for instance by Milnor \cite{Milnor} and Gordon, Webb, and Wolpert 
\cite{GordonWebbWolpert}), as well as some positive ones.  

Hezari and Zelditch, {going in the affirmative direction,} proved in~\cite{HZ} 
that, given an ellipse ${\mathcal E}$, any one-parameter $C^\infty$-deformation
$\Om_\e$ which preserves the Laplace spectrum (with respect to either Dirichlet 
or Neumann boundary conditions) and the $\mathbb Z_2\times\mathbb Z_2$ 
symmetry group of the ellipse has to be \emph{flat} ({\it i.e.}, all derivatives have 
to vanish for $\e = 0$). {Popov--Topalov \cite{PT} recently extended these results 
(see also \cite{Zelditch}).} Further historical remarks on the inverse spectral problem 
can  be also found in~\cite{HZ}. In \cite{OPS1,OPS2,OPS3} Osgood, Phillips and Sarnak 
showed that isospectral sets are necessarily compact in the $C^\infty$  topology
in the space of domains with $C^\infty$ boundary. In~\cite{Sarnak} Sarnak 
conjectures that the set of smooth convex domains isospectral to a given smooth 
convex domain is finite (for a partial progress on this question,  see \cite{DKW}). \\

One of the difficulties in working with the length spectrum   is that all of these information 
on the periodic orbits come in a non-formatted way. For example, 
we lose track of the rotation number corresponding to each length. A way to 
overcome this difficulty is to ``organize'' this set of information in a more systematic 
way, for instance by associating to each length the corresponding rotation number.  
This new set  is called the {\it Marked Length Spectrum} of $\Omega$ and denoted 
by $\cM\cL_\Om$:
\[
\cM\cL(\Omega) :=  \left\{ \left({\mbox{length}(\g), {\mbox{rot}(\g)}}\right): \; 
\mbox{$\g$ periodic orbit of the billiard  in $\Omega$} \right\},
\]
where $\mbox{rot}(\g)$ denotes the rotation number of $\gamma$.\footnote{
In the case of negatively curved surfaces without boundary the marked length 
spectrum consists of pairs of homotopy classes and length of the shortest geodesic 
in that homotopy class. Guillemin and Kazhdan \cite{GK} proved local rigidity 
with respect to this marked length spectrum. Global version of this result was 
obtained by Otal \cite{Ot} and Croke \cite{Cr}.}\\

One could also refine this set of information by considering 
not the lengths of all orbits, but selecting some of them.  More precisely, for each rotation number $p/q$ in lowest terms, one could consider the maximal length 
among those having rotation number $p/q$. We call this map the {\it Maximal Marked 
Length Spectrum} of $\Omega$, namely $\mathcal M\mathcal L^{\rm max}(\Omega) : 
\mathbb Q \cap [0, 1/2]\ \to \mathbb R$ given by:
\begin{eqnarray*}
\cM\cL^{\rm max}_\Om({p}/{q}) =   \max \Big \{ \mbox{lengths of periodic orbits with rot. number}\; p/q \Big \}.
\end{eqnarray*}

\begin{remark}\label{remarkbetaMLS}
The maximal marked length spectrum is  closely related to {\it Mather's minimal average action} (or {\it Mather's $\beta$-function}) of the associated billiard map in 
the domain, as it was  pointed out in \cite{Siburg}. Briefly speaking, this function -- which can be defined for any exact area preserving twist map, not necessarily 
a billiard map -- associates to any fixed rotation number (not only rational ones) 
the minimal average action of orbits with that  rotation number (whose existence, 
inside a suitable interval, is ensured by the twist condition). These action-minimizing 
orbits are of particular interest from a dynamical point of view and play a key-role 
in what is nowadays called {\it Aubry-Mather theory}; we refer the reader to 
\cite{Bangert, MatherForni, Siburg, SorLecNotes} for a presentation of this topic.  
\\
In the case of billiard maps, since the {\it action} coincides (up to a negative sign) 
with the euclidean length of the segment joining  two subsequent rebounds, we 
have that the minimal average action of periodic orbits can be expressed in terms 
of the maximal marked length spectrum; namely:
         \begin{equation}\label{betaandMLS}
         \beta_\Omega(p/q) = - \frac{1}{q} {\mathcal ML}^{\rm max}_\Om({p}/{q}) 
         \qquad \forall\; 0<p/q\leq 1/2.
         \end{equation}
In particular, this object encodes many interesting dynamical information on 
the billiard map.  For example, using the result in \cite{Mather90}, one can deduce 
that $\beta$ is differentiable at $p/q$ if and only if there exists a rational caustic 
of rotation number $p/q$.  See \cite{Siburg} for a detailed presentation of this and 
many other properties.
\end{remark}

\medskip

Let us now address the following question.
\medskip

\noindent{\bf Question.}  {\it Let $\Omega_1$ and $\Omega_2$ be two strictly convex planar domains with smooth boundaries and assume that they have the same maximal marked Length spectrum, namely $\cM\cL^{\rm max}_{\Om_1} \equiv \cM\cL^{\rm max}_{\Om_2}$ (or equivalently, $\beta_{\Omega_1} \equiv \beta_{\Omega_2}$). 
Is it true that $\Omega_1$ and $\Omega_2$ are isometric?}

\medskip

\begin{remark}
It is known that if $\Omega$ has the same marked length spectrum of a disc, then it is indeed a disc; for a proof of this result, see for example \cite[Corollary 3.2.17]{Siburg}. 
Another proof can be obtained by looking only at the Taylor coefficients of the $\beta$-function at $0$ (which are related to the so-called {\it Marvizi-Melrose invariants}); it turns out that the first and the third order coefficients always satisfy an inequality, which becomes an equality if and only if the domain is a disc (see \cite[Section 8]{MM} and \cite[Corollary 1]{SorDCDS}).
\end{remark}

\medskip

It would be interesting to find a similar characterization for elliptic billiards, namely that 
the maximal marked length spectrum (resp., the $\beta$-function) univocally determines ellipses amongst all possible Birkhoff billiards.

In \cite[Proposition 1]{SorDCDS}, by looking at the Taylor expansion of 
the $\beta$-function at $0$ (actually, only at the first and third order coefficients), 
it was pointed out a much weaker result, namely that the isospectrality condition 
determines  univocally a given ellipse within the family of  ellipses (up to rigid motions,  
{\it i.e.}, the composition of a translation and a rotation)).\\

From our Main Theorem, we can now deduce the following spectral rigidity results 
for ellipses.

\begin{corollary}[\bf Local length--spectral uniqueness 
of ellipses]
\label{cor:length-rig} 
Let $\Omega$ be a smooth strictly convex domain $\Omega$ sufficiently close to 
an ellipse.
\begin{itemize}
\item[i)]  If $\Omega$ has  the same maximal marked length spectrum  (or Mather's $\beta$-function)
of an ellipse, then it is an ellipse. 
\item[ii)] If its Mather's $\beta$-function  is differentiable at all rationals $1/q$ with $q\geq 3$, then $\Omega$ is an ellipse.
\end{itemize}
\end{corollary}

\medskip

Moreover, the following spectral rigidity result holds.

\begin{corollary}[\bf Spectral rigidity of ellipses]
\label{cor:spectra-rig} 
\hfill
\begin{itemize}
\item[i)] Ellipses are {\em (maximal) marked-length-spectrally rigid}, meaning that if 
$\Omega_t$ is a smooth deformation of an ellipse which keeps fixed the (maximal) marked length spectrum, 
then it consists of a rigid motion. 
\item[ii)] Ellipses are {\em length-spectrally rigid}, meaning that if 
$\Omega_t$ is a smooth deformation of an ellipse which keeps fixed the  length spectrum, 
then it consists of a rigid motion. \\
\end{itemize}
\end{corollary}

\begin{proof} ({Corollary \ref{cor:length-rig}})
Assertion i) follows from assertion ii), using  \eqref{betaandMLS} and recalling that the $\beta$-function of an ellipse is differentiable in $[0,1/2)$, since the corresponding billiard map is integrable. 
As for the proof of ii), it follows from the differentiability assumptions on $\beta$ and from what recalled at the end of Remark \ref{remarkbetaMLS} (see also \cite{Mather90, Siburg}), that there exist  integrable rational  caustics for all rotation number $1/q$ for any $q\geq 3$. Hence
our billiard is rationally integrable (see Definition \ref{defrationalint}).
Applying the Main Theorem, since $\Omega$ is close to an ellipse, then it must be an ellipse.
\end{proof}

\medskip

\begin{proof} ({Corollary  \ref{cor:spectra-rig}})
Assertion i) follows from Corollary \ref{cor:length-rig} ii) and the fact that the $\beta$ function (equivalently, the maximal marked length spectrum) univocally determines a given ellipse within the family of  ellipses (up to rigid motions); see \cite[Proposition 1]{SorDCDS}.\\
To prove assertion ii), one needs to use  \cite[Proposition 3.2.2]{Siburg}, which shows that  a
$C^0$ iso-length spectral deformation is necessarily an iso-marked length spectral deformation. 
Then, the claim  follows by applying  i).
\end{proof}

\medskip

\subsection{Organization of the article}
For the reader's convenience, here follows a brief description of how the article is organized.

In Section \ref{sec:dynamicellipse} we describe our setting and introduce elliptic coordinates (see subsection \ref{secellipticpolarcoords}), while 
in Section \ref{sec:action-angle-ellipse}  we recall some  definitions and some needed properties of elliptic integrals and elliptic functions (see subsection \ref{subsec:elliptcinteg})
and use them to provide a more precise  description of the billiard dynamics inside an ellipse (see subsection \ref{subsec:ellipticdynamics}).

In Section \ref{secstrategy} we outiline the scheme  of the proof of our Main Theorem,
 both for perturbations of circular billiards (see subsection \ref{schemecircle}) and for perturbations of general elliptic ones (see subsection \ref{schemeourproof});  we refer to this latter subsection for a detailed description of the contents of Sections \ref{sec:action-angle-ellipse}--\ref{sec:proof-MainThm}.

 In order to make the presentation clearer and easier to follow, we deferred several proofs of technical claims and some complementary material  to Appendices \ref{App:paramellipses}--\ref{AppendixTechnical}. Finally, in Appendix \ref{affineflowideas} we outline a possible strategy  to approach the {global} Birkhoff conjecture, by means of the {affine 
length shortening flow}.\

\medskip 

\subsection{Acknowledgements.} 
VK acknowledges partial support of the NSF grant DMS-1402164 and the hospitality of the ETH Institute for Theoretical Studies and 
the support of  Dr. Max R\"ossler,  the Walter Haefner Foundation and the ETH Zurich Foundation. 
AS acknowledges the partial support of the Italian MIUR research grant: PRIN- 2012-74FYK7
``{\it Variational and perturbative aspects of nonlinear differential problems}''.
VK is grateful to Jacopo De Simoi and Guan Huang for  useful discussions. AS would like to thank Pau Martin 
and Rafael Ram\'irez-Ros for  useful discussions during his stay at UPC.
The authors are also indebted to  Hamid Hezari, whose  valuable remarks led to Corollaries \ref{cor:length-rig} 
and \ref{cor:spectra-rig}. {Finally,  the authors wish to express their sincere gratitude to a referee for really 
careful reading of the paper and many useful suggestions, which led to significant improvements of 
the exposition and clarity of the proof.}

\vspace{10 pt}

%%%%%%%%%%%%%%
\section{Notation and Setting} \label{sec:dynamicellipse}

Let us consider the ellipse 
$$
{\E_{e_0,c}} =\left \{(x,y)\in \R^2:\; \frac{x^2}{a^2} + \frac{y^2}{b^2} = 1\right\},
$$
centered at the origin and with semiaxes of lenghts, respectively, $0<b\leq a$; 
in particular $e_0$ denotes its eccentricity, given by $e_0=\sqrt{1-\frac{b^2}{a^2}} \in [0,1)$ 
 and $c=\sqrt{a^2-b^2}$ the semi-focal distance.
{Observe that when $e_0=0$, then $c=0$ and $\E_{0,0}$ degenerates to a $1$-parameter 
family of circles centered at the origin.} \\

The family of confocal elliptic caustics  in $\E_{e_0,c}$ is given by 
(see also figure \ref{ellipse-billiard}):
\begin{equation}\label{caustic}
C_{\l} = \left\{
(x,y)\in \R^2: \; \frac{x^2}{a^2-\l^2} +  \frac{y^2}{b^2-\l^2}=1
\right\} \qquad 0<\l<b.
\end{equation}
Observe that the boundary itself corresponds to $\l=0$, while the limit case $\l=b$ 
corresponds to the the two foci ${\mathcal F}_{\pm}=(\pm \sqrt{a^2-b^2},0)$.
{Clearly, for $e_0=0$ we  recover the family of concentric circles described in Figure \ref{circle-billiard}.}

\medskip

\subsection {\bf Elliptic polar coordinates}  \label{secellipticpolarcoords}
A more convenient coordinate frame for addressing this question is provided by the so-called 
{\it elliptic polar coordinates} (or, simply, elliptic coordinates)
$(\mu,\f) \in \R_{\geq 0} \times \R/2\pi \Z$, given by:
$$
\left\{
\begin{array}{l}
x= c \cosh \mu \, \cos\f \\
y= c \sinh \mu \, \sin\f,
\end{array}
\right.
$$
where $c=\sqrt{a^2-b^2}>0$ represents the semi-focal distance {(in the case $e_0=0$, this parametrization degenerates to the usual polar coordinates)}. 
Observe that for each $\mu_*> 0$, the equation $\mu\equiv \mu_*$ represents  a confocal ellipse, while for each $\f_* \in [0,2\pi)\setminus\{0,\frac{\pi}{2}, \pi, \frac{3\pi}{2}\}$ the equation $\f \equiv \f_*$ corresponds to one of the two branches of a confocal hyperbola; these grid-lines are mutually orthogonal. Moreover, the degenerate cases $\mu_*\equiv 0$ and $\f_*\equiv 0,\frac{\pi}{2}, \pi, \frac{3\pi}{2}$ describe, respectively, the (cartesian) segment $\{-c\leq x \leq c\}$, and the (cartesian) half-lines $\{x\geq c\}$, $\{y\geq 0\}$, $\{x\leq -c\}$ and $\{y\leq 0\}$.\\

Therefore, in these elliptic polar coordinates  $\E_{e_0,c}$ becomes:
$$
\E_{e_0,c} =  \left\{(\mu_0,\f),\; \f \in  \R/2\pi \Z \right\},
$$
where  $\mu_0 =\mu_0(e_0):= \arcosh\left({1}/{e_0}  \right)$ {(the dependence on $c$ is in the definition of the coordinate frame)}.\\

\medskip

Let us denote by $\Ell$ the set of ellipses in $\R^2$ 
{with circles being degenerate points}.
This is a $5$-dimensional family of strictly convex curves  parametrized, for example, by the cartesian coordinates 
of its centre $(x_0,y_0)\in \R^2$, the semi-focal distance $c>0$, 
the parameter $\mu_0>0$ corresponding to the eccentricity, 
and the angle $\theta\in [0,\pi)$ between the major semiaxis 
and the $x$-axis (notice that $\theta$ is not well defined 
for circles).  More specifically, for each 
$(x_0,y_0,c,\mu_0,\theta)  \in \R^2 \times (0,+\infty)^2\times [0,\pi)$ 
we associate the (parametrized) ellipse 
{\footnotesize
\begin{eqnarray}\label{paramellipse}
\E(x_0,y_0,c,\mu_0,\theta) :=
\left\{ 
\left(
\begin{array}{c}
x-x_0\\
y-y_0
\end{array}
\right) =
\left(
\begin{array}{cc}
\cos \theta  & - \sin\theta\\
\sin \theta & \cos\theta
\end{array}
\right)
\left(
\begin{array}{c}
c \cosh \mu_0 \cos\f\\
c \sinh \mu_0 \sin\f
\end{array}
\right),\; \f\in [0,2\pi)\right\}.
\end{eqnarray}
}

\smallskip

In the following we will use the shorthand $\E_{e_0,c}$ for 
$\E(0,0, c, \mu_0(e_0), 0)$. In particular, $\E_{0,c}$ consists of 
a {1-parameter family of circles centered at the origin}.

\medskip

\section{Action-angle coordinate of elliptic billiards}\label{sec:action-angle-ellipse}
Here we define and study action-angle coordinates for elliptic billiards. 

\subsection {Elliptic integrals and Jacobi elliptic functions} \label{subsec:elliptcinteg}
Let us recall some basic definitions on elliptic integrals and elliptic functions that will be used in the following; we refer the reader, for instance, to \cite{Akhiezer} for a more comprehensive presentation. \\

Let $0\leq k<1$. We define the following {\it elliptic integrals}.\\

\begin{itemize}
\item {\it Incomplete elliptic integral of the first kind}:
$$
F(\f;k) := \int_0^\f \frac{1}{\sqrt{1- k^2 \sin^2 \f}} d\f.
$$
In particular, $k$ is called the {\it modulus} and $\f$ the {\it amplitude}. Moreover, the quantity $k':=\sqrt{1-k^2}$ is often called
the {\it complementary modulus}. Observe that for $k=0$ we have $F(\f;0)=\f$; on the other hand, $F(\varphi;1)$ has a pole at $\varphi=\frac{\pi}{2}$.\\

\item {\it Complete elliptic integral of the first kind}:
$$
K(k) = F\left(\frac{\pi}{2};k\right).\\
$$

\end{itemize}

\medskip

Let us  recall that an {\it elliptic function} is a doubly-periodic meromorphic function, {\it i.e.}, it is periodic in two directions and hence, it is determined by its values on a fundamental parallelogram. Of course, a non-constant elliptic function cannot be holomorphic, as it would be a bounded entire function, and by Liouville's theorem it would be constant.
In particular, elliptic functions must have at least two poles  in a fundamental parallelogram (counting multiplicities); it is easy to check, using the periodicity, that a contour integral around its boundary must vanish, implying that the residues of all simple poles must cancel out.

\medskip

{\it Jacobi Elliptic functions}  are obtained by inverting incomplete elliptic integrals of the first kind. More specifically, 
let
\begin{eqnarray}\label{JEF}
u=F(\f;k) = \int_0^\f \frac{d\tau}{\sqrt{1- k^2 \sin \tau}} 
\end{eqnarray}
($u$ is often called the {\it argument}). If $u$ and $\f$ are related as above (we can also write $\f =  \am (u;k)$, called the {\it amplitude} of $u$) then we define the Jacobi elliptic functions as:
\begin{eqnarray*}
\sn (u;k) &:=& \sin (\am (u;k)) \\
\cn (u;k) &:=& \cos (\am (u;k)).
\end{eqnarray*}

\medskip

\begin{remark}\label{rmk14}
These two elliptic functions have periods $4K(k)$ (in the real direction) and $4iK(k')$ (in the imaginary direction). Moreover, they have two simple poles: at $u_1=iK(k')$, with residue, respectively, $1/k$ and  $-i/k$, and  at $u_2=2K(k)+iK(k')$ with residue, respectively, $-1/k$ and $i/k$.
\end{remark}

\medskip

\subsection{Elliptic billiard dynamics and caustics} \label{subsec:ellipticdynamics}
Now we want to provide a more precise description of the billiard dynamics in $\E_{e_0,c}$.

The following result has been proven in \cite{CF} (see also \cite[Lemma 2.1]{DCR}).

\medskip

\begin{proposition}\label{prop1}
Let  $\l \in (0, b)$ and let
$$
k_\l^2:= \frac{a^2-b^2}{a^2-\l^2} \quad {\rm and} \quad
\d_\l := 2\, F( \arcsin (\l/b); k_\l).
$$
Let us denote, in cartesian coordinates, $q_\l(t) := (a\, \cn (t;k_\l), b\, \sn (t;k_\l))$.
Then, for every $t\in [0, 4K(k_\l))$  the segment joining $q_\l(t)$ and $q_\l(t+\d_\l)$ is tangent to the caustic $C_\l$.
\end{proposition}

\medskip

Observe that:
\begin{itemize} 
\item $k_\l$ is a strictly increasing function of $\l \in (0,b)$; in particular $k_\l \rightarrow e_0$ as $\l \rightarrow 0^+$, while $k_\l \rightarrow1$ as $\l\rightarrow b^-$.
Observe that $k_\l$ represents the eccentricity of the ellipse $C_\l$.
\item
$\delta_\l$ is also a strictly increasing function of $\l \in (0,b)$; in fact, $F(\f;k)$ is clearly strictly increasing in both $\f$ and $k \in [0,1)$. Moreover, 
$\d_\l \rightarrow 0$ as $\l\rightarrow 0^+$, and $\d_\l \rightarrow +\infty$ as $\l\rightarrow b^-$.
\end{itemize}

\medskip 

\begin{remark}
Using  elliptic polar coordinate, one can  easily check that 
$\tanh^2 \mu =  1 - \frac{a^2-b^2}{a^2-\l^2}$ and therefore
\begin{eqnarray*}
k(\mu) =  \sqrt{1- \tanh^2 \mu}  = \frac{1}{\cosh \mu},
\end{eqnarray*}
which is exactly the eccentricity of the confocal ellipse of parameter $\mu$.
\end{remark}

\bigskip

Let us now consider the parametrization of the boundary induced by the dynamics on  the caustic $C_\l$:
\begin{eqnarray*} 
Q_\l : \R/2\pi\Z &\longrightarrow& \R^2 \nonumber \\
\theta &\longmapsto& q_\l \left(\frac{4K(k_\l)}{2\pi} \, \theta \right). \nonumber
\end{eqnarray*}

We define the {\it rotation number} associated to the caustic $C_\l$ to be
\begin{equation}\label{defrotnumber}
\omega_\l:= \frac{\d_\l}{4K(k_\l)} = 
\frac{F(\arcsin (\l/b); k_\l)}{2K(k_\l)}.
\end{equation}
In particular $\omega_\l$ is strictly increasing as a function of $\l$ and
$\omega_\l \longrightarrow 0$ as $\l\rightarrow 0^+$, while $\omega_\l \rightarrow \frac{1}{2}$ as $\l\rightarrow b^-$.\\

\medskip

It is easy to deduce  from the above expressions that, in elliptic coordinates $(\mu, \f)$, the boundary parametrization induced by the caustic $C_\l$ is given by 
\begin{equation}\label{philambda}
S_\l(\theta):= (\mu_\l(\theta), \f_\l(\theta)) =\left(\mu_0, \am \left( \frac{4K(k_\l)}{2\pi}\,\theta; k_\l \right) \right).
\end{equation}
More precisely, the orbit starting at $S_\l(\theta)$ and tangent to $C_\l$, hits the boundary  at $S_\l(\theta + 2\pi\, \omega_\l)$.

%%%%%%%%%%%%%%%%%%%%%%%%%%%%

\section{Outline of the proof} \label{secstrategy}

In this section we provide a description of the strategy that we will follow to prove our Main Theorem.

\medskip 

\subsection{A scheme for proving Main Theorem  for  circular billiards. } \label{schemecircle}
For small eccentricities Main Theorem 
 was proven in \cite{ADK} and 
we now  describe  the proof therein. Let us start with the simplified 
setting of integrable infinitesimal deformations of a circle. This provides 
an insight into the strategy of the proof in the general case.  

\medskip 

Let $\E^{\rho_0}_{0,0}$ be a circle centered at the origin and radius $\rho_0>0$.
Let $\Om_\e$ be a one-parameter family of deformations
given in the polar coordinates $(\rho, \f)$ by   
$$
\partial\Om_\e=\{(\rho,\varphi)=(\rho_0+\e
\rho(\varphi)+O(\e^2),\varphi)\}.
$$
Consider the Fourier expansion of $\rho$ :
\begin{align*}
  \rho ( \varphi) =\rho'_0 + \sum_{k > 0} \rho_{k}\sin (k\varphi)+ \rho_{-k} \cos (k\varphi).
\end{align*}

\begin{theorem}[Ram\'irez-Ros \cite{RR}] \label{thm:RR} 
If $\Om_\e$ has an integrable rational caustic $\Gm_{1/q}$ 
of rotation number $1/q$, for any $\e$ sufficiently small, then 
we have $\rho_{kq} = 0$ for any integer $k$.
\end{theorem}
\medskip

Let us now assume that the domains $\Om_\e$ are $2$-rationally integrable for 
all sufficiently small $\e$ and ignore for a moment dependence of parametrisation: 
then the above theorem implies that $\rho_k = \rho_{-k} = 0$ for $k > 2$, \ie
\begin{align*}
  \rho(\varphi)&=\rho'_0+\rho_{-1}\cos \varphi+\rho_1\sin \varphi+
            \rho_{-2}\cos 2\varphi+\rho_2\sin 2\varphi\\
          &= \rho^*_0+\rho_1^*\cos (\varphi-\varphi_1)+
            \rho_2^*\cos 2(\varphi-\varphi_2)
\end{align*}
where $\varphi_1$ and $\varphi_2$ are appropriately chosen
phases.
\begin{remark}\label{r_motionDescription}
  Observe that
  \begin{itemize}
  \item $\rho^*_0$ corresponds to an homothety;
  \item $\rho_1^*$ corresponds to a translation in the
    direction forming an angle $\varphi_1$ with the polar axis
    $\{\varphi = 0\}$;
  \item $\rho_2^*$ corresponds to a deformation of the circle into an ellipse
    of small eccentricity, whose major axis forms an angle $\varphi_2$ with the polar axis.\\
  \end{itemize}
  This implies that, {infinitesimally} (as $\e\to 0$), 
rationally integrable deformations of a circle are tangent to 
the $5$-parameter family of ellipses.
\end{remark}  

\medskip

{Notice that, in the above strategy, one needs to take $\e\to 0$ as   
$q \to \infty$. This means that we cannot take $\e > 0$ small, but only 
infinitesimal; hence one cannot use directly the above theorem 
to prove the main result. A more elaborate strategy is needed. }

\medskip

\subsection{Our scheme of the proof of Main Theorem  for elliptic billiards} \label{schemeourproof} 
One of the noteworthy contributions of this paper is 
the analysis of perturbations of ellipses of arbitrary eccentricity 
$0 \le e_0 < 1$. Let us outline the main steps involved in the proof. 

\medskip

{Let $\E_{e_0,c}$ be an ellipse of eccentricity $0<e_0<1$ and semifocal 
distance $c>0$, and let $(\mu,\varphi)$ be the associated elliptic coordinates. 
Any domain $\Om$ close  to $\E_{e_0,c}$  can be written (in the elliptic 
coordinates associated to $\E_{e_0,c}$) in the form    
\begin{align*}
  \partial\Om=\{( \mu_0 + \mu_1(\varphi),\varphi):\varphi\in[0,2\pi)\},
\end{align*}
where $\mu_1$ is a smooth $2\pi$-periodic function (see also \eqref{formulaperturbellipse}). Recall that the ellipse $\E_{e_0,c}$ admits all 
integrable rational caustics of rotation number $1/q$ for $q > 2$. \\

By analogy with \cite{ADK} we proceed as follows: 
}\\

\noindent{Step 1} ({\it Dynamical modes}):  
 In Section \ref{sec:preserva-rational}, we consider the one-\hskip0pt{}parameter 
integrable deformation of an ellipse $\E_{e_0,c}$, given by the family of rationally 
integrable domains $\Omega_{\e}$, whose boundaries are given, using the elliptic 
coordinates associated to $\E_{e_0,c}$, by
\begin{align*}
  \partial\Om_\e:= 
  \{( \mu_0 + \e \mu_1(\varphi) + O(\e^2),\varphi):\varphi\in[0,2\pi)\}.
\end{align*}
In Lemma \ref{lemmaintegralszero} we show that if for all $\e$, 
  $\Om_\e$ has an integrable rational caustic
$\Gm^\e_{1/q}$ of rotation number $1/q$, with  $q > 2$,  then
\begin{align}
  \label{eq:vanishing}
  \langle \mu_1,c_q\rangle_{L^2} =0,\quad   \langle \mu_1,s_{q}\rangle_{L^2} = 0,
\end{align}
where $\langle\cdot,\cdot\rangle_{L^2}$  denotes the standard inner product in 
$L^2(\R/2\pi\Z)$ and $\{c_q,s_{q}:q > 2 \}$ are suitable {\it dynamical modes}, 
which can  be explicitly defined using the action-angle coordinates; see \eqref{eq:c_q-function}.
See also Remark \ref{ADK-dot-bnd} for a more quantitative version, that we need 
since we  are interested in perturbations of ellipses and not necessarily deformations.

\medskip

\noindent{Step 2} ({\it Elliptic motions}): 
In Section \ref{sec:elliptic-motions} we consider infinitesimal deformations of ellipses by homotheties, translations, rotations and hyperbolic rotations  (we call them {\it elliptic motions} since they preserve the class of ellipses) and derive their infinitesimal generators $e_h$, $e_{\tau_1}$, $e_{\tau_2}$, $e_{hr}$ and $e_r$; see \eqref{trans1}--\eqref{funchyprot}. Moreover, in  Proposition \ref{changingellipse} we prove a certain approximation result for ellipses.

\medskip

\noindent\textbf{Step 3} ({\it Basis property}):
In Section \ref{sec:adapted-basis} we show that the collection of dynamical modes and elliptic motions form a basis of $L^2(\R/2\pi\Z)$. 
In subsections \ref{sec_analiticitycqsq} and \ref{sec_sing_ell_mot} we will consider their complex extensions and study in details their singularities; this analysis will be important to 
deduce their linear independence (Proposition \ref{linearindep}).
Moreover, in Proposition \ref{prop:basis} we show that they do generate the whole $L^2(\R/2\pi \Z)$, hence they form a (non-orthogonal) basis.

 \medskip

\noindent\textbf{Step 4} ({\it Approximation}):  
In Section \ref{sec:proof-MainThm} we prove an approximation lemma (Lemma \ref{finallemma}) and use it to complete the proof of Main Theorem (see subsection \ref{proofmainthm}), by means of an approximation procedure similar to the one in \cite[Section 8]{ADK}.

 \medskip

%%%%%%%%%%%%%

\section {Preservation of rational caustics}\label{sec:preserva-rational}
In this section we want to investigate perturbations of ellipses, for which the  associated billiard map continues to admit rationally integrable caustics corresponding to some rational rotation numbers.\\

Let us consider an ellipse $\E_{e_0,c}$ and let  $\partial \Omega_\e$ be an {infinitesimal perturbation} of the form
\begin{equation}\label{formulaperturbellipse}
\left\{\begin{array}{l}
x= c \,\cosh (\mu_0 + \e\mu_1(\f) + O(\e^2)) \cos \f \\
y= c \,\sinh (\mu_0 + \e\mu_1(\f) + O(\e^2)) \sin \f
\end{array}
\right.
\end{equation}
{for  $\e\rightarrow 0^+$}.  To simplify notation we  write
$$
\partial \Omega_\e = \E_{e_0,c} + \e\mu_1 + O(\e^2),
$$
which must be understood in the elliptic coordinates with semi-focal distance $c$.

\medskip 

Let us denote  $\mu_\e:=\e\mu_1 +O(\e^2)$ and let $h_\e$ be the generating function 
of the billiard map inside $\Omega_\e$; in particular,
\begin{equation}\label{GeneratingFunction} 
h_\e(\f, \f') = h_{0}(\f,\f') + \e h_1(\f,\f') +{O_{e_0,c,\|\mu_\e\|_{C^1}}(\e^2)},  
\end{equation}
where $h_{0}$ denotes the generating function of the billiard map inside $\E_{e_0,c}$ 
and $O_{e_0,c,\|\mu_\e\|_{C^1}}(\e^2)$ denotes a term bounded 
by $\e^2$ times a factor depending on $e_0,\,c,$ and $\|\mu_\e\|_{C^1}$.
{Notice that this formula makes sense only for infinitesimal perturbations.}
\\

Let us recall the following result {(see \cite[Corollary 9 and Proposition 11]{PR}).}

\begin{proposition}\label{prop:caustic-preserve}
Assume that the billiard map associated to $\partial \Omega_\e$ has a rationally integrable caustic corresponding to rotation number, in lowest term, $p/q \in (0,1/2)$.\\
If we denote by $\{\f_{{p}/{q}}^k\}_{k=0}^q$  the periodic orbit of the billiard map in $\E_{e_0,c}$ with rotation number ${p}/{q}$ and starting at $\f_{p/q}^0=\f$ {\rm(}these orbits are all tangent to a caustic $C_{\l_{{p}/{q}}}$, for some $\l_{p/q} \in (0,b)$, see \eqref{caustic}{\rm)}, then
\be \label{eq:q-caustic-preserve}
L_1(\f):=\sum_{k=0}^{q-1} h_1 (\f_{p/q}^k, \f_{p/q}^{k+1}) = 2 \l_{p/q}
\sum_{k=1}^q \mu_1(\f_{p/q}^k)\equiv \, c_{p/q},
\ee
where $c_{p/q} $ is a constant depending only on $p/q$.\\
\end{proposition}

Let us consider rotation numbers $1/q$, with $q\geq 3$, and denote by $\l_q$ the value of $\l$ corresponding to the caustic of rotation number $1/q$. 
Similarly, $k_{\l_q}$ denotes the associated modulus (see Proposition \ref{prop1}).\\

Therefore, with respect to the action-angle variables \eqref{philambda}, we have that for any $\theta\in \R/2\pi\Z$:
\[
\sum_{k=1}^q \mu_1(\f_{\l_q}(\theta+2\pi k/q))
\equiv\, {\rm constant}.
\]
If $u(x)$ denotes either $\cos x$ and $\sin x$, the above equality implies that
\[
\int_0^{2\pi} \mu_1(\f_{\l_q}(\theta)) \,
u(q \,\theta)\,d\theta=0, 
\]
which, using the expression in (\ref{philambda}) is equivalent to
\[
\int_0^{2\pi} \mu_1\left( \am \left( \frac{4K(k_{\l_q})}{2\pi}\,\theta ; k_{\l_q} \right)\right) \,
u(q \,\theta)\,d\theta=0.\\
\]

\medskip

Consider now the change of coordinates 
$$\f=
\am \left( \frac{4K(k_{\l_q})}{2\pi}\,\theta ; k_{\l_q} \right) \quad
\Longleftrightarrow \quad
\theta = \frac{2\pi}{ 4K(k_{\l_q})} F(\f; k_{\l_q}).
$$
Then
\begin{eqnarray*}
d\theta &=&  \frac{2\pi}{ 4K(k_{\l_q})}  \frac{d}{d\f} \left(  F(\f; k_{\l_q}) \right) =
\frac{2\pi}{ 4K(k_{\l_q})}  \frac{1}{\sqrt{1-k_{\l_q}^2 \sin^2\f}}
\end{eqnarray*}
and the above integral becomes
\begin{equation}\label{conditionpreservationintegral}
\int_0^{2\pi} \mu_1(\f) \ \frac{u \left( \frac{2\pi\, q}{4K(k_{\l_q})} F(\f; k_{\l_q}) \right)}{ 
\sqrt{1-k_{\l_q}^2 \sin^2\f}}\,d\f =0.
\end{equation}

Define  for each $q\geq 3$:
\be \label{eq:c_q-function}
\beal  
c_q(\f) &:=&
\frac{\cos \left( \frac{2\pi\, q}{4K(k_{\l_q})} F(\f; k_{\l_q}) \right)}{ 
\sqrt{1-k_{\l_q}^2 \sin^2\f}}\\
s_q(\f) &:=&
\frac{\sin \left( \frac{2\pi\, q}{4K(k_{\l_q})} F(\f; k_{\l_q}) \right)}{ 
\sqrt{1-k_{\l_q}^2 \sin^2\f}
}
\enal \ee
or equivalently in the complex form:
$$
E_q(\f) := \frac{e^{ 2\pi i\frac{q}{4K(k_{\l_q})} F(\f; k_{\l_q}) }}{\sqrt{1-k_{\l_q}^2 \sin^2\f}}.\\
$$

\medskip

\begin{lemma} \label{lemmaintegralszero}
Assume that the billiard map in $\partial\Omega_\e= \E_{e_0,c} + \e\mu_1 + O(\e^2)$ has rationally integrable caustics 
corresponding to rotation numbers $1/q$ for all $q\geq 3$.  
Then,
$$
\int_0^{2\pi} \mu_1(\f)\,c_q(\f) \,d\f = \int_0^{2\pi} \mu_1(\f)\,s_q(\f) \,d\f = 0 \qquad \forall \, q\geq 3.
$$
Moreover, if we denote $\mu_\e=\e\mu_1 + O(\e^2)$, then:
$$\int_0^{2\pi} \mu_\e(\f) \,c_q(\f) d\f = \int_0^{2\pi} \mu_\e(\f) \,s_q(\f) d\f ={O_{e_0,c,q}}(\e^2),$$  
where 
$O_{e_0,c,q}(\e^2)$ is a term whose absolute value is bounded by $\e^2$ 
times a factor depending on $e_0,\, c,$ and $q$. 
\end{lemma}

\smallskip 

\begin{remark} \label{ADK-dot-bnd}
It follows from  \cite[Lemma 13]{ADK} that {assuming $q<c(e)\|\mu\|^{-1/8}$ we have}
$$
\int_0^{2\pi} \mu_\e(\f) \,c_q(\f) d\f = 
\int_0^{2\pi} \mu_\e(\f) \,s_q(\f) d\f = {O_{e_0,c}}(q^8\|\mu\|_{C^1}^2),$$
where 
$O_{e_0,c}(q^8\|\mu\|_{C^1}^2)$ is a term whose absolute value is bounded by 
$q^8\|\mu\|_{C^1}^2$ times a factor depending on $e_0,c$ {and $C^5$-norm of $\mu$.}\\
In order to apply \cite[Lemma 13]{ADK} we need to translate notations:
in \cite[Section 4, pp. 7--8]{ADK} action-angle variables are introduced and 
in  \cite[middle of page 16]{ADK}  $X_q$ is defined, which 
coincides with what we denote $\f_{\l_q}$ (compare with \eqref{philambda}, where  
$\l_q$ is such that $\om_{\l_q}=1/q$, or with 
Appendix \ref{AppendixTechnical}).  With this notation, the above integral is estimated as in \cite[Lemma 13]{ADK}.
Notice also that the Lazutkin density $\mu$ in \cite[(14) on page 14]{ADK} coincides with our \eqref{dxldf}. Thus, integrating with 
respect to Lazutkin parametrization with Lazutkin density is the same 
as integrating with respect to $\f$.  
\end{remark}

\begin{proof}
The first part follows from \eqref{conditionpreservationintegral}.
As for the second part,  observe that 
\begin{eqnarray*}
\int_0^{2\pi} |c_q(\f)| \,d\f &=& 
\int_0^{2\pi}
\frac{\left| \cos \left( \frac{2\pi\, q}{4K(k_{\l_q})} F(\f; k_{\l_q}) \right)\right|}{ 
\sqrt{1-k_{\l_q}^2 \sin^2\f}} d\f\\
&=& \frac{4K(k_{\l_q})}{2\pi q}
\int_0^{2\pi q} |\cos t|\,dt
= \frac{8K(k_{\l_q})}{\pi}.
\end{eqnarray*}

In particular, recall that $e_0< k_{\l_q} \leq k_{\l_3} <1$ for all $q\geq 3$ and that $k_{\l_q} \longrightarrow e_0$ as $q\rightarrow +\infty$. 
Hence, using the first statement of the proposition:
\begin{eqnarray*}
\int_0^{2\pi} \mu_\e(\f) \,c_q(\f) d\f &=& 
\int_0^{2\pi} O(\e^2) c_q(\f) \,d\f
= O_{e_0,c,q}(\e^2).
\end{eqnarray*}
\end{proof}

%%%%%%%%%%%%%%%%%%%%%%%%%%%%%%%%%%%%%
%%%%%%%%%%%%%%%%%%%%%%%%%%%%%%%%%%%%%

\section{Elliptic Motions} \label{sec:elliptic-motions}

We call translations, rotations, hyperbolic rotations,  and homothety {\it elliptic motions};
indeed, all of these transformations keep  the class of ellipses invariant.  

In  Appendix \ref{App:motions}, we  show that infinitesimal perturbations of an ellipse $\E_{e_0,c}$ by these motions,  correspond  to these functions (expressed in the elliptic coordinate frame with semi-focal distance $c$):\\

\begin{itemize}
\item  Translations:
\begin{eqnarray}
e_{\tau 1}(\f) &:=& \frac{\cos\f}{{1-e_0^2\cos^2 \f}} \label{trans1}\\
e_{\tau 2} (\f) & :=& \frac{\sin\f}{{1-e_0^2\cos^2 \f}};\label{trans2}
\end{eqnarray}
\item Rotations:
\begin{equation} \label{rot}
e_r(\f) \ :=\  \frac{\sin (2\f) } {{1-e_0^2 \cos^2\f}};
\end{equation}
\item Homotheties:
\begin{equation}\label{homot}
e_h(\f) \ := \  \frac{1}{{1-e_0^2\cos^2 \f}};
\end{equation}
\item Hyperbolic rotations:
\begin{equation}\label{funchyprot}
e_{hr}(\f) \ := \ \frac{\cos (2\f)}{{1-e_0^2 \cos^2\f}}.\\
\end{equation}
\end{itemize}

\bigskip

We  say that a strictly convex smooth domain $\Omega$ is a {\it deformation of an ellipse} if there exist  ${\mathcal E}= \E(x_0,y_0,c,\mu_0,\theta)$ and a function 
$$
\mu_1=\mu_1(x_0,y_0,c,\mu_0,\theta) \in C^{\infty}(\R/2\pi\Z)
$$ such that

\begin{eqnarray*}
\partial \Omega= \left\{ 
\left(
\begin{array}{c}
x-x_0\\
y-y_0
\end{array}
\right)\! =\!
\left(
\begin{array}{cc}
\cos \theta  & - \sin\theta\\
\sin \theta & \cos\theta
\end{array}
\right)\!\!
\left(
\begin{array}{c}
c \cosh (\mu_0 +\mu_1(\f)) \cos\f\\
c \sinh (\mu_0 +\mu_1(\f)) \sin\f
\end{array}
\right), \,\f\in [0,2\pi)
\right\}\!.\\
\end{eqnarray*}

By abusing notation, in the following we will  write
\begin{equation}\label{perturbingEllipses}
 \partial \Omega = \E(x_0,y_0,c,\mu_0 + \mu_1,\theta) = \E(x_0,y_0,c,\mu_0,\theta) + \mu_1.\\
\end{equation}

\medskip

We will need  the following approximation result.\\

\begin{proposition}\label{changingellipse}
Let us consider the ellipse $\E_{e_0,c}$ and let 
$$
\mu_1(\f) := a_0 e_{h}(\f) + a_1 e_{\tau 1}(\f) + b_1  e_{\tau 2}(\f) + a_2  e_{hr}(\f) + b_2 e_{r}(\f),
$$
where $a_0, a_1,b_1,a_2, b_2$ are assumed to be sufficiently small.
Then, there exist a constant $C=C(e_0,c)$ and an ellipse $\widetilde{\E} = \E_0 + \mu_{\widetilde{\E}}$ such that 
$$
\| \mu_1 - \mu_{\widetilde{\E}}\|_{C^1} \leq C\|\mu_1\|^2_{C^1}.
$$
\end{proposition}

\medskip  

The proof is presented in Appendix \ref{App:motions}.

\bigskip

\section{An adapted basis for $L^2(\T)$}\label{sec:adapted-basis}

 In this section we want to determine a suitable basis of $L^2(\T)$, where hereafter $\T = \R/2\pi\Z$.
This basis will be constructed by means of elliptic motions $\{e_{h}, e_{\tau 1}, e_{\tau 2}, e_{hr}, e_r \}$, see \eqref{trans1}--\eqref{funchyprot}, and the functions
$\{c_q,s_q\}_{q\geq 3}$ defined in \eqref{eq:c_q-function}. 

In order to prove their linear indepedence, we need to consider their analytic extension to $\C$ and study their singularities.

\subsection{Analyticity properties of $c_q$ and $s_q$} \label{sec_analiticitycqsq}

Let us start by considering the complex extensions of the functions
$\{c_q, s_q\}_{q\geq 3}$ defined in \eqref{eq:c_q-function}:
\be\label{eq:dyn-basis}
\beal
c_q(z) &:=&
\frac{\cos \left( \frac{2\pi\, q}{4K(k_{q})} F(z; k_{q}) \right)}{ 
\sqrt{1-k_{q}^2 \sin^2z}},\\
s_q(z) &:=&
\frac{\sin \left( \frac{2\pi\, q}{4K(k_{q})} F(z; k_{q}) \right)}{ 
\sqrt{1-k_{q}^2 \sin^2z}},
\enal
\ee
where, to simplify the notation, we have denoted  $k_{q}:=k_{\l_q}$. In particular, $k_q$  represents the eccentricity of the caustic $C_q:=C_{\l_q}$ with rotation number $1/q$; moreover,  $k_q \in (e_0, 1)$ for all 
$q\geq 3$ ($e_0$ denotes the eccentricity of the boundary),  it  is strictly decreasing in $q$, and  $k_q \longrightarrow e_0$ as $q\rightarrow +\infty$. Denote 
$\rho_{k_q} = \arcosh \left(\frac{1}{k_q}\right)$ and $\rho_{0} = \arcosh \left(\frac{1}{e_0}\right)$.
\\

We are interested in the complex extensions of these functions and in their singularities.\\

\begin{proposition}\label{prop_anal_cqsq}
For $q\geq 3$, the functions $c_q$ and $s_q$ have an holomorphic  extension to the complex strip {$\Sigma_q =  \left\{ z\in \C: \; \left|{\rm Im}(z)\right| <  \rho_{k_q} \right\}$}.
This extension is maximal in the sense that these functions have singularities at  
$\frac \pi 2 +\pi n \pm i \rho_{k_q}$ (which are ramification singularities). 
\end{proposition}

\medskip 

This proposition will be proven in Appendix \ref{secanaliticitycqsq}.\\

\begin{remark}\label{remstripanalit}
{\rm
Observe that  $\rho_{k_q}$ is a strictly increasing function as a function of $q$ and 
$\rho_{k_3} \leq \rho_{k_q} \longrightarrow \arcosh\left({1}/{e_0}\right) =\rho_0$ as $q\rightarrow +\infty$. \\
Moreover, since $k_3(e_0)$ is a strictly increasing function of $e_0$ and $k_3(e_0)\longrightarrow 1$ as 
$e_0\rightarrow 1^-$, then $\rho_{k_{3}}(e_0)$ is a strictly decreasing function of $e_0$ and 
$\rho_{k_3}(e_0)\longrightarrow 0$ as $e_0\rightarrow 1^-$.\\}
\end{remark}

\subsection{Analyticity properties of $e_{\tau 1}$, $e_{\tau 2}$, $e_r$, $e_h$ and $e_{hr}$} \label{sec_sing_ell_mot}

Now let us discuss  the analyticity properties of the complex extensions of the elliptic motions defined in \eqref{trans1}--\eqref{funchyprot}:
\begin{eqnarray*}
e_h(z) &:=&  \frac{1}{{1-e_0^2\cos^2 z}}\\
e_{\tau 1}(z) &:=& \frac{\cos z}{{1-e_0^2\cos^2 z}} = e_h(z) \, \cos z  \\
e_{\tau 2} (z) & :=& \frac{\sin z}{{1-e_0^2\cos^2 z}}=  e_h(z) \, \sin z\\
e_r(z) &:=&   \frac{\sin (2 z) } {{1-e_0^2 \cos^2z}} = e_h(z) \, \sin (2z)\\
e_{hr}(z) &:=& \frac{\cos (2 z)}{{1-e_0^2 \cos^2z}}= e_h(z) \, \cos (2z).\\
\end{eqnarray*}

The analyticity and the singularities of these functions are the same as those of  $e_{h}(z)$.  More specifically:\\

\begin{proposition}\label{secanaliticityeh}
The function $e_{h}(z)$ is analytic except at the following singular points (which are poles):
$$
\zeta_n =  n\pi  \pm i \, \rho_{0} \qquad {\rm for}\; n\in \Z.
$$
In particular its maximal strip of analyticity is given by
$$\Sigma_{{\rho_0}} =  \left\{ z\in \C: \; \left|{\rm Im}(z)\right| <  \rho_{0} \right\}.\\$$
\end{proposition}

\medskip

We will prove this proposition in Appendix \ref{appsingularitiesellipticmotions}.

\medskip

\begin{remark}{\rm
Since $
\rho_{0} > \rho_{k_q}$ for any $q\geq 3$,
we conclude that
$e_{\tau 1}$, $e_{\tau 2}$, $e_r$, $e_h$, $e_{hr}$  cannot be generated as a finite linear combination of functions $s_q(z)$ and $c_q(z)$ with $q\geq 3$.\\
}
\end{remark}

\subsection{Linear independence}
It follows from the discussion in subsections \ref{sec_analiticitycqsq} and \ref{sec_sing_ell_mot}
that  looking at singularities of these functions, it is possible to deduce the following proposition.
\medskip 
{
\begin{proposition}\label{linearindep}
The functions $e_h$, $ e_{\tau 1}$, $e_{\tau 2}$, $e_{hr}$,  $e_{r}$, $\{s_q\}_{q\geq 3}$ and $\{c_q\}_{q\geq 3}$ are linearly independent, namely, none of them can be written as a finite linear combination of the others.\\
\end{proposition}
}

\begin{proof}
Clearly, $e_{\tau 1}$, $e_{\tau 2}$, $e_r$, $e_h$, $e_{hr}$ are linear independent. 
Looking  at the  singularities of the complex extensions of these functions, it follows  that:
\begin{itemize}
\item $e_h$, $e_{\tau 1}$, $e_{\tau 2}$, $e_{hr}$, $e_r$  cannot be generated as a finite linear combination of $s_q$ and $c_q$ with $q\geq 3$;
\item for any $q_0 \geq 3$, $s_{q_0}$ and $c_{q_0}$ cannot be generated as a finite linear combination of $e_{\tau 1}$, $e_{\tau 2}$, $e_{r}$,  $e_{h}$,  $e_{hr}$, $\{s_q\}_{q\neq q_0}$ and $\{c_q\}_{q\neq q_0}$.
\end{itemize}
\end{proof}

\begin{remark}
A much more subtle and delicate issue is to understand whether these function can be obtained as  infinite combinations of the others. This matter is related to our discussion in subsection \ref{sec:basis-prop} and in Appendix \ref{lastbutnotleast}.\\
\end{remark}

%%%%%%%%%%%%%%%%%%%%%%%%%%%%%%%%%%%%%%
%%%%%%%%%%%%%%%%%%%%%%%%%%%%%%%%%%%%%%
%%%%%%%%%%%%%%%%%%%%%%%%%%%%%%%%%%%%%%%

\subsection{Weighted $L^2(\T)$ space}

Let us denote by $\|\cdot\|_{L^2_{e_0}}$ the $L^2$-norm induced by the inner 
product with weight $w_{e_0}(\f):= (1-e_0^2 \cos^2\f)$, {\it i.e.},
$$
\langle f,g\rangle_{L^2_{e_0}} := \langle {w_{e_0}}\, f, {w_{e_0}}\, g \rangle_{L^2}.
$$ 
For $0\leq e_0 <1$, this norm is clearly equivalent to the usual $L^2$-norm; in fact 
for each $f\in L^2(\T)$ we have
$$
(1-e_0^2) \|f\|_{L^2}   \leq \|f\|_{L^2_{e_0}} \leq \|f\|_{L^2}.
$$

We denote by $L^2_{e_0}(\T)$  the space $L^2(\T)$ equipped with $\|\cdot\|_{L^2_{e_0}}$.\\

Clearly, with the choice of this weighted norm, the functions $e_h, e_{\tau1},e_{\tau 2}, e_{r}, e_{hr}$ are  mutually orthogonal in $L^2_{e_0}$  (observe in fact that 
when multiplied by the weight, they become $\cos (k\f)$ for some $k=0,1,2$ or $\sin (k\f)$ for some $k=1,2$).

In particular:
\begin{eqnarray*}
\|e_h\|_{L^2_{e_0}} &=& \sqrt{2\pi} \\
\|e_{\tau 1}\|_{L^2_{e_0}} &=& 
\|e_{\tau 2}\|_{L^2_{e_0}} \; =\;
\|e_r\|_{L^2_{e_0}} \;=\;
\|e_{hr}\|_{L^2_{e_0}}\;=\;
\sqrt{\pi}.
\end{eqnarray*}

On the other hand, for $q\geq 3$:
\begin{eqnarray*}
\|c_q\|^2_{L^2_{e_0}} &=& 
\int_0^{2\pi}
\frac{\cos^2 \left( \frac{2\pi\, q}{4K(k_{\l_q})} F(\f; k_{\l_q}) \right)}{ 
{1-k_{\l_q}^2 \sin^2\f}} (1-e_0^2\cos^2\f)^2 d\f\\
&=& 4K(k_{\l_q})
\int_0^1 \cos^2 \left( 2\pi\, q\, \xi  \right)   \frac{(1-e_0^2\cos^2\f(\xi))^2}{\sqrt{1-k_{\l_q}^2 \sin^2\f(\xi)}}
\,d\xi\\
&=& \frac{2 K(k_{\l_q})}{q\pi}
\int_0^{2\pi q} \cos^2  t  \;   \frac{(1-e_0^2\cos^2\f(\xi(t)))^2}{\sqrt{1-k_{\l_q}^2 \sin^2\f(\xi(t))}}\,dt.\\
\end{eqnarray*}
Hence:
$$
(1-e_0^2)^2 {2 K(k_{\l_q})}
 \leq \|c_q\|^2_{L^2_{e_0}} \leq  \frac{2 K(k_{\l_q})}{\sqrt{1-k_{\l_q}^2}}.
$$
In particular, using that 
$K(\cdot)$ is an increasing function and $k_{\l_q}$ is decreasing with respect to $q$, 
we can obtain uniform bounds:
$$
(1-e_0^2)^2 {2 K(e_0)}
 \leq \|c_q\|^2_{L^2_{e_0}} \leq  \frac{2 K(k_{\l_3})}{\sqrt{1-k_{\l_3}^2}} \qquad \forall q\geq 3.
$$

Similarly, for the functions $s_q$.\\

In order to simplify our notation, hereafter we will denote

\begin{equation}\label{changenotation1}
{{\bbe}}_0 := \frac{e_{h}}{\sqrt{2\pi}}, \quad \bbe_1 := \frac{e_{\tau 2}}{\sqrt{\pi}}, \quad \bbe_2 := \frac{e_{\tau 1}}{\sqrt{\pi}}, \quad \bbe_3 := \frac{e_{r}}{\sqrt{\pi}} \quad  \bbe_4 := \frac{e_{hr}}{\sqrt{\pi}}
\end{equation}
and 
\begin{equation}\label{changenotation2}
\bbe_{2k} := \dfrac{c_k}{\|c_k\|_{L^2_{e_0}}},  \quad \bbe_{2k-1}: = \frac{s_k}{\|s_k\|_{L^2_{e_0}}} \qquad \forall\;k\geq 3.\\
\end{equation}

\bigskip

The family $\{\bbe_k\}_{k=0}^{+\infty}$ consists of linearly independent normal vectors in $L^2_{e_0}$. We want to show that they are  a basis.

\medskip

\subsection{Basis property}\label{sec:basis-prop}

In this subsection we want to prove that $\{\bbe_k\}_{k\geq 0}$ form a basis of $L^2_{e_0}(\T)$, or equivalently of $L^2(\T)$. 
{We need to show that they form a complete set of generators.} \\

Let us start with the following proposition.
\begin{proposition}\label{newpropindep}
Let $q_0\geq 3$; then
$$
 \langle \left\{\bbe_k\right\}_{0\leq k \leq  2q_0}\rangle 
 \cap 
\overline{ \langle \left\{\bbe_k\right\}_{k> 2q_0}
\rangle}  = \{0\}.\\
$$
\end{proposition}

\medskip

\noindent {The proof of this proposition is postponed to Appendix \ref{lastbutnotleast}.
}

\medskip

Let us now introduce the linear map $\mathcal L_{q_0}:L^2(\T) \to L^2(\T)$ defined by mapping {the standard 
Fourier basis into the following functions:}
{\begin{eqnarray}\label{deflinearmap}
\frac{1}{\sqrt{2\pi}} &\longmapsto& \frac{1}{\sqrt{2\pi}}\nonumber\\
\frac{1}{\sqrt{\pi}}\cos (q\f)  &\longmapsto& \frac{1}{\sqrt{\pi}}\cos (q\f) \qquad \mbox{for}\; 0<q\leq q_0\nonumber\\
\frac{1}{\sqrt{\pi}}\sin (q\f)  &\longmapsto& \frac{1}{\sqrt{\pi}}\sin (q\f) \qquad \mbox{for}\; 0<q\leq q_0\\
\frac{1}{\sqrt{\pi}}\cos (q\f)  &\longmapsto&c_q(\varphi) \qquad \qquad \;\;\; \mbox{for}\; q> q_0\nonumber\\
\frac{1}{\sqrt{\pi}}\sin (q\f)  &\longmapsto&s_q(\varphi) \qquad \qquad \;\;\; \mbox{for}\; q> q_0\nonumber .
\end{eqnarray}
}

\begin{lemma}\label{invertibilityimpliesbasis}
Suppose there is $q_0\geq 3$ such that the linear map $\mathcal L_{q_0}$ 
is invertible. Then, $\{\bbe_{k}\}_{k\geq0}$ is a basis of $L^2(\T)$.
\end{lemma}

\begin{proof}
Since the corresponding linear map is invertible, then the collection 
$$\left\{\frac{1}{\sqrt{2\pi}}, \frac{1}{\sqrt{\pi}}
\cos (q\f), \frac{1}{\sqrt{\pi}}\sin (q\f) \right\}_{0<q\le q_0} \ \ \ \ \!\!\!\!\!\! \cup\;\; \{c_q(\f),s_q(\f)\}_{q>q_0}$$
also forms a basis and, in particular, it spans the whole space $L^2(\T)$. 
This implies that the subspace $$ \overline{\langle \left\{c_q,s_q\right\}_{q> q_0}\rangle} = 
\overline{\langle \left\{\bbe_k\right\}_{k> 2q_0}\rangle}$$ 
has codimension $2q_0+1$. \medskip

{It follows from Proposition  \ref{linearindep} that  the linear subspace spanned by 
$\left\{ \bbe_k\right\}_{0\leq k\leq 2q_0}$ has dimension $2q_0+1$ and from 
Proposition \ref{newpropindep}  that}
$$
\overline{\langle \left\{\bbe_k\right\}_{k> 2q_0}}
\rangle \cap \langle \left\{\bbe_k\right\}_{0\leq k \leq  2q_0}\rangle = \{0\}.
$$
We can conclude from this that
$$ \overline{\langle \left\{\bbe_k\right\}_{k> 0}\rangle } = L^2(\T).$$ 
Hence,  $\left\{\bbe_k\right\}_{k\geq 0}$ form a set of 
generators of $L^2(\T)$ and therefore  a basis.
\end{proof}

\medskip

The problem now reduces to show that the linear map $\mathcal L_{q_0}$, defined by \eqref{deflinearmap}, is invertible for some $q_0\geq 3$. \\

For $q\in \Z_+$ and $j\geq 3$, let us consider the elements of the (infinite) correlation matrix $\widetilde{A}=(\ta_{i,h})_{i,h=0}^{\infty}$,  whose entries are  
\begin{eqnarray}\label{matr:corr-coeff}
\beal 
\ta_{2q,2j} &:=&   \langle \cos (q \f), c_{j} \rangle_{L^2}  \\
\ta_{2q,2j+1} &:=&   \langle \cos (q \f), s_{j} \rangle_{L^2}  \\  
\ta_{2q+1,2j} &:=&   \langle \sin (q \f), c_{j} \rangle_{L^2} \\ 
\ta_{2q+1,2j+1} &:=&   \langle \sin (q \f), s_{j} \rangle_{L^2}.  \\
\enal 
\end{eqnarray}

\begin{lemma} \label{lemmaestimateaqj} 
There exists $\rho=\rho(e_0,c)>0$ such that for all $q\in\N$ and $j\geq 6$:
\begin{eqnarray*}
\ta_{q,j}  = {2} K(k_{[j/2]}) \,\delta_{q,j} + O_{e_0,c}\left(  j^{-1} \,
e^{-\rho\, {|q-j|}} \right),
\end{eqnarray*}
where $[\cdot]$ denotes the integer part, $\delta_{q,j}$ the Dirac's delta, and 
$O_{e_0,c}(*)$ means that  the absolute value of the corresponding term is bounded by $*$ 
times a factor depending only on $e_0$ and $c$.
\end{lemma}

The proof of the above lemma will be given in Appendix \ref{AppendixTechnical}.

\begin{proposition}\label{prop:basis}
There exists $q_0=q_0(e_0,c)\geq3$ such that $\mathcal L_{q_0}$ is invertible as an operator acting on $L^2(\T)$. In particular, it follows from Lemma \ref{invertibilityimpliesbasis} that
$\{\bbe_k\}_{k\geq 0}$ is a basis of $L^2(\T)$.
\end{proposition}

\begin{proof} 
Let us  show that  there exists $q_0=q_0(e_0,c)\geq3$ such that the linear map $\mathcal L_{q_0}$ is invertible. 
Consider the infinite dimensional matrix ${A}=(a_{q,j})_{q,j}$ defined as:
{
\begin{equation}\label{atilde}
a_{q,j} = \left\{\begin{array}{rl}
\delta_{q,j} & \mbox{if $j<2q_0$, $q\geq 0$}\\
\frac{1}{\sqrt{2\pi}} \ta_{0,j} & \mbox{if $j\geq 2q_0$} \\
\frac{1}{\sqrt{\pi}} \ta_{q,j} & \mbox{if $j\geq 2q_0$, $q\geq 1$}. \\
\end{array}
\right.\\
\end{equation}
}

\medskip

{Using Lemma \ref{lemmaestimateaqj} and the fact that $K(k_j) \geq K(e_0)>0$ 
for all $j\geq 3$, we obtain  
$$
|a_{q,q}| \geq \min \left\{1, \dfrac{2}{\sqrt \pi}\  K(e_0) + 
O_{e_0,c}\left(  \frac{1}{q\, e^{\rho}} \right)\right\}.
$$
Observe that since $K(e_0)\geq \frac{\pi}{2}$ for $0\leq e_0<1$, then if 
one chooses $q_0$ sufficiently large, then the above minimum is achieved by $1$.}

{
Denote   by 
$\mathcal D_{q_0}$ the diagonal linear operator given by the diagonal elements of $\cL_{q_0}$.}
Notice that $\mathcal D_{q_0}$ is invertible and has bounded norm of the inverse; 
{in particular, for $q_0$ sufficiently large, $\|\mathcal D_{q_0}^{-1}\|_2\leq 1$. 
Again using Lemma \ref{lemmaestimateaqj}} we also have for each $q\geq 0$: 
\[
\sum_{q\ge q_0}\sum_{j=0, j\neq q }^\infty \left| a_{q,j}\right|^2 \le \frac{C}{q_0},
\]
for some suitable constant $C=C(e_0,c)$.  For any predetermined $\delta=\delta(e_0,c)>0$  
by choosing $q_0$ large enough we obtain
\begin{equation}\label{sumaj}
\sum_{q\ge q_0}\sum_{j=0, j\neq q }^\infty \left|{a_{q,j}}  \right|^2 < \delta(e_0,c).
\end{equation}
Using Cauchy-Schwarz, \eqref{sumaj} implies that with respect to the $L^2$-norm $\|\cdot\|_2$ we have 
\[
\|\mathcal L_{q_0}- {\mathcal D_{q_0}}\|_2 \le 
\delta(e_0,c) \le \frac{1}{2} \leq \frac 12 \|\mathcal D_{q_0}^{-1}\|^{-1}_2.
\]
This implies that $\mathcal L_{q_0}$ is invertible and concludes the proof. 
\end{proof}

\medskip

\section{Proof of the Main Theorem}\label{sec:proof-MainThm}

In this section we prove our Main Theorem. Let us first start by stating 
and proving the following approximation lemma similar to 
 \cite[Lemma 24]{ADK}.\\

\begin{lemma}[{\bf Approximation Lemma}] \label{finallemma}  
Let us consider the ellipse $\E_{e_0,c}$ and let $\partial\Omega$ be a rationally integrable $C^{39}$-deformation 
of $\E_{e_0,c}$, identified by  
a $C^{39}$ function $\mu$, {\it i.e.}, $\partial \Omega=\E_{e_0,c}+\mu$. {For every $L>0$, there exists 
a constant $C=C(e_0,c,L)$ such that if $\|\mu\|_{C^{39}}\leq L$, then the following holds.}
There exist an ellipse $\overline{\E}=\overline{\E}(\bar{x}_0,\bar{y}_0, \bar{c},\bar{\mu}_0,\bar{\theta})$ and a function $\overline{\mu}=\overline{\mu}({\overline{\f}})$ {\rm(}where $\overline{\f}$ is the angle with respect to the elliptic coordinate frame associated to $\overline{\E}$ {\rm)}, such that  
$\partial \Omega=\overline{\E} + \overline{\mu}$ {\rm(}see \eqref{perturbingEllipses}{\rm)} and
$$
\|\overline{\mu}\|_{C^1} \leq C(e_0,c,L) \|\mu\|^{703/702}_{C^1}.
$$
\end{lemma}

\medskip

\begin{proof}
Let us consider the basis ${\mathcal B}_{e_0}:=\{\bbe_j\}_{j\geq 0}$ of $L^2_{e_0}(\T)$, introduced in \eqref{changenotation1} and \eqref{changenotation2}; moreover, we
denote by $$\V_{e_0}:= \langle \bbe_0, \bbe_{1}, \bbe_{2}, \bbe_{3},  \bbe_4 \rangle$$ the $5$-dimensional space generated by elliptic motions.
Let us decompose
$$
\mu = \mu_{\V_{e_0}} + \mu^{\perp}, 
$$
where $ \mu^{\perp}$ is orthogonal in $L^2_{e_0}(\T)$ to the subspace $\V_{e_0}$ and 
$
\mu_{\V_{e_0}} := \sum_{j=0}^4 a_j\bbe_j \in \V_{e_0}.
$
Using the orthogonality in $L^2_{e_0}(\T)$ and the fact that ${\mathcal B}_{e_0}$ is a basis, we obtain
$$
\| \mu_{\V_{e_0}}\|_{L^2_{e_0}}^2 + \| \mu^{\perp}\|^2_{L^2_{e_0}} = \|\mu\|_{L^2_{e_0}}^2 \leq C\| \mu\|_{C^1}^2,
$$
for some $C=C(e_0,c)$. This implies that $a_j =O_{e_0,c}(\|\mu\|_{C^1})$ for $0\leq j \leq 4$; since the 
functions $\bbe_j$'s are analytic, we obtain
\begin{equation}\label{eq:estimateprojectedmu}
\| \mu_{\V_{e_0}}\|_{C^{k}} \leq C(e_0,c,k) \|\mu\|_{C^1}.
\end{equation}
We claim that 
\begin{equation}\label{lastestimate}
\|\mu^{\perp}\|_{C^1} \leq C(e_0, c, \|\mu\|_{C^k}) \|\mu\|_{C^1}^{1+\d},
\end{equation}
where the above constant depends monotonically on $\|\mu\|_{C^k}$, 
and {$\d$ will turn out to be equal to $1/702$.}
This is enough to complete the proof. In fact, applying Proposition \ref{changingellipse} with $\E_{e_0,c}$ and $\mu_{\V_{e_0}}$, we obtain
an ellipse $\widetilde{\E} = \E_{e_0,c} +\mu_{\widetilde{\E}}$ such that
$$
\|\mu_{\V_{e_0}} - \mu_{\widetilde{\E}}\|_{C^1} \leq C \|\mu_{\V_{e_0}}\|^2_{C^1} \leq C\|\mu\|^2_{C^1},
$$
where the last inequality follows from \eqref{eq:estimateprojectedmu}.
We choose $\overline{\E}:=\widetilde{\E}$; if we consider $\partial \Omega=\overline{\E}+ \overline{\mu}$, then we conclude from Lemma \ref{lemma13} that 
\begin{eqnarray*}
\|\overline{\mu}\|_{C^1} &\leq &C(e_0,c) \|\mu -\mu_{\widetilde{\E}}\|_{C^1} = C(e_0,c) \|\mu_{\,\V_{e_0}} + \mu^{\perp} -\mu_{\widetilde{\E}}\|_{C^1}\\
&\leq& C(e_0,c) \left(\|\mu_{\,\V_{e_0}} - \mu_{\widetilde{\E}} \| + \|\mu^{\perp}\|_{C^1}\right)\\
&\leq& C(e_0, c, \|\mu\|_{C^k})\, \|\mu\|_{C^1}^{1+\d}.
\end{eqnarray*}

\medskip 

Therefore, let us prove \eqref{lastestimate}.
Let us define the Fourier coefficients 
$$
\widehat{\mu}^\perp_{j} := \langle  \mu^{\perp }, \bbe_j \rangle_{e_0} ,
$$
which are clearly zero for $j=0,\ldots, 4$ (due to orthogonality). In particular we have (see for example \cite[Corollary 23]{ADK})
$$\|\mu^{\perp}\|^2_{L^2_{e_0}} \leq C(e_0,c) \sum_{j=5}^{\infty} \left| \widehat {\mu}^\perp_{j} \right|^2.$$
It follows from Lemma \ref{lemmaintegralszero} and
 Remark \ref{ADK-dot-bnd} that  
$$
\left| \widehat {\mu}^\perp_{j} \right| = O_{e_0,c}(q^8\|\mu\|^2_{C^1}).
$$
Fix some positive $\a<1/8$. Choose $q_0=[\|\mu\|^{-\a}_{C^1}]$, where $[\cdot]$ denotes the integer part and $\a>0$ will be determined in the following.
Below $C(e_0,c)$ denotes a constant depending on $e_0$ and $c$. 
Using the above estimates we get  for $5\leq q \leq q_0$:
$$
\left| \widehat {\mu}^\perp_{j} \right| \le  C(e_0,c)q^8\|\mu\|^2_{C^1} \le 
C(e_0,c)\|\mu\|^{2-8\a}_{C^1}.
$$
Then, summing over $5 \le q \le q_0$, we obtain
\[
\sum_{q=5}^{q_0}\ \left| \widehat {\mu}^\perp_{j} \right|^2 \le 
C(e_0,c)\|\mu\|^{4-17\a}_{C^1}.
\]
On the other hand, Lemma \ref{decayfourier} gives  
\[
\left|{\mu}^\perp_{j} \right|^2 \le C(e_0,c) \frac{\|\mu\|^2_{C^1}}{q^2}.
\]
Therefore, summing over $q > q_0$ we conclude that
\[
\sum_{q=q_0+1}^{+\infty}\ \left| \widehat {\mu}^\perp_{j} \right|^2 \le 
C(e_0,c)\|\mu\|^{2+\a}_{C^1}.
\]
Combining the two above estimates and optimizing for $\a$ 
(\ie choosing $\a =1/9$), we conclude that 
$\left| \widehat {\mu}^\perp_{j} \right|\le C(e_0,c)\|\mu\|^{19/18}_{C^1}.$

Now, observe that
$$
\|\mu^{\perp}\|_{C^1} \leq \|D \mu^{\perp}\|_{L^1} + \|D^2 \mu^{\perp}\|_{L^1} \leq \|D \mu^{\perp}\|_{L^2} + \|D^2 \mu^{\perp}\|_{L^2}. 
$$
Using standard Sobolev interpolation inequalities (see for example \cite{GT}): for any $\delta>0$ and any $1\leq j \leq 2$ we have:
$$
\|D^j \mu^{\perp}\|_{L^2} \leq C\left(
\Delta \|\mu^{\perp}\|_{C^k} + \Delta^{-j/(k-j)} \|\mu^{\perp}\|_{L^2}
\right).
$$
Optimizing the above estimate,  we choose $\Delta =\|\mu\|_{C^1}^{703/702}$.

Using the above estimates and the fact that $\|\cdot\|_{L^2_{e_0}}$ is equivalent to $\|\cdot\|_{L^2}$, we conclude that
\eqref{lastestimate} holds, {by taking $\delta=\frac{1}{702}$}.
\end{proof}

\medskip

\subsection{Proof of the Main Theorem} \label{proofmainthm}
First of all, observe that up to applying a rotation and a translation (that do not alter rational integrability, nor the other hypotheses), we can assume that  $\E_0 = \E_{e_0,c}$.

{Let us denote by $\Ell_{\sigma}(\E_{0})$ the set of ellipses whose {Hausdorff distance}
from $\E_0$ is not larger than $\sigma$:
$$
\Ell_{\sigma}(\E_{0}) = \left\{
\E' \subset \R^2: \: {\rm dist}_{H}(\E',\E_0) \leq {\sigma}
\right\},
$$
where $\sigma$ is  sufficiently small (to be determined).}

{Let us denote by ${\mathcal P}_{\sigma}(\E_0)$ 
the set of parameters corresponding to ellipses in 
$\Ell_{\sigma}(\E_{0})$:
$$
{\mathcal P}_{\sigma}(\E_0) := \left\{
(x,y,c,\mu,\theta)  \in \R^2 \times (0,+\infty)^2\times [0,\pi): \; 
\E(x,y,c,\mu,\theta) \in \Ell_{\sigma}(\E_0)
\right\}.
$$
Then, ${\mathcal P}_{\sigma}(\E_0)$ is compact in 
$\R^2 \times (0,+\infty)^2\times [0,\pi)$. Notice that 
the size of this set is independent of $\e$.}\\

Let $\mu$ be a $C^k$ perturbation, with $\|\mu\|_{C^k} < K$ and $\|\mu\|_{C^1}<\e$, and consider the domain given by
$$
\partial \Omega = \E_0 +\mu.
$$

Observe that there exists a constant $M=M(e_0,c,K)$   
such that if {$\E \in \Ell_{\sigma}(\E_{0})$} and 
$\partial\Omega= \E +\widetilde{\mu}$, then
\begin{equation}\label{defEMME}
{\rm dist}_{H}(\E,\partial \Omega) \leq M \|\widetilde{\mu}\|_{C^0}.\\
\end{equation}

For any {$\nu\in {\mathcal P}_\sigma(\E_0)$}, let us 
denote by $\E_\nu$ the corresponding ellipse and by 
$\mu_{\nu}$ the perturbation such that 
$\partial \Omega = \E_\nu +\mu_\nu$. 
Observe that the elliptic coordinate frame corresponding to 
$\E_\nu$ varies analytically with respect to $\nu$; hence, 
$\mu_{\nu}$ also
changes analytically with respect to $\nu$.
In particular, we can assume $\e$  sufficiently small so that 
for any {$\nu\in {\mathcal P}_\sigma(\E_0)$} we have 
$\|\mu_{\nu}\|_{C^k} < 2K$.  \\
The function $\nu \longmapsto \|\mu_\nu\|_{C^1}$ is, therefore, continuous and, 
being {${\mathcal P}_\sigma(\E_0)$} compact, it achieves a minimum at some 
$\nu^*\in {\mathcal P}_\sigma(\E_0)$. 
$$
0\leq \|\mu_{\nu^*}\|_{C^1} \leq \|\mu\|_{C^1} < \e. 
$$

Let us assume that  $\|\mu_{\nu^*}\|_{C^1}\neq 0$ and  apply Lemma \ref{finallemma} to $\E_{\nu^*}$ and $\mu_{\nu^*}$, thus obtaining $\overline{\E}_{\nu^*}$ and $\overline{\mu}_{\nu^*}$, such that
\begin{equation}\label{pre-contradiction}
\|\overline{\mu}_{\nu^*}\|_{C^1} \leq C \|\mu_{\nu^*}\|^{1+\delta}_{C^1} <  \|\mu_{\nu^*}\|_{C^1}
\end{equation}
where we have assumed $\e$ to be sufficiently small. {Notice that as $\|\mu_{\nu^*}\|_{C^1}$
decreases, $\|\overline{\mu}_{\nu^*}\|_{C^1}$ decreases. Therefore, $\e$ is small enough,
$\bar \E$  from Lemma \ref{finallemma} belongs to the set ${\mathcal P}_\sigma(\E_0)$,
which has non-emtpy interior and is independent of $\e$. }

Using the triangle inequality, for sufficiently small  $\e$, we have: 
\begin{eqnarray*}
{\rm dist}_{H}(\E_0,\overline{\E}_{\nu^*}) &\leq&
{\rm dist}_{H}(\E_0, \partial \Omega) +
{\rm dist}_{H}(\partial  \Omega, \overline{\E}_{\nu^*})\\
&\leq& 2M\e \leq \sigma.
\end{eqnarray*}
Hence, $\overline{\E}_{\nu^*} \in \Ell_{\sigma}(\E_{0})$ 
and therefore $\overline{\E}_{\nu^*}  = \E_{\overline{\nu}^*}$ 
for some $\overline{\nu}^* \in {\mathcal P}_\sigma(\E_0)$.
This and \eqref{pre-contradiction} contradict the minimality 
of {${\nu^*}$ in ${\mathcal P}_\sigma(\E_0)$.} 
As a consequence $\mu_{\nu^*}\equiv 0$ and therefore 
{$\partial\Omega \in {\mathcal P}_\sigma(\E_0)$.}  
\qed

\vspace{7 pt}

%%%%%%%%%%%%%%%%%%%%%%%%%%%%%%%%%%%%%%%%%%

%%%%%%%%%%%%%%%%%%%%%%%%%%%%%%%%%%%%%%%

\appendix 

\section{Parametrizing ellipses}
\label{App:paramellipses}

Let us consider the ellipse 
$$
{\E_{e_0,c}} =\left \{(x,y)\in \R^2:\; \frac{x^2}{a^2} + \frac{y^2}{b^2} = 1\right\}
$$
centered at the origin and with semiaxes of lengths, respectively, $0<b\leq a$; 
in particular, as before, $e_0=\sqrt{1-\frac{b^2}{a^2}} \in [0,1)$ denotes its eccentricity, while $c=\sqrt{a^2-b^2}$ the semi-focal distance.

\medskip 

We want to recall various parametrizations of ellipses that have been mentioned and used in the proofs.

\medskip 

\begin{itemize}

\item {\bf Polar coordinates:}  
 $(r,\f) \in (0,+\infty) \times \R/2\pi\Z$:
$$
\E_{e_0,c}: \quad
\left\{
\begin{array}{l}
x= a  \cos\f \\
y= b  \sin\f.\\
\end{array}
\right.
$$
Observe that this choice parametrizes the ellipse counterclockwise, with $(x(0),y(0))=(a,0).$ 
\\
In these coordinates the radius of curvature of the ellipse is given by
\begin{eqnarray} \label{radiuscurvature}
\rho(\f) &=& \left|\frac{(\dot{x}^2+\dot{y}^2)^{\frac{3}{2}}}{\dot{x}\ddot{y} - \dot{y}\ddot{x}} \right| = 
\frac{(a^2 \sin^2\f + b^2 \cos^2\f)^{\frac{3}{2}}}{ab}\nonumber\\
&=& \frac{a^2}{b} (1-e_0^2 \sin^2 \f)^{\frac{3}{2}}.
\end{eqnarray}

\medskip

\item {\bf Arc-length parametrization:}
$$
\E_{e_0,c}: \; \left\{
\begin{array}{l}
x = x(s)\\
y = y(s)
\end{array}
\right. \quad {\rm for}\; s\in [0, |\E_{e_0,c}|),
$$
where $|\E_{e_0,c}|$ denotes the perimeter of $\E_{e_0,c}$ and we fix, for example, the starting point at $(x(0),y(0))=(a,0)$ and the counterclockwise  orientation. 
In terms of the polar coordinate $\f$ we have:
\begin{eqnarray} \label{arclength}
s(\f) = a \int_0^{\f} \sqrt{1-e_0^2 \sin^2\f}\,d\f,
\end{eqnarray}
from which
\begin{equation}\label{dsdf}
\frac{ds(\f)}{d\f} = a \sqrt{1-e_0^2 \sin^2\f}.
\end{equation}
In particular, the perimeter of $\E_{e_0,c}$ can be computed quite explicitly:
\begin{eqnarray*}
|\E_{e_0,c}| &=& \int_0^{2\pi} \sqrt{a^2 \cos^2\f+ b^2 \sin^2\f}\,d\f =  
a \int_0^{2\pi} \sqrt{1-e_0^2 \sin^2\f}\,d\f \\
&=& 4a \int_0^{\frac{\pi}{2}} \sqrt{1-e_0^2 \sin^2\f}\,d\f =: 4a\, E(e_0),
\end{eqnarray*}
where $E(e_0):= \int_0^{\frac{\pi}{2}} \sqrt{1-e_0^2\sin^2 \f}\,d\f$
is called {\it complete elliptic integral of the second type} (see for instance \cite{Akhiezer}).

\medskip

\item {\bf Elliptic polar coordinates}  See subsection \ref{secellipticpolarcoords}.\\

\item {\bf Lazutkin parametrization:}
Following an idea by Lazutkin in \cite{Lazutkin}, let us introduce the following reparametrization
\begin{equation} \label{Lazutkin}
x_\ell(s) :=    C_\ell^{-1}\, \int_0^s \rho^{-\frac{2}{3}}(\tau) d\tau,
\end{equation}
where $s$ denotes the arc-length parameter, $\rho$ the radius of curvature computed in (\ref{radiuscurvature}) and 
$
C_\ell := \int_0^{|\E_{e_0,c}|} \rho^{-\frac{2}{3}}(\tau) d\tau  
$
is a normalizing factor so that $x_\ell(|\E_{e_0,c}|) =1$ (sometimes it is called the {\it Lazutkin perimeter}).\\

Observe that, using (\ref{radiuscurvature}), (\ref{dsdf}),  and (\ref{Lazutkin}), we obtain $x_\ell$ as a function of the polar angular coordinate  $\f$:
\begin{eqnarray*} 
x_\ell(\f) &= & C_\ell^{-1} \int_0^{\f} \rho^{-\frac{2}{3}}(s(\f)) \frac{ds(\f)}{d\f}\,d\f \nonumber\\
&=& C_\ell^{-1}\, \frac{b^{\frac{2}{3}}}{a^{\frac{1}{3}}}  \int_0^{\f} 
 \frac{d\f}{\sqrt{1-e_0^2 \sin^2 \f}}.
\end{eqnarray*}
In particular,
\begin{equation}\label{dxldf}
\frac{dx_\ell(\f)}{d\f} = \frac{b^{\frac{2}{3}}}{a^{\frac{1}{3}}} 
 \frac{C_\ell^{-1}}{\sqrt{1-e_0^2 \sin^2 \f}}.
\end{equation}

\medskip

\begin{remark}
For any smooth strictly convex domain $\Omega$, let us denote by $|\partial \Omega|$ the perimeter of $\Omega$. Let us consider  the {\it Lazutkin change of coordinates}
$L_\Omega:  [0,|\partial \Omega|) \times [0,\pi]  \longrightarrow  \R/\Z \times [0,\delta]$:
\begin{eqnarray*}
 (s,\f)   &\longmapsto&  \left(x=C^{-1}_\Omega \int_0^s \rho^{-{2/3}}(s)ds, \;
y=4C_\Omega^{-1}\rho^{{1/3}}(s)\ \sin \f/2 \right),
\end{eqnarray*}
where $C_\Omega := \int_0^{|\partial \Omega|} \rho^{-{2/3}} (s)ds$ and   $\delta>0$ is sufficiently small.

In these new coordinates the billiard map becomes very simple (see \cite{Lazutkin}):

\be \label{lazutkin-billiard-map}
f_{L_{\Omega}}(x,y) = \Big( x+y +O(y^3),y + O(y^4) \Big)
\ee

In particular, near the boundary $\{\f=0\} = \{y=0\}$, the billiard map $f_{L_\Omega}$ reduces to
a small perturbation of the integrable map $(x,y)\longmapsto (x+y,y)$.
Using this result and KAM theorem, Lazutkin proved in \cite{Lazutkin}
that if $\partial \Omega$
is sufficiently smooth (smoothness is determined by KAM theorem),
then there exists a positive measure set of caustics (which correspond to KAM invariant curves), which accumulates on
the boundary and on which the motion is smoothly conjugate to a rigid rotation with irrational rotation number.
\end{remark}
\end{itemize}

\smallskip

\section{Elliptic Motions  and a proof of Proposition
\ref{changingellipse} }
\label{App:motions}

We start by studying perturbations of ellipses within the family 
of ellipses. Once enough analytic tools are developed we 
prove Proposition \ref{changingellipse}.  Up to suitable 
translation and rotation, we can assume -- using the parametrization introduced in \eqref{paramellipse} --, that the unperturbed 
ellipse has the form $\E_{e_0,c} = \E(0,0,c, \mu_0, 0)$; in particular, 
its eccentricity is $e_0=1/\cosh \mu_0$.\\

\subsection*{Perturbing by an homothety} Let $\l\in \R$ and consider an homothety of factor $e^\l$. We want to write the dilated/contracted ellipse $\E_{\l} := e^\l \E_{e_0,c}$ as
$$
\E_\l = \E_{e_0,c} + \mu_\l,
$$
which is equivalent to
$$
\E(0,0,e^\l c,\mu_0, 0) = \E(0,0,c,\mu_0 + \mu_\l,0).
$$
Hence, we have to solve the following system of equations: \\
$$
\left\{
\begin{array}{l}
c \cosh(\mu_0 + \mu_\l(\f)) \cos \f = e^\l c \cosh \mu_0 \cos \f_\l \\
c \sinh(\mu_0 + \mu_\l(\f)) \sin \f = e^\l c \sinh \mu_0 \sin \f_\l
\end{array}
\right.
$$
where one should observe that the angle $\f$ changes as well. In particular, $\mu_\l=o(1)$ and $\Delta \f:= \f_\l - \f = o(1)$.
Applying Taylor formula and simplifying, we obtain:  \\
{\footnotesize
\begin{eqnarray*}
&&
\left\{
\begin{array}{l}
\big[\cosh\mu_0 +\sinh\mu_0\, \mu_\l(\f) + o(\l) \big] \cos \f = (1+\l) \cosh \mu_0  [\cos \f - \sin\f \,  \Delta \f]  + o(\l)  \\
\big[\sinh\mu_0 +\cosh\mu_0 \, \mu_\l(\f) + o(\l) \big] \sin \f = (1+\l)  \sinh \mu_0 [\sin \f + \cos\f \,  \Delta \f] + o(\l) 
\end{array}
\right.\\
&&\\
&&
\left\{
\begin{array}{l}
\sinh\mu_0\, \cos \f\, \mu_\l + \cosh\mu_0 \,\sin\f \, \Delta\f = \l \cosh\mu_0 \,\cos\f + o(\l)\\
\cosh\mu_0 \,\sin\f \, \mu_\l - \sinh\mu_0\, \cos \f\, \Delta\f = \l \sinh\mu_0 \,\sin\l\f + o(\l).
\end{array}
\right.\\
\end{eqnarray*}
}
Therefore (we are  interested in $\mu_\l$):
\begin{eqnarray}\label{homothety}
\mu_\l(\f) &=&
\frac{\l\, \sinh\mu_0\cosh\mu_0}{(\sinh^2\mu_0\, \cos^2 \f +\cosh^2\mu_0 \,\sin^2\f )}  + o(\l) \nonumber\\
&=& 
\frac{\l\, \sqrt{1-e_0^2}}{1- e_0^2 \cos^2\f }   + o(\l). \\\nonumber
\end{eqnarray}

\medskip

\subsection*{Perturbing by a translation}
Let $\tau=(\tau_x,\tau_y)\in \R^2$ and consider a translation by $\tau$. We want to write the  translated ellipse  $\E_{\tau}$ as
$$
\E_\tau = \E_{e_0,c} + \mu_\tau,
$$
which is equivalent to
$$
\E(\tau_x,\tau_y, c,\mu_0, 0) = \E(0,0,c,\mu_0 + \mu_\tau,0).
$$
Hence, we have to solve the following system of equations:
$$
\left\{
\begin{array}{l}
c \cosh(\mu_0 + \mu_\tau(\f)) \cos \f = \tau_x + c \cosh \mu_0 \cos \f_\tau \\
c \sinh(\mu_0 + \mu_\tau(\f)) \sin \f = \tau_y+  c \sinh \mu_0 \sin \f_\tau
\end{array}
\right.
$$
where one should observe that the angle $\f$ changes as well. 
In particular, $\mu_\tau=o(1)$ and $\Delta \f:= \f_\tau - \f = o(1)$.
Applying Taylor formula and simplifying, we obtain:
{\footnotesize
\begin{eqnarray*}
&&
\left\{
\begin{array}{l}
\big[\cosh\mu_0 +\sinh\mu_0\, \mu_\tau + o(\|\tau\|) \big] \cos \f = 
\frac{\tau_x}{c} + \cosh \mu_0  [\cos \f - \sin\f \,  \Delta \f  + o(\|\tau\|) ] \\
\big[\sinh\mu_0 +\cosh\mu_0 \, \mu_\tau + o(\|\tau\|) \big] \sin \f = 
\frac{\tau_y}{c} + \sinh \mu_0 [\sin \f + \cos\f \,  \Delta \f  + o(\|\tau\|) ]
\end{array}
\right.\\
&&\\
&&
\left\{
\begin{array}{l}
\sinh\mu_0\, \cos \f\, \mu_\tau + \cosh\mu_0 \,\sin\f \, \Delta\f = \frac{\tau_x}{c}    + o(\|\tau\|)\\
\cosh\mu_0 \,\sin\f \, \mu_\tau - \sinh\mu_0\, \cos \f\, \Delta\f =  \frac{\tau_y}{c}  + o(\|\tau\|).
\end{array}
\right.\\
\end{eqnarray*}
}
Therefore:
{\footnotesize
\begin{eqnarray} \label{translation}
\mu_\tau(\f) &=&
\frac{1}{(\sinh^2\mu_0\, \cos^2 \f +\cosh^2\mu_0 \,\sin^2\f )}  \left[
\frac{\tau_x}{c}\sinh\mu_0\, \cos \f +  \frac{\tau_y}{c} \cosh\mu_0 \,\sin\f
\right]+ o(\|\tau\|)\nonumber\\
&=& \frac{e_0}{c(1-e_0^2 \cos^2\f)} \left[
\tau_x \sqrt{1-e_0^2} \cos \f + \tau_y \sin \f
\right]+ o(\|\tau\|).
\end{eqnarray}
}

\medskip

\subsection*{Perturbing by a rotation}
Let $\theta \in [0,2\pi)$ and consider a rotation by $\theta$ (counterclockwise); we denote 
$$
R_\theta := \left(
\begin{array}{cc}
\cos \th  & - \sin \th\\
\sin \th & \cos \th
\end{array}
\right).
$$
We are interested in the rotated ellipse $\E_{\th}$ and we want to write it (in elliptic coordinates) as
$$
\E_\th = \E_{e_0,c} + \mu_\th,
$$
which is equivalent to
$$
\E(0,0, c,\mu_0, \th) =   \E(0,0,c,\mu_0 + \mu_\th,0).
$$
Hence, we have to solve the following system of equations:
\begin{eqnarray*}
\left(
\begin{array}{c}
c \cosh(\mu_0 + \mu_\th(\f)) \cos \f \\
c \sinh(\mu_0 + \mu_\th(\f)) \sin \f 
\end{array}
\right)
= R_\th 
\left(
\begin{array}{c}
c \cosh \mu_0 \cos \f_\th\\
c \sinh \mu_0 \sin \f_\th
\end{array}
\right)
\end{eqnarray*}
where one should observe that the angle $\f$ changes as well. In particular, $\mu_\th=o(1)$ and $\Delta \f:= \f_\th - \f = o(1)$.
Applying Taylor formula and simplifying, we obtain:
{\footnotesize
\begin{eqnarray*}
\left(
\begin{array}{c}
\big[\cosh\mu_0 +  \sinh \mu_0\, \mu_\th \big] \cos \f \\
\big[\sinh\mu_0 +  \cosh\mu_0 \, \mu_\th\big] \sin \f 
\end{array}
\right)
&=& 
\left(
\begin{array}{cc}
1 & -\th\\
\th & 1
\end{array}
\right)
\left(
\begin{array}{c}
\cosh \mu_0 [\cos \f - \sin\f \,  \Delta \f  + o(\th) ]\\
\sinh \mu_0 [\sin \f + \cos\f \,  \Delta \f  + o(\th) ]
\end{array}
\right)\\
\end{eqnarray*}
}
which implies
{\footnotesize
\begin{eqnarray*}
\left(
\begin{array}{c}
\sinh \mu_0 \,\cos \f \\
\cosh\mu_0  \,\sin \f 
\end{array}
\right)\,\mu_\th
-
\left(
\begin{array}{c}
\cosh \mu_0 \,\sin \f \\
\sinh\mu_0  \,\cos \f 
\end{array}
\right)\,\Delta\f
  &=& 
\left(
\begin{array}{c}
-\sinh \mu_0 \,\sin \f \\
\cosh\mu_0  \,\cos \f 
\end{array}
\right)\, \th + o(\th).
\end{eqnarray*}
}
Hence, we conclude
\begin{eqnarray}\label{rotation}
\mu_\th &=&
\frac{\theta\,\left[
\sin\f\cos\f (\cosh^2 \mu_0-\sinh^2\mu_0)
\right]}{(\sinh^2\mu_0\, \cos^2 \f +\cosh^2\mu_0 \,\sin^2\f )}  + o(\th)\nonumber\\
&=& \frac{\theta\, e_0^2}{2(1-e_0^2\cos^2\f)} \sin2\f  
+ o(\th).
\end{eqnarray}

\medskip

\subsection*{Perturbing by an hyperbolic rotation}
 Let us consider the matrix 
$$
\Lambda =\Lambda(\l):= \left(
\begin{array}{cc}
e^\l  & 0\\
0  & e^{-\l}
\end{array}
\right) \qquad \mbox{with }\; \l\in \R;
$$
we are interested in the  ellipse $\E_{\Lambda}$ obtained by applyting this tranformation to $\E_0$ and we want to write it (in elliptic coordinates) as
$$
\E_\Lambda = \E_{e_0,c} + \mu_\Lambda,
$$
which is equivalent to
$$
\E(0,0, c,\mu_0, \th) =   \E(0,0,c,\mu_0 + \mu_\Lambda,0).
$$
Hence, we have to solve the following system of equations:
\begin{eqnarray*}
\left(
\begin{array}{c}
c \cosh(\mu_0 + \mu_\Lambda(\f)) \cos \f \\
c \sinh(\mu_0 + \mu_\Lambda(\f)) \sin \f 
\end{array}
\right)
= \Lambda
\left(
\begin{array}{c}
c \cosh \mu_0 \cos \f_\Lambda\\
c \sinh \mu_0 \sin \f_\Lambda
\end{array}
\right)
\end{eqnarray*}
where one should observe that the angle $\f$ changes as well. In particular, $\mu_\Lambda=o(1)$ and $\Delta \f:= \f_\Lambda - \f = o(1)$.
Applying Taylor formula and simplifying, we obtain:
{\footnotesize
\begin{eqnarray*}
\left(
\begin{array}{c}
\big[\cosh\mu_0 +  \sinh \mu_0\, \mu_\Lambda \big] \cos \f \\
\big[\sinh\mu_0 +  \cosh\mu_0 \, \mu_\Lambda\big] \sin \f 
\end{array}
\right)
&=& 
\left( 
\begin{array}{cc}
1+\l &0\\
0 & 1-\l
\end{array}
\right)
\left(
\begin{array}{c}
\cosh \mu_0 [\cos \f - \sin\f \,  \Delta \f  ]\\
\sinh \mu_0 [\sin \f + \cos\f \,  \Delta \f ]
\end{array}
\right) +   o(\l)\\
\end{eqnarray*}
}
which implies
{\footnotesize
\begin{eqnarray*}
\left(
\begin{array}{c}
\sinh \mu_0 \,\cos \f \\
\cosh\mu_0  \,\sin \f 
\end{array}
\right)\,\mu_\Lambda
-
\left(
\begin{array}{c}
\cosh \mu_0 \,\sin \f \\
\sinh\mu_0  \,\cos \f 
\end{array}
\right)\,\Delta\f
  &=& 
\l\,\left(
\begin{array}{c}
\cosh \mu_0 \,\cos \f \\
- \sinh\mu_0  \,\sin \f 
\end{array}
\right) + o(\l).
\end{eqnarray*}
}
Hence, we conclude
{\footnotesize 
\begin{eqnarray} \label{hyprot}
\mu_\Lambda &=&
\frac{\l \sinh\mu_0\cosh \mu_0 (\cos^2\f -\sin^2\f) }{(\sinh^2\mu_0\, \cos^2 \f +\cosh^2\mu_0 \,\sin^2\f )} \nonumber\\
&=&
\frac{\l}{1-e_0^2\cos^2\f}   \cos 2\f 
+ o(\l).\\ \nonumber
\end{eqnarray}
}

\medskip

\subsection*{Perturbation of Ellipses and Proof of Proposition \ref{changingellipse}}
Let us  first start with the following lemma, which is similar to 
\cite[Lemma 7]{ADK}.\\

\begin{lemma}\label{lemma13}
Let  $\E_{e_0,c} = \E(0,0,c,\mu_0,0)$ be an ellipse of eccentricity $e_0=1/\cosh \mu_0$ and semi-focal distance $c$, and
suppose that $\Omega$ is a  perturbation of $\E_{e_0,c}$,  which can be written (in the elliptic coordinate  frame $(\mu,\f)$ associated to $\E_{e_0,c}$) as
$
\Omega = \E_{e_0,c} + \mu_\Omega(\f).
$
Consider another ellipse $\overline{\E}$ sufficiently close to $\E_{e_0,c}$, which can be written (in elliptic coordinates frame associated to $\E_{e_0,c}$)  as
$$
\overline{\E} = \E_{e_0,c} + \mu_{\overline{\E}}.
$$
If $\overline{\E}$ is sufficiently close to $\E_{e_0,c}$, we can write (in the elliptic coordinate frame $(\overline{\mu},\overline{\f})$ associated to $\overline{\E}$)
$\Omega = \overline{\E} + \overline{\mu}_\Omega(\overline{\f})$, for some function $\overline{\mu}_\Omega$. Then, there exists $C=C(e_0,c)$ such that for every $\f\in [0,2\pi)$ we have
\begin{equation}\label{stima1}
|\mu_\Omega (\f) - (\mu_{\overline{\E}}(\f) + \overline{\mu}_\Omega(\f) )| \leq C \|\mu_{\overline{\E}}\|_{C^1} \| \| \mu_\Omega - \mu_{\overline{\E}}\|_{C^1}.
\end{equation}
Moreover,  for any $C'>1$, if $\overline{\E}$ is sufficiently close to $\E_{e_0,c}$ then we have
\begin{equation}\label{stima2}
\frac{1}{C'} \|\mu_\Omega - \mu_{\overline{\E}}\|_{C^1} \leq \|\overline{\mu}_{\Omega}\|_{C^1} \leq
{C'} \|\mu_\Omega - \mu_{\overline{\E}}\|_{C^1}.
\end{equation}
\end{lemma}

\bigskip

\begin{proof}
Let  
$$\overline{\E} = \E(\overline{x}_0, \overline{y}_0, \overline{c}, \overline{\mu}_0, \overline{\th}) = \E_{e_0,c} + \mu_{\overline{\E}}(\overline{\f}).\\$$ 

Consider the analytic change of coordinates between the coordinate frame $(\mu,\f)$ associated to $\E_{e_0,c}$ and the coordinate frame $(\overline{\mu},\overline{\f})$ associated to $\E$; we have:
\begin{equation} \label{changecoords}
\left\{
\begin{array}{l}
\overline{\mu}(\mu,\f) = \overline{\mu}_0+ \big[\mu - \mu_0-\mu_{\overline{\E}}(\f)\big] (1+ \rho_\mu(\mu-\mu_0,\f)) \\
\overline{\f}(\mu,\f)= \f+ \rho_\f(\mu-\mu_0,\f),
\end{array}
\right.
\end{equation}
where $\rho_\mu$ and $\rho_\f$ are analytic functions which are $C_1 \|\mu_{\overline{\E}}\|_{C^r}$-small in any $C^r$-norm, where $C_1=C_1(e_0,c, r)$. Observe that
$\overline{\mu}(\mu_0+\mu_{\overline{\E}}(\f)) \equiv \overline{\mu}_0$.\\

Let us observe the following facts:\\

\begin{itemize}
\item It follows from \eqref{changecoords} that
$$
\overline{\mu}_0+\overline{\mu}_\Omega\left(
\overline{\f}(\mu_0+\mu_\Omega(\f),\f) \right) = \overline{\mu} \left(
\mu_0+\mu_{\Omega}(\f),\f
\right).
$$
Taking the derivatives on both sides and using \eqref{changecoords} we obtain:
\begin{eqnarray*}
&&\overline{\mu}_{\Omega}'(\overline{\f}(\mu_0+\mu_\Omega(\f),\f)) \left[ 1+ \frac{\partial \rho_\f}{\partial \mu} (\mu_\Omega(\f),\f) \, \mu_\Omega'(\f) +
\frac{\partial \rho_\f}{\partial \f} (\mu_\Omega(\f),\f) \right]\\
&&\; =\;
\mu_{\Omega}'(\f)
\left[1+ \rho_\mu (\mu_\Omega(\f),\f) +  
\left(
\mu_\Omega(\f)-\mu_{\overline{\E}}(\f)
\right) \frac{\partial \rho_\mu}{\partial \mu} (\mu_\Omega(\f),\f)
\right]\, \\
&& \quad
-\; \mu_{\overline{\E}}'(\f) \left[1+ \rho_{\mu}(\mu_{\Omega}(\f), \f) \right] 
 +\; \left[\mu_\Omega(\f) - \mu_{\E}(\f) \right] 
\frac{\partial\rho_\mu}{\partial \f} (\mu_{\Omega}(\f),\f) \\
&&\;=\;
\left( \mu_{\Omega}'(\f) - \mu_{\overline{\E}}'(\f) \right)
\left[1+ \rho_\mu (\mu_\Omega(\f),\f)  \right] \\
&& \quad + \;
\left( \mu_{\Omega}(\f) - \mu_{\overline{\E}}(\f) \right)
\left[
\frac{\partial\rho_\mu}{\partial \mu} (\mu_{\Omega}(\f),\f)\mu_{\Omega}'(\f) +
\frac{\partial\rho_\mu}{\partial \f} (\mu_{\Omega}(\f),\f)
\right].
\end{eqnarray*}
Hence:
\begin{eqnarray}\label{eq:from-muto-mubar}
\overline{\mu}_{\Omega}'(\overline{\f}(\mu_0+\mu_\Omega(\f),\f)) &=& 
\frac{O_{e_0,c}(\|  \mu_{\Omega} - \mu_{\overline{\E}}\|_{C^1})  \left[1+ O_{e_0,c}(\|\mu_{\overline{\E}}\|_{C^1})
\right]
}
{1+ O_{e_0,c}(\|\mu_{\overline{\E}}\|_{C^1})},
\end{eqnarray}
where $O_{e_0,c}(\cdot)$ means that  {its absolute value 
is bounded by absolute value of $(\cdot)$ and a 
constant which } depends on $e_0$ and $c$. 
\\

\item 
Let  us denote by $\overline{\f}_\Omega (\f):= \overline{\f}(\mu_0+\mu_\Omega(\f),\f)$; it follows from \eqref{changecoords} that it is a diffeomorphism
and 
$$
\overline{\f}_\Omega'(\f) = 1+ O_{e_0,c}(\|\mu_{\overline{\E}}\|_{C^1}).
$$
In particular:
\begin{eqnarray*}
\overline{\mu}_\Omega' (\overline{\f})
&= &
\left(
\overline{\mu}_\Omega \circ \overline{\f}_\Omega \circ \overline{\f}_\Omega^{-1}
\right)'  (\overline{\f}) \\
&=&
\overline{\mu}_{\Omega}'(   
\overline{\f}_\Omega (\f)) \cdot \overline{\f}_\Omega'(\f)
\cdot \left(   \overline{\f}_\Omega^{-1} \right)'(\overline{\f}
).
\end{eqnarray*}
Along with \eqref{eq:from-muto-mubar}, this implies \eqref{stima2}.

\item Moreover, using that $\overline{\f}(\mu_0+\mu_\Omega(\f),\f) - \f = 
O_{e_0,c}(\|\mu_{\overline{\E}}\|_{C^0})
$ we obtain:
\begin{eqnarray*}
\overline{\mu}_\Omega\left(
\overline{\f}(\mu_0+\mu_\Omega(\f),\f) \right) - \overline{\mu}_\Omega(\f) &=&
\int_\f^{\overline{\f}(\mu_0+\mu_\Omega(\f),\f)} \overline{\mu}_{\Omega}' (t)dt \\
&=& O_{e_0,c}(\|  \mu_{\Omega} - \mu_{\overline{\E}}\|_{C^1}\|\mu_{\overline{\E}}\|_{C^1}).\\
\end{eqnarray*}

\item 
Since
$$
\Omega= \E_{e_0,c} + \mu_\Omega(\f)  = \overline{\E} + \overline{\mu}_\Omega(\overline{\f}),
$$
then we have:
\begin{eqnarray*}
\overline{\mu}_0 + \overline{\mu}_{\Omega}(\overline{\f}_{\Omega}(\f)) &=& \overline{\mu} \left(
\mu_0+\mu_\Omega(\f),\f
\right) \\
&=& 
\overline{\mu}_0+ \big[\mu_\Omega(\f) -\mu_{\overline{\E}}(\f)\big] (1+ \rho_\mu(\mu_\Omega(\f),\f));
\end{eqnarray*}
therefore
\begin{eqnarray*}
\overline{\mu}_{\Omega}(\overline{\f}_{\Omega}(\f)) - \left( \mu_\Omega(\f) -\mu_{\overline{\E}}(\f)\right) &=&
 \left( \mu_\Omega(\f) -\mu_{\overline{\E}}(\f)\right) \rho_\mu(\mu_\Omega(\f),\f) \\
 &=& O_{e_0,c}(\|  \mu_{\Omega} - \mu_{\overline{\E}}\|_{C^0}\|\mu_{\overline{\E}}\|_{C^0}).
\end{eqnarray*}
\end{itemize}

\bigskip

Summarizing all of the above information, we get:
\begin{eqnarray*}
\overline{\mu}_{\Omega}({\f}) -  \mu_\Omega(\f) + \mu_{\overline{\E}}(\f) 
&=& 
\big[\overline{\mu}_{\Omega}(\overline{\f}_{\Omega}(\f))   -  \mu_\Omega(\f) + \mu_{\overline{\E}}(\f) \big] + 
\big[
\overline{\mu}_{\Omega}(\overline{\f}_{\Omega}(\f))  - \overline{\mu}_{\Omega}({\f}) 
\big] \\
&=& 
O_{e_0,c}(\|  \mu_{\Omega} - \mu_{\overline{\E}}\|_{C^1}\|\mu_{\overline{\E}}\|_{C^1}),
\end{eqnarray*}
and this concludes the proof of \eqref{stima1}.  
\end{proof}

\bigskip

Now we are  ready to prove Proposition \ref{changingellipse}.

\begin{proof}({Proposition \ref{changingellipse}}) 
We use the notation introduced in \eqref{trans1}-\eqref{funchyprot} and proven in the first part of this section. 
Moreover, since we will be working with  
elliptic coordinate frames associated to different ellipses $\E_k$, we will adopt the convention 
to denote functions with a superscript $^{(k)}$,  when we consider them  with respect to the angle 
associated to the ellipse $\E_k$.\\

Let us denote $\Omega=\E_{e_0,c}+\mu^{(0)}$.
We consider different steps of approximation.
\begin{itemize}
\item[1)] Let us now consider the ellipse $\E_1$ obtained by translating $\E_{e_0,c}$ by a vector  
$$\tau=\left( \frac{a_1\, c}{e_0 \sqrt{1-e_0^2}}, \frac{b_1\,c}{e_0} \right).$$
Let $\mu^{(0)}_{\E_1}$ such that $\E_1=\E_{e_0,c}+\mu^{(0)}_{\E_1}$
 and $\mu^{(1)}_1$ be such that $\Omega=\E_1+\mu^{(1)}_1$.
It follows from \eqref{translation} that
\begin{equation}\label{estimatetrans}
\left\| \mu^{(0)}_{\E_1}  - ({a_1 e^{(0)}_{\tau1} + b_1 e^{(0)}_{\tau2}}) \right\|_{C^1} = O_{e_0,c} (a_1^2+b_1^2).
\end{equation}
Then, using Lemma \ref{lemma13}  and \eqref{estimatetrans} we obtain
\begin{eqnarray*}
\left\| \mu_1^{(0)} -  ({a_0 e^{(0)}_{h} + a_2 e^{(0)}_{hr} + b_2 e^{(0)}_r}) \right\|_{C^1} &\leq&
\| \mu_1^{(0)} - (\mu^{(0)}- \mu^{(0)}_{\E_1})\|_{C^1} \\
&+&
\left\| \mu^{(0)}_{\E_1}  - ({a_1 e^{(0)}_{\tau1} + b_1e_{\tau2}^{(0)}})  \right\|_{C^1}\nonumber \\
&=& O_{e_0,c} \left( \| \mu^{(0)}\|^2_{C^1} \right)\!;  
\end{eqnarray*}
in particular  we have used that 
$\|\mu^{(0)}_{\E_1}\|_{C^1} = O_{e_0,c}(\sqrt{a_1^2+b_1^2}) = O_{e_0,c}(\|\mu^{(0)}\|_{C^1})$.
Let us denote $\overline{\f}_1= \overline{\f}_1(\f)$ the angle associated to $\E_1$; it follows from computations similar to \eqref{translation} that
$$\|\overline{\f}_1- \f\|_{C^1}= O_{e_0,c}\left(\sqrt{a_1^2+b_1^2}\right).$$ 
Then, we conclude that:
\begin{eqnarray}
\left\| \mu^{(1)}_1 -  ({a_0 e^{(1)}_{h} + a_2 e^{(1)}_{hr}+ b_2 e^{(1)}_r})\right\|_{C^1}
&=& O_{e_0,c} \left( \| \mu^{(0)}\|^2_{C^1} \right).  \label{primastima} \\ \nonumber
\end{eqnarray}

\medskip

\item[2)] Let us consider the dilated/contracted ellipse 
$$\E_2= e^{\frac{a_0}{\sqrt{1-e_0^2}}} \E_1;$$
 let $\mu^{(1)}_{\E_2}$ be such that $\E_2=\E_1+\mu^{(1)}_{\E_2}$  and  $\mu^{(2)}_2$  such that $\Omega=\E_2+\mu^{(2)}_2$.
It follows from \eqref{homothety} that
\begin{equation}\label{estimatehom}
\left\| \mu^{(1)}_{\E_2} - {{a_0}e_{h}^{(1)}} \right\|_{C^1} = O_{e_0,c} (a_0^2).
\end{equation}
Then, proceeding as above and using Lemma \ref{lemma13},  \eqref{primastima} and \eqref{estimatehom}, we obtain
\begin{eqnarray*}
\left\| \mu_2^{(1)} -   ({a_2 e^{(1)}_{hr}+ b_2 e^{(1)}_r}) \right\|_{C^1} &\leq&
\| \mu_2^{(1)} - (\mu_1^{(1)}- \mu^{(1)}_{\E_2})\|_{C^1} \\
&+&
\left\| \mu^{(1)}_1 -  ({a_0 e^{(1)}_{h} + a_2 e^{(1)}_{hr}+ b_2 e^{(1)}_r})\right\|_{C^1}\\
&+&
\left\| \mu^{(1)}_{\E_2}  - {{a_0 e^{(1)}_{h} }} \right\|_{C^1}\nonumber \\
&=& O_{e_0,c} \left( \| \mu^{(0)}\|^2_{C^1} \right)\!;  
\end{eqnarray*}
Let us denote $\overline{\f}_2= \overline{\f}_2(\overline{\f}_1)$ the angle associated to $\E_2$; it follows from computations similar to \eqref{homothety} that
$$\|\overline{\f}_2- \overline{\f}_1\|_{C^1}= O_{e_0,c}\left(a_0\right).$$ 
Then, we conclude that:
\begin{eqnarray}
\left\| \mu_2^{(2)} -   ({a_2 e^{(2)}_{hr}+ b_2 e^{(2)}_r}) \right\|_{C^1}
&=& O_{e_0,c} \left( \| \mu^{(0)}\|^2_{C^1} \right).  \label{secondastima}\\ \nonumber
\end{eqnarray}

\medskip

\item[3)] Let us consider the rotated ellipse 
$$\E_3= R_{\frac{2 a_2}{e_0^2} } \E_2;$$
let $\mu^{(2)}_{\E_3}$ be such that $\E_3=\E_2+\mu^{(2)}_{\E_3}$ and  let $\mu^{(3)}_3$ be such that $\Omega=\E_3+\mu^{(3)}_3$.
It follows from \eqref{rotation} that
\begin{equation}\label{estimaterot}
\left\| \mu^{(2)}_{\E_3} - {b_2 e_r^{(2)}} \right\|_{C^1}=O_{e_0,c} (b_2^2).
\end{equation}
Proceeding as above (Lemma \ref{lemma13} and similar estimates) we get:
\begin{eqnarray*}
\left\| \mu^{(3)}_3 - {a_2 e_{hr}^{(3)} } \right\|_{C^1} &=&
O_{e_0,c} \left( \| \mu^{(0)}\|^2_{C^1} \right).\\
\end{eqnarray*}

\medskip

\item[4)] Finally, let us consider the ellipse obtained by means of an hyperbolic rotation $\Lambda(a_2)$:
$$\E_4= \Lambda(a_2)\, \E_3.$$
Let $\mu^{(3)}_{\E_4}$ such that $\E_4=\E_3+\mu^{(3)}_{\E_4}$ and let 
$\mu^{(4)}_4$ be such that $\Omega=\E_4+\mu^{(4)}_4$.
It follows from \eqref{hyprot} that
$$
\left\| \mu^{(3)}_{\E_4} - {a_2 e_{hr}^{(3)}} \right\|_{C^1} = O_{e_0,c}(a_2^2).
$$

In particular, proceeding as above, we conclude also in this case that
\begin{eqnarray*}
\| \mu_4^{(4)}\|_{C^1}
&=& O_{e_0,c} \left( \| \mu^{(0)}\|^2_{C^1} \right).\\
\end{eqnarray*}
\end{itemize}

\medskip

To conclude the proof, we denote $\widetilde{\E}:=\E_4$ and we consider $\mu_{\widetilde{\E}}$ such that 
$
\widetilde{\E}=\E_{e_0,c} +\mu_{\widetilde{\E}}.
$
It follows from Lemma \ref{lemma13} (second part of the statement) that
$$
\| \mu^{(0)} - \mu^{(0)}_{\widetilde{\E}}\|_{C^1} = O_{e_0,c}\left(
\| \mu_4^{(4)}\|_{C^1}
\right) = O_{e_0,c} \left( \| \mu^{(0)}\|^2_{C^1} \right),
$$
and this concludes the proof of the proposition.
\end{proof}

\medskip

\section{Analytic extensions and their singularities} \label{AppendixSingularities}

\subsection{Proof of Proposition \ref{prop_anal_cqsq}} \label{secanaliticitycqsq}
\hfill\\

Let us start by studying the zeros of 
\begin{equation}\label{hkappa}
h_k(z)=1-k^2 \sin^2 z
\end{equation}
for $0< k <1$. \\

\begin{remark}
Observe  that $k_q>0$ unless $e_0=0$, {\it i.e.}, the boundary of the billiard is a circle and $k_q\equiv 0$ for any $q\geq 3$; in this latter case, $h_0(z)\equiv 1$ and there are no zeros: in fact, $c_q$ and $s_q$ correspond to $\cos(q\,z)$ and $\sin (q\,z)$ which are entire functions. Hence, we consider only the case $0<k_q<1$.\\
\end{remark}

\blm \label{lm:zeroes-hk}
Let $k$ satisfy $0<k<1$
$$
h_k(z)=0 \qquad \Longleftrightarrow \qquad z_n = \left(\frac{\pi}{2} + n\pi \right) \pm i \, \rho_k \qquad {\rm for}\; n\in \Z.\\
$$
\elm 

\begin{proof}
 Recall that:
\begin{eqnarray*}
\sin(x+iy) = \sin x \cosh y + i \cos x \sinh y,
\end{eqnarray*}
therefore,
{\small 
\begin{eqnarray} \label{sin2}
\sin^2(x+iy) &=& (\sin^2 x \cosh^2 y - \cos^2 x \sinh^2 y) + 2 i \sin x \cos x \sinh y \cosh y \nonumber\\
&=& [\sin^2 x \cosh^2 y - \cos^2 x (\cosh^2 y-1)] +  i \sin (2x)  \sinh y \cosh y \nonumber\\
&=& [\cosh^2 y (\sin^2 x - \cos^2 x) + \cos^2 x]
 +  i \sin (2x)  \sinh y \cosh y \nonumber\\
&=& [-\cosh^2 y \cos (2x) + \cos^2 x]
 +  i \sin (2x)  \sinh y \cosh y.
\end{eqnarray}
}

In particular, denoting $z=x+iy$ we have
{\small 
\begin{eqnarray}
h_k(z) &=& 1- k^2 \sin^2 z \label{hk}\\
&=& [1 - k^2 \cos^2 x + k^2 \cosh^2 y \cos (2x) ]
 -  i k^2 \sin (2x)  \sinh y \cosh y \nonumber
\end{eqnarray}
}
and hence for $0<k<1$:
$$
h_k(z)=0 \qquad \Longleftrightarrow \qquad 
\left\{
\begin{array}{l}
1 - k^2 \cos^2 x + k^2 \cosh^2 y \cos (2x) =0\\
\sin (2x)  \sinh y \cosh y = 0.
\end{array}
\right.
$$
The second equation has solutions:
$$
({ i})\; \; x=\frac{m\pi}{2}  \quad ({\rm with}\; m\in \Z) \qquad {\rm or} \qquad  ({ii})\;\; y=0.
$$
If we plug those solutions in the first equation we obtain:
\begin{itemize}
\item[({\it i})] Let $x=\frac{m\pi}{2}$ and  let us distinguish two cases. 
\begin{itemize}
\item[a)] if $x=n\pi$, then the first equation becomes
$$
1- k^2  + k^2 \cosh^2 y > 0 \qquad \mbox{for} \quad 0<k<1. 
$$
\item[b)] if $x=\frac{(2n+1)\pi}{2}$, then the first equation becomes
$$
1 -  k^2 \cosh^2 y=0,
$$
hence
$$
\cosh^2 y=1/k^2 \qquad \Longleftrightarrow \qquad y_{\pm}=\pm \rho_k:= \pm \arcosh \left(1/k \right),$$
which is well defined since $0<k<1$.
\end{itemize}
\item[({\it ii})] If $y=0$, then the first equation becomes:
\begin{eqnarray*}
0 &=& 1 - k^2 \cos^2 x + k^2 \cos (2x) \\
&=& 1 - k^2 \sin^2 x,
\end{eqnarray*}
which does not admit solutions for $0<k<1$.
\end{itemize}
Summarizing, for $0<k<1$
$$
h_k(z)=0 \qquad \Longleftrightarrow \qquad z_n = \left(\frac{\pi}{2} + n\pi \right) \pm i \, \rho_k \qquad {\rm for}\; n\in \Z.
$$
\end{proof}

\medskip 

If we denote by $\Sigma_\rho$ the open complex strip of (half) width $\rho>0$ around the real axis, {\it i.e.},
$$\Sigma_{\rho}:= \left\{ z\in \C: \; \left|{\rm Im}(z)\right| <  \rho \right\},$$
then we conclude that $h_k$  is an entire function that, for $0<k<1$, does not
vanish in the strip  $\Sigma_{\rho_k}$. \medskip

Now we want to consider the complex function $\sqrt{h_k(z)}$ and understand its domain of analyticity. Recall the following elementary result from complex analysis:\medskip

\noindent {\it Let $f$ be a nowhere vanishing holomorphic function in a simply connected region $\Omega$. Then $f$ has  a holomorphic logarithm, and hence, a holomorphic square-root in $\Omega$.}\\

Therefore, we can conclude that the functions $\sqrt{h_k(z)}$ and $1/\sqrt{h_k(z)}$ are analytic in $\Sigma_{\rho_k}$.\\

If we consider, for $0<k<1$, the function $F(\f;k) := \displaystyle{\int_0^{\f} \frac{d\f}{\sqrt{1-k^2\sin^2 \f}} }$, then its complex extension is given by
$$
F(z; k) := \int_0^{z} \frac{d\z}{\sqrt{h_k(\z)}} .
$$
It follows from Cauchy's theorem that this function is well-defined and analytic in $\Sigma_{\rho_k}$. This completes the proof of Proposition \ref{prop_anal_cqsq}. \qed

\subsection{Proof of Proposition \ref{secanaliticityeh}} \label{appsingularitiesellipticmotions}
\hfill \medskip 
 
Observe that, using the notation introduced in \eqref{hkappa}, 
$$e_{h}(z)=\frac{1}{{h_{e_0}\left(z+\frac{\pi}{2}\right)}}.$$
It follows from the discussion in Section \ref{secanaliticitycqsq} that this function has singularities (which are poles) at
$$
\zeta_n =  n\pi  \pm i \, \rho_{0} \qquad {\rm for}\; n\in \Z,
$$
where $\rho_{0}= \arcosh \left(1/e_0 \right) = \mu_0$. In particular its maximal strip of analyticity is given by
$$\Sigma_{\rho_{0}} =  \left\{ z\in \C: \; \left|{\rm Im}(z)\right| <  \rho_{0} \right\}.$$
This concludes the proof of Proposition \ref{secanaliticityeh}. \qed
\\

%%%%%%%%%%%%%%%%%%%%%%%%%%%%%%%%%%%%%%%%

\section{Proof of Proposition \ref{newpropindep}} \label{lastbutnotleast}

Consider the following variational problem.

Given $0< j\leq 2q_0$, we would like to see how much $\bbe_j(\f)$ is 
linearly independent of the vector subspace 
{
$$
\Lambda_{q_0} := \overline{\langle \{\bbe_k\}_{ k>2 q_0}
\rangle}.
$$ }
{Observe that it suffices to consider an arbitrary $q_0$, since we already 
have linear independence for every finite subcollection. }\\

{We start by considering the case $j\geq 5$; for the other case, see Remark \ref{rem-jcaseless5}}.
Let us define $v_j$  as the vector realizing the minimal 
$L_{e_0}^2$-distance from the unit vector 
$\bbe_j$  to the subspace $\Lambda_{q_0}$; namely, 
if
\begin{eqnarray}\label{eq:v-e}
v_j &:=& \bbe_j - \sum_{k> 2q_0} d_{jk} \bbe_k,
\end{eqnarray}
then we require that  $v_j$ is orthogonal to all $\bbe_k$, 
for  $k> 2q_0$. Hence, we consider the 
{$L^2_{e_0}$-scalar product} of $v_j$ with $\bbe_m$, for 
$m > 2q_0$, and we impose that it is equal to  zero:
\begin{equation}\label{dotproductsvanish}
v_j \cdot {\bbe_m} =  \bbe_j \cdot \bbe_m - \sum_{k> 2q_0} 
d_{jk} \,(\bbe_k \cdot \bbe_m)=0. \\
\end{equation}
\medskip

\noindent {\bf Strategy of proof:} Notice that by definition 
each vector $v_j,\ 0<j\le 2q_0$ is the projection of $\bbe_j$ 
onto the orthogonal complement to $\Lb_{q_0}$. If 
the vectors $\{v_j,\ 0<j\le 2q_0\}$ are linearly independent 
(see Corollary \ref{letshopeitiscorrect}), then the subspaces 
$\langle \left\{\bbe_k\right\}_{0\leq k \leq  2q_0}\rangle $ 
and 
$\overline{ \langle \left\{\bbe_k\right\}_{k> 2q_0}\rangle}$ 
have zero dimensional intersection 
(see Proposition \ref{newpropindep}).
This, in turn, implies that $\{\bbe_j,\ j>0\}$ form 
a basis of $L^2(\T)$ 
(see   Lemma \ref{invertibilityimpliesbasis}).\\

{ The key idea to check linear independence of vectors 
$\{v_j,\ 0<j\le 2q_0\}$ is the same as in the case of finite 
linear combinations (see Proposition \ref{linearindep}).
In the case $\{\bbe_j,\ 0<j \le 2q_0\}$ singularities of 
the complex extensions are explicit and pairwise disjoint for 
$\bbe_i$ and $\bbe_j$ with $i\ne j$. We modify this idea 
for $\{v_j,\ 0<j\le 2q_0\}$ as follows: 
for each $5\leq j\leq 2q_0$ we would like to compare 
the maximal strips of analyticity of $\bbe_j$ and $v_j$ 
related by \eqref{eq:v-e}. 
Notice that the width of the maximal strip 
of analyticity of $\bbe_j$ equals $\rho_{k_{\widehat j}}$, 
while the width of the maximal strip of analyticity of $v_j-\bbe_j$ equals 
to the strip of analyticity of $\bbe_j-\sum_{m> 2q_0} d_{jm}\,\bbe_m$,
which turns out to equal $\rho_{k_{\widehat j}}$ 
(Corollary \ref{cor:analytic-strips}). Infinite linear independence 
will then follow, see Corollary \ref{letshopeitiscorrect}.\\

Let us introduce some notation. For $j\leq 2q_0 < k,m$, we define
\be \label{eq:a-entries}
a_{km}:= \bbe_k\cdot \bbe_m = \int_0^{2\pi} \bbe_k(\f)\, \bbe_m(\f)\, {(1-e_0^2\cos^2\f)^2}d\f
\ee
and 
\be \label{eq:b-entries}
b_{jm}:=\bbe_j\cdot \bbe_m = \int_0^{2\pi} \bbe_j(\f)\, \bbe_m(\f) \,{(1-e_0^2\cos^2\f)^2} d\f,
\ee 
{where the scalar product is meant in the weighted  space $L^2_{e_0}$.}

\medskip

Hence, we obtain the (infinite) {row} vector 
$$\vec B_{q_0}:=(b_{jm})_{m > 2q_0}$$
and the (infinite) square matrix 
$$
 A_{q_0}:=(a_{km})_{k,m > 2q_0}.
$$
In particular, {if we denote by $\vec{D}_{q_0}$  the infinite row vector
$$
\vec{D}_{q_0}:= (d_{jk})_{k> 2q_0},
$$
then equation \eqref{dotproductsvanish} becomes
\be \label{eq:D-matrix} 
\vec{D}_{q_0} \, A_{q_0}  =  \vec B_{q_0} .
\ee
In particular, if  $A_{q_0}$ is invertible, then 
$$
\vec{D}_{q_0}=  \vec B_{q_0} \, A_{q_0}^{-1}.
$$
}

\medskip 

Now we need to study $A_{q_0}$ and $\vec B_{q_0}$ for large $q_0$. 
{Notice that the matrix $A_{q_0}$ is a small perturbation of the identity,  
because by Lemma \ref{lm:corr-decay} for 
$k\ne m\to +\infty$ its elements $a_{km}$ decay exponentially 
{(we will make this more quantitative in the following)}.} The {vector} $\vec B_{q_0}$ 
has also components exponentially decaying in $m$ (it  follows from the estimates in 
Lemma \ref{lm:corr-decay} too). {To compare maximal strips of analyticity of $v_j$ 
and $\bbe_j$ for each $j\le 2q_0$ we need to estimate the exponent of the speed of 
decay of elements of $\vec{D}_{q_0}$.  Our analysis starts with the following lemma.}\\

\noindent {{\bf Notation.} Hereafter, given an integer $q\in\N$, we will denote  
$\widehat{q}:=\left[ \frac{q+1}{2}\right],$ where $[\,\cdot\,]$ denotes the integer part. 
This cumbersome notation is needed since for every integer $q$ we have couples 
$\bbe_{2q}$ and $\bbe_{2q-1}$ corresponding to the same rotation number $1/q$.
Whenever it is possible, in the forthcoming statements and proofs, we will try 
to ease notation as much as possible.
}

\medskip

\begin{theorem} \label{lm:corr-decay} 
For every $e_0>0$, there exists
$q_0=q_0(e_0)$ such that the following holds. For each $j\geq 3$  there exists
$\lambda_j \in (0,1)$ such that for any $\delta>0$ there is ${C_j:=C_j(e_0,\delta)}>0$ 
such that for  each $3\le j \leq m$ 
\be \label{eq:upper-corr-bound<}
{|a_{jm}-\delta_{jm}|\le C_j (\lambda_j+\delta)^{{\widehat{m}}},}
\ee
{where $\widehat{m}:=\left[ \frac{m+1}{2}\right]$.}
Moreover, for $2q_0<  j \le m$ we have   
\be \label{eq:upper-corr-bound>}
|a_{jm}-\delta_{jm}|\le C^*\, (\lambda^*+\delta)^{{\widehat{m}}},
\ee
{for some $C^*=C^*(e_0,\delta)$ and $\lambda^*=\lambda^*(e_0)<1$.}
\end{theorem}

\smallskip

\begin{remark} 
{We will see 
that we can choose  $\lambda_j=\exp[-\rho_{k_j}(1+\kappa^*))]$, 
for some {suitable} $\kappa^*=\kappa^*(e_0)>0$. 
Moreover, by studying  the growth of the constants $C_j$, we show that we can choose
{$\lambda^*= \exp[-(\sigma_\infty (\rho_{k_{q_0}})  - \rho_{k_{q_0}})]$, where
$\sigma_\infty (\rho_{k_{q_0}})  - \rho_{k_{q_0}}>0$ 
(see  \eqref{defsigmainfty} for a definition of $\sigma_\infty(\cdot))$.}}\\
\end{remark}

%\medskip

\begin{proof} 
Recall from \eqref{changenotation2}  that
$$
\bbe_{2j} := \dfrac{c_j}{\|c_j\|_{L^2_{e_0}}},  \quad \bbe_{2j-1}: = \frac{s_j}{\|s_j\|_{L^2_{e_0}}} \qquad \forall\;j\geq 3.
$$
In particular, up to multiplication by constants, we have:
$$
\bbe_{2j}(\f) \asymp  \frac{\cos(j \frac{2\pi}{4K(k_j)} F(\f,k_j))}{\sqrt{1-k_j^2\sin^2\f}},
\qquad \bbe_{2j-1}(\f) \asymp  \frac{\sin({{j}} \frac{2\pi}{4K(k_j)} F(\f,k_j))}{\sqrt{1-k_j^2\sin^2\f}}.
$$

Let us now denote
$$
t_{2j}(\f) \,=\, t_{2j-1}(\f) \,:=\, \frac{2\pi}{4K(k_j)} F(\f,k_j)
$$
and their inverses
$$
 \f_{2j} (t) \,=\, \f_{2j-1} (t) \,:=\, {\rm am}\,\left(\frac{4K(k_j)}{2\pi} t, k_j\right);
$$
then 
$$
\bbe_{2j}(\f) \asymp  \cos(j t_j(\f)) \, \frac{dt_j}{d\f}(\f) \quad {\rm and} \quad
\bbe_{2j-1}(\f) \asymp  \sin(j t_j(\f)) \, \frac{dt_j}{d\f}(\f).
$$

\medskip

{We need to compute $\bbe_j \cdot \bbe_m$. 
Observe that if $j=m$, then it is $1$, since they are unit vectors with respect to the $L^2_{e_0}$-scalar product. 
Let us assume that $j < m$. 
Doing a change of coordinate in the corresponding integral, we get {(we consider the case in which  
both indices are even, since the other cases are analogous)}:
\begin{eqnarray*}
\bbe_{2j} \cdot \bbe_{2m} &=& \int_0^{2\pi} \bbe_{2j}(\f) \,\bbe_{2m}(\f)\, (1-e_0^2\cos^2\f)^2\,d\f 
\\
&=& 
\int_0^{2\pi} \bbe_{2j}(\f_m(t)) \, \cos (m t)\, \frac{dt_m}{d\f_m} \frac{d\f_m}{dt}\,(1-e_0^2\cos^2\f_m(t))^2\,dt \\
&=&\int_0^{2\pi} \bbe_{2j}(\f_m(t))\,(1-e_0^2\cos^2\f_m(t))^2 \,\cos (m t) \,dt. 
\end{eqnarray*}
}

{Hence, we are computing the $m$-th Fourier coefficients of the function} 
\begin{eqnarray}\label{Ejm}
E_{jm}(t) &:=& {\bbe_{2j}(\f_m(t)) \,}
{(1-e_0^2\cos^2{\f_m(t)})^2} \nonumber\\
&=&
\frac{\cos \left(j \frac{2\pi}{4K(k_j)} F\left({\f_m(t)},k_j\right)\right)} {\sqrt{1-k_j^2\sin^2 \left(  {\f_m(t)}\right)}}\ {(1-e_0^2\cos^2\f_m(t))^2}\nonumber\\
&=&
\frac{\cos \left(j F\left({\rm am}\,(\frac{4K(k_m)}{2\pi} t, k_m),k_j\right)\right)} {\sqrt{1-k_j^2 \, \sn^2 \left( \frac{4K(k_m)}{2\pi} t, k_m\right)}}\ {\left(1-e_0^2\, \cn^2 \left( \frac{4K(k_m)}{2\pi} t, k_m\right)\right)^2}.
\end{eqnarray}

{In order to compute the decay rate of its Fourier coefficients, we need to analyze its maximal strip of analyticity.}

\medskip 

Recall that $k_j$ represents the eccentricity of the caustic of rotation number $1/j$. In particular, it is  strictly decreasing  with respect to $j$ and
$$
k_j > k_m > e_0\qquad \forall\,2< j< m.
$$

\medskip

First of all, observe (see Remark \ref{rmk14}) that  $\sn(z,k)$ and $\cn(z,k)$ have simple poles with imaginary parts $iK(k')$, 
where $k'$ denotes the complementary modulus $k':=\sqrt{1-k^2}$. Hence, 
$\sn ( \frac{4K(k_m)}{2\pi} t, k_m)$ has maximal strip of analyticity of width equal to 
$2\pi\frac{K(k_m')}{4K(k_m)}$.\\
On the other hand,  $\cos(\cdot)$ is an entire function.
Thus, the singularities of $E_{jm}$ can be of two types: singularities of the last bracket
and vanishing of the denominator. The first type singularity occurs at 
$i2\pi\frac{K(k_m')}{4K(k_m)}$.\\

Hence, it remains only to study when the denominator of $E_{jm}$ vanishes:
$$
1-k_j^2\sin^2 \left(  {\rm am}\,\left(\frac{4K(k_m)}{2\pi} \zeta, k_m\right)\right) =0.
$$
Proceeding as in Lemma \ref{lm:zeroes-hk}, if follows that the above equality is achieved when
\begin{eqnarray*}
{\rm am}\,\left(\frac{4K(k_m)}{2\pi} \zeta, k_m \right) = \frac{\pi}{2} + \pi n \pm i \rho_{k_j},
\end{eqnarray*} 
where $\rho_{k_j} = \arcosh\, (1/k_j)$.
In particular, the solutions of this equation are:
\begin{eqnarray*}
\zeta_n &:=& \frac{2\pi}{4K(k_m)} F\left(\frac{\pi}{2} + \pi n \pm i \rho_{k_j}, k_m \right) \\
&=& \frac{2\pi}{4K(k_m)} \left( F\left(\frac{\pi}{2} \pm i \rho_{k_j}, k_m \right) + 2n K(k_m)
\right) \\
&=& \frac{2\pi}{4K(k_m)}  F\left(\frac{\pi}{2} \pm i \rho_{k_j}, k_m \right) + \pi n.
\end{eqnarray*} 
Observe that $\rho_{k_j} < \rho_{k_m}$, so the points $\frac{\pi}{2} \pm  i \rho_{k_j}$ are inside the strip of analyticity of $F(\cdot ;k_m)$.

The above expression can be expanded further. In fact, observe that
\begin{eqnarray*}
F\left(\frac{\pi}{2} \pm  i \rho_{k_j}, k_m \right) &=&
K(k_m) + \int_{\frac{\pi}{2}}^{\frac{\pi}{2}\pm i \rho_{k_j}} \frac{1}{\sqrt{1-k_m^2\sin^2 z}} 
\,dz \\
&=& K(k_m) \pm i \int_{0}^{\rho_{k_j}} \frac{1}{\sqrt{1-k_m^2\cosh^2 t}} 
\,dt, \\
\end{eqnarray*}
where in the last equality we have used that $\sin^2 \left( \frac{\pi}{2} + it\right) = \cosh^2 t$.
Hence, the singularities are at
\begin{eqnarray*}
\zeta_n &:=& \frac{2\pi}{4K(k_m)}  F\left(\frac{\pi}{2} \pm i \rho_{k_j}, k_m \right) + \pi n \\
&=& \frac{\pi}{2} + \pi n \pm i  \frac{2\pi}{4K(k_m)} \int_{0}^{\rho_{kj}} \frac{1}{\sqrt{1-k_m^2\cosh^2 t}}\, dt.
\end{eqnarray*}

The quantity
\begin{eqnarray} \label{eq:sigma-m-rho}
\sigma_{m}(\rho_{k_j}) &:=&  \frac{2\pi}{4K(k_m)}  \min \left\{ \int_{0}^{\rho_{k_j}} \frac{1}{\sqrt{1-k_m^2\cosh^2 t}}\,dt\;, \;
K(\sqrt{1-k_m^2})\right\} \nonumber\\
&=& {\frac{2\pi}{4K(k_m)}  \int_{0}^{\rho_{k_j}} \frac{1}{\sqrt{1-k_m^2\cosh^2 t}}\,dt,}
\end{eqnarray}
provides the width of the strip of analyticity of $E_{jm}$; {the proof of the last equality
follows from Lemma \ref{damnedinequality} with $x=k_m$ and $y=k_j$, observing that 
 $0<k_m<k_j$ for $j< m$.
}\\

Notice that the entries $a_{2j,2m}$, defined by \eqref{eq:a-entries},  can be viewed as 
Fourier coefficients of the functions $E_{jm}$. The latter ones has the strip of analyticity, given by   
$\sigma_m(\rho_{k_j})$. 
For fixed $j$, these widths  are strictly decreasing in $m$ and, in the limit as $m\rightarrow +\infty$, 
they tend to    
\begin{equation}\label{defsigmainfty}
\sigma_{\infty}(\rho_{k_j}):=\frac{2\pi}{4K(e_0)}  \int_{0}^{\rho_{k_j}} \frac{1}{\sqrt{1-e_0^2\cosh^2 t}}\,dt,
\end{equation}
which is strictly increasing in $j$. In fact, consider the function 
$$
W(x,y) := \frac{2\pi}{4K(x)}  \int_{0}^{\arcosh 1/y} \frac{1}{\sqrt{1-x^2\cosh^2 t}}\,dt.
$$
defined for $0<x<y<1$. It suffices to show that it is  increasing with respect to $x$. 
Since $x=k_m$ is decreasing with respect to $m$, it will follow that it is decreasing. 
This can be shown using lengthy, but elementary manipulation or using 
Mathematica. In Figure \ref{figplots} we present two plots: the first one is the graph of 
$W$ and the second one is the graph of the partial derivative of $W$ 
with respect to $x$, which turns out to be positive.

\begin{figure} [h!] 
\begin{center}
\includegraphics[scale=0.29]{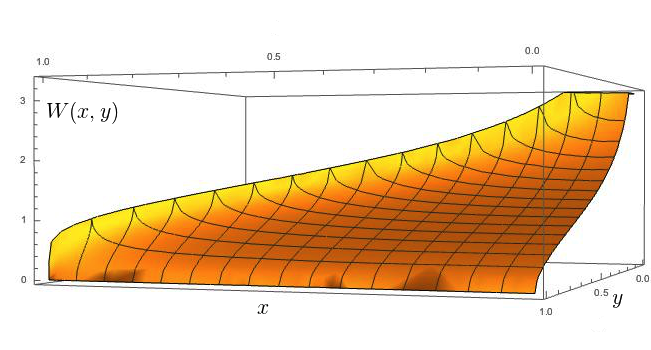} \ \ \ \ \ \ 
\includegraphics[scale=0.285]{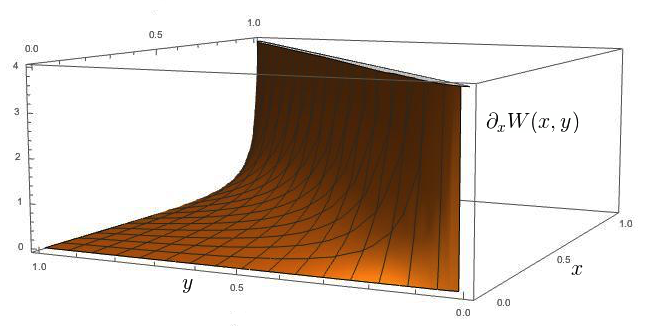}
\caption{Plots of $W(x,y)$ and $\partial_x W(x,y)$.}
\label{figplots}
\end{center}
\end{figure}

We can now deduce \eqref{eq:upper-corr-bound<} by applying 
Paley-Wiener theorem\footnote{\label{PWthm}
Let us briefly recall the statement of this theorem (see for example: 
\url{http://www.math.lsa.umich.edu/~rauch/555/fouriercomplex.pdf}):\\
{\bf Theorem (Paley-Wiener)}. {\it 
If $f$ is an analytic periodic function in the strip $\{|\Im z|< a\}$ for some $a>0$, 
then its Fourier coefficients $c_n$ satisfy the following property: 
for any $\epsilon>0$ there exists $C(\epsilon)>0$ such that 
$|c_n| \leq C(\epsilon) e^{(-a+\epsilon)|n|}$ for every $n \in \Z$. 
Conversely, if $\{c_n\}_n$ satisfy the above property, then  
$f :=\sum_{n\in\Z} c_n e^{inz}$ has an analytic continuation to the strip $\{|\Im z|< a\}$.
}\\
In particular, $C(\epsilon)$ is bounded from above by the supremum of
$|f|$ on the strip $\{|\Im z|\leq a-\epsilon\}$.\\
}.
Observe that we can choose $\lambda_m=\exp[-\sigma_\infty(\rho_{k_j})]$ 
(we will show in Proposition \ref{prop:anal-bnds} that
$\sigma_\infty(\rho_{k_j})> \rho_{k_m}(1+\kappa^*) )$, 
for some suitable $\kappa^*=\kappa^*(e_0)>0$ that will be explicitely determined).\\

\medskip
{Now, we want to prove \eqref{eq:upper-corr-bound>}. In order to do this, 
we need to get a better control on the constants $C_j$. In particular, we need to estimate
$$
\left|\cos \left(j \frac{2\pi}{4K(k_j)} F\left({\f_m(z)},k_j\right)\right)\right| = \left|\cos \left(j  \frac{2\pi}{4K(k_j)} F\left({\rm am}\,(\frac{4K(k_m)}{2\pi} z, k_m),k_j\right)\right)\right|
$$
on the complex strip of width $(\rho_{k_j}-\delta)$.
Since $|\cos(x+iy)|$ grows like $e^{|y|}$, then we need to estimate 
$$
\left| \Im \left(  \frac{2\pi}{4K(k_j)} F\left({\rm am}\,(\frac{4K(k_m)}{2\pi} t, k_m),k_j\right)\right)  \right| = 
\frac{2\pi}{4K(k_j)} \left| \Im \left(   F\left({\rm am}\,(\frac{4K(k_m)}{2\pi} t, k_m),k_j\right)\right)  \right|
$$
on the strip of width $(\rho_{k_j}-\delta)$.\\
Since $F(\cdot, k_m)$ and $\varphi_m$ are  one the inverse of the other, then it follows that for $2q_0<j\leq m$, then there exists $\widetilde{C}(e_0,\delta)>0$ such that
$$
 \left| \Im \left(   F\left({\rm am}\,(\frac{4K(k_m)}{2\pi} z, k_m),k_j\right)\right) \right| \leq \widetilde{C} \,\frac{4K(k_m)}{2\pi} |\Im (z)|
$$
for every $z$ in the complex strip of width $\rho_{{k_j}}-\delta$.
}

{
Hence,
$$
|E_{jm}(z)| \leq C\, \exp(j (\rho_{{k_j}}-\delta ))
$$
in the $(\rho_{k_j}-\delta)$-strip, for some $C=C(e_0,\delta)$.
\\
}

{Now with this bound at hand we can deduce from Paley-Wiener theorem (see footnote \ref{PWthm}) that
\begin{eqnarray*}
| {\mathbb e}_{2j} \cdot {\mathbb e}_{2m} | &\le&  C^*(e_0, \delta ) \, \exp(j (\rho_{k_j} - \delta))\, 
\exp(-m\, \sigma_\infty (\rho_{k_j}))\\
&\leq&C^*(e_0, \delta )\,  \exp(- m (\sigma_\infty (\rho_{k_{q_0}})  - \rho_{k_{q_0}} + \delta)).
\end{eqnarray*}
Since $\rho_{k_j} >\rho_{k_{q_0}}$ for every $j>q_0$, we can choose 
$\lambda^* = \exp(- (\sigma_\infty (\rho_{k_{q_0}})  - \rho_{k_{q_0}}))$. 
{We point out that $\sigma_\infty (\rho_{k_{q_0}})  - \rho_{k_{q_0}} > \kappa^* \rho_{k_{q_0}}>0$, 
as it follows from Proposition \ref{prop:anal-bnds}}.
}
\end{proof}

\bigskip

{Let us prove this Lemma, that was used in the proof of Theorem \ref{lm:corr-decay}.}\\

{
\begin{lemma}\label{damnedinequality}
For $0<x\leq y <1$ we have
$$
I(x,y):=\int_{0}^{\arcosh (1/y)} \frac{1}{\sqrt{1-x^2\cosh^2 t}}\,dt \leq
K(\sqrt{1-x^2}),
$$
with equality only for $x=y$.\\
\end{lemma}
}
{
\begin{proof}\footnote{{There is an alternative proof of this Lemma using the Reduction Theorem 
for General Elliptic Integrals (see e.g. \url{https://dlmf.nist.gov/19.29}). One can represent both integrals 
using the canonical form $R_F$ and then relate them using the representation formula for $R_F$ and in terms of 
$R_C$ (see \url{https://dlmf.nist.gov/19.23)}}}
Clearly, $I(x, y)$ is strictly increasing with respect to $x$, while  
$K(\sqrt{1-x^2})$ is strictly decreasing with respect to $x$.\\
The claim follows from the fact that for any $0<y<1$, we have   
$$I(y,y)=K(\sqrt{1-y^2}).$$
In fact, consider the following change of variable in the integral defining $I(y,y)$: 
$$
\cosh^2 t -1= (1/y^2-1) \sin^2 \theta,
$$
which implies
\begin{eqnarray*}
\sinh t = \sqrt{1+(1/y^2-1)} \sin \theta 
\end{eqnarray*}
and 
\begin{eqnarray*}
\cosh t &=& \sqrt{1+(1/y^2-1) \sin^2 \theta} \\
dt &=& \frac{\sqrt{1/y^2-1} \cos \theta\, d\theta}{\sqrt{1+ (1/y^2-1)\sin^2\theta}}.
\end{eqnarray*}
Then:
\begin{eqnarray*}
I(y,y) &=& \int_0^{\pi/2}   \frac{1}{\sqrt{1-y^2} \cos\theta}    \frac{\sqrt{1/y^2-1} \cos \theta\, d\theta}{\sqrt{1+ (1/y^2-1)\sin^2\theta}} \\
&=& \frac{1}{y} \int_0^{\pi/2} \frac{d\theta}{\sqrt{1+ (1/y^2-1)\sin^2\theta}}\\
&=& \frac{1}{y} \int_0^{\pi/2} \frac{d\theta}{\sqrt{\cos^2 \theta + 1/y^2  \sin^2\theta}}\\
&=& \frac{1}{y} \int_0^{\pi/2} \frac{d\theta}{\sqrt{1/y^2 - (1/y^2-1) \cos^2\theta}}\\
&=& \int_0^{\pi/2} \frac{d\theta}{\sqrt{1 - (1-y^2) \cos^2\theta}}\\
&=& \int_0^{\pi/2} \frac{d\theta}{\sqrt{1 - (1-y^2) \sin^2\theta}} = K(\sqrt{1-y^2}).\\
\end{eqnarray*}
\end{proof}
}

\bigskip

{The width of the  strip of analyticity of $v_j-\bbe_j$ depends on 
the exponent of the speed of decay of elements of $\vec{D}_{q_0}$. 
We will compare   now the width of strips of analyticity of $v_j$ and $v_j-\bbe_j$ for each 
$j<2q_0$.  \\
}

{We  need the following estimate to compare {$\sigma_{ m}(\rho_{k_{ j}})$} with {$\rho_{k_{j}}$}}.

\begin{proposition} \label{prop:anal-bnds}
There is a decreasing sequence $\kappa_m \ge \kappa^*:=\kappa^*(e_0)>0$ such that for any $m>j\geq 3$ we have 
$$
\rho_{k_j} < \sigma_m(\rho_{k_j})-\rho_{k_j} \kappa_m.
$$
In particular, 
\[
\rho_{k_j} < \sigma_\infty(\rho_{k_j})-\rho_{k_j} \kappa^*.\\
\]
\end{proposition}

\medskip

\begin{proof}
Recall the definition of $\sigma_m(\rho_{k_j})$ in \eqref{eq:sigma-m-rho}.
There is {$\kappa'_m=\kappa'_m(k_m)>0$} such that   
\begin{eqnarray}\label{defkappa}
4\rho_{k_j} K(k_m)&=&\rho_{k_j} \int_0^{2\pi} \dfrac{dt}{\sqrt{1-k_m^2\sin^2 t}}\;  \nonumber\\
&< &\; \rho_{k_j}\left( \frac{\pi}{\sqrt{1-k_m^2}}+\frac{\pi}{\sqrt{1-k_m^2/2}}\right)
\nonumber\\
&{=:}&\; \rho_{k_j}\left( \frac{2\pi}{\sqrt{1-k_m^2}}-\kappa'_m\right)
\\
&=& 2\pi \int_0^{\rho_{k_j}} \dfrac{dt}{\sqrt{1-k_m^2 }} \;-\rho_{k_j}\kappa'_m  \nonumber\\ 
&<& \; 2\pi \int_0^{\rho_{k_j}} \dfrac{dt}{\sqrt{1-k_m^2 \cosh^2 t}}-\rho_{k_j}\kappa'_m,\nonumber
\end{eqnarray}
{where
$$\kappa'_m:= \pi\left(\frac{1}{\sqrt{1-k_m^2}}-\frac{1}{\sqrt{1-k_m^2/2}}\right),
$$
which is striclty decreasing\footnote{This follows from the fact that the function $\frac{1}{\sqrt{1-x^2}}-\frac{1}{\sqrt{1-x^2/2}}$ is strictly increasing in $[0,1)$.} as a function of $m$ and, as $m\rightarrow +\infty$, tends to
$$
\kappa'_\infty:= \pi\left(\frac{1}{\sqrt{1-e_0^2}}-\frac{1}{\sqrt{1-e_0^2/2}}\right) >0.
$$
}

Dividing on both sides of \eqref{defkappa} by $4 K(k_m)$ we get 
$$
\rho_{k_j} < \sigma_m(\rho_{k_j})-\dfrac{\rho_{k_j}\kappa'_m}{4K(k_m)}\qquad \mbox{for any $m>j\geq 3$}.
$$
{Denote $\kappa_m=\frac{\kappa'_m}{4K(k_m)}$; this function is also strictly decreasing\footnote{This follows from the fact that the function $\frac{1}{K(x)}\left(\frac{1}{\sqrt{1-x^2}}-\frac{1}{\sqrt{1-x^2/2}}\right)$ is strictly increasing in $[0,1)$.} as a function of $m$ and, as $m\rightarrow +\infty$, tends to
$$
\kappa^*:= \frac{\pi}{4K(e_0)}\left(\frac{1}{\sqrt{1-e_0^2}}-\frac{1}{\sqrt{1-e_0^2/2}}\right) >0.
$$
}
\end{proof}

\bigskip

Let ${\mathbb I}$ denote the Identity (infinite) matrix and let us denote
\[
A_{q_0} ={\mathbb I} +\Delta A_{q_0},  
\]   
where $\Delta A_{q_0}:=(a_{km}-\delta_{km})_{k,m > 2q_0}$
and 
\[
|a_{km}-\delta_{km}|\le C^*{(\lambda^* + \delta)^{q_0}.}  
\]

\begin{lemma}
\label{lm:convolve-bound}
Using the same notation as in Lemma \ref{lm:corr-decay}, assume that 
$q_0$ is chosen so that 
\be \label{eq:q0-est}
\sum_{ m> 2q_0} C^*(\lambda^*+\delta)^{{\widehat m}} \le \frac 14.
\ee
Then, for any $h,k> 2q_0$
we have 
\[
\left | \sum_{m> 2q_0} (a_{hm}-\delta_{hm})(a_{mk}-\delta_{mk}) \right | \le
\frac{C^*}{4}(\lambda^*+\delta)^{{\max\{\widehat k, \widehat h\}}}.  
\]
{In particular, this implies that   
\begin{equation}\label{DA2}
|(\Delta A_{q_0})^2_{h,k}|\leq \frac{C^*}{4}(\lambda^*+\delta)^{{\max\{\widehat k, \widehat h\}}}.
\end{equation}
Inductively, one can show that for every $N\geq 2$
\begin{equation}\label{DAN}
|(\Delta A_{q_0})^N_{h,k}| \leq \frac{C^*}{4^{N-1}}(\lambda^*+\delta)^{{\max\{\widehat k, \widehat h\}}}.
\end{equation}
}
\end{lemma}

\begin{proof} Without loss of generality we  assume $2q_0<h\leq k$ (indeed, estimates are symmetric with respect to switching 
indices $h$ and $k$). {Using \eqref{eq:upper-corr-bound>}  and   \eqref{eq:q0-est}:}
\begin{eqnarray*}
\left|  \sum_{m > 2q_0} (a_{hm}-\delta_{hm})(a_{mk}-\delta_{mk})\right| &\le& 
C^*(\lambda^*+\delta)^{{\widehat k}} \sum_{m> 2q_0} |a_{hm}-\delta_{hm}| \\
&\leq&
C^*(\lambda^*+\delta)^{{\widehat k}} \left(\sum_{m>2q_0} C^*(\lambda^*+\delta)^{{\widehat m}} \right)
\\
&=& \frac{C^*}{4}\ (\lambda^*+\delta)^{{\widehat k}}, 
\end{eqnarray*}
{which  implies \eqref{DA2}. As for \eqref{DAN}, it suffices to proceed by induction on $N$: assume that the estimate holds for $N\geq 2$, then
\begin{eqnarray*}
|(\Delta A_{q_0})^{N+1}_{h,k}\| &\leq& \left|  \sum_{m > 2q_0} (a_{hm}-\delta_{hm})(\Delta A_{q_0})^{N}_{m,k}\right|\\
&\leq&
\frac{C^*}{4^{N-1}}(\lambda^*+\delta)^{{\widehat k}} \left(\sum_{m>2q_0} C^*(\lambda^*+\delta)^{{\widehat m}} \right)
\\
&=& \frac{C^*}{4^{N}}\ (\lambda^*+\delta)^{{\widehat k}}. 
\end{eqnarray*}}
\end{proof}

\bigskip

Let us  now consider 
$$A_{q_0}^{-1} = ({\mathbb I}+ \Delta A_{q_0})^{-1}=
 {\mathbb I} + \sum_{N\ge 1} (-\Delta A_{q_0})^{N}.
$$
Applying Lemma \ref{lm:convolve-bound}, we deduce that 
 the $(k,m)$ entry of the matrix
\[
A_{q_0}^{-1}- {\mathbb I},
\]
that we denote by $a^-_{km}$, is bounded by 
\[
|a^-_{km}| \le 2C^*(\lambda^*+\delta)^{{\widehat m}}. \\
\]

\bigskip

Then, combining this with the estimates on the decays of the elements of $\vec{B}_{q_0}$ proved in Lemma \ref{lm:corr-decay}, we obtain the  following lemma 
(in particular, it uses the that fact  that $\sum_{m>2q_0} |b_{j,m}|<+\infty$).
\\

\begin{lemma} \label{lemmadjk}
Let $d_{jk}$ be the $(j,k)$--entry of 
\[
\vec{D}_{q_0}=  \vec B_{q_0} \cdot A_{q_0}^{-1},
\]
with $j\leq 2q_0< k$, then {there exists} $C^*>0$ such that 
for all $k> 2q_0$ we have 
\[
|d_{jk}|\le C^* (\lambda^*+\delta)^{{\widehat k}}.
\]
\end{lemma}

\medskip

%tpcgfcnsna
%nnnprtggc
For each $5\leq j\leq 2q_0$ we need to compare the  maximal strips of analyticity of $\bbe_j$ and $v_j$ related by \eqref{eq:v-e}.
Notice that the width of the maximal strip of analyticity of $\bbe_j$ equals 
$\rho_{k_{\widehat j}}$.
{On the other hand, using the estimates in Lemma \ref{lemmadjk} and the analitity properties of $\bbe_k$,
we  conclude that
$\sum_{k>2q_0}  d_{jk} \bbe_k$ has strip of analiticity not smaller than 
$
\sigma_\infty (\rho_{k_{q_0}})  - \rho_{k_{q_0}} + \rho_{k_{\widehat{k}}}   > \rho_{k_{\widehat{j}}},
$ for $j\leq 2q_0<k$. Hence $v_j$ has width of analiticity $\rho_{k_{\widehat{j}}}$.\\
}

\bcor \label{cor:analytic-strips}
For each $5\leq j\le 2q_0$ the functions $v_j$ and $\bbe_j$ 
related by \eqref{eq:v-e} are real analytic and have maximal 
strips of analyticity  $\rho_{k_{\widehat j}}$.
\ecor

\medskip

\begin{remark}\label{rem-jcaseless5}
The case corresponding to $0<j\leq 4$ can be treated similarly.  Recalling the definitions of these $\bbe_j$ in subsection \ref{sec_sing_ell_mot} (see also \eqref{trans1}--\eqref{funchyprot}),
it follows that the main modifications correspond to a simpler expression for $E_{jm}$ in \eqref{Ejm}, in which the denominator disappears and the singularities are 
given by the ones of $\f_m(t)$:
$$
E_{jm}(t) =  u(\hat{j}\, \f_m(t))\left(1-e_0^2\, \cos^2 \f_m(t)\right),
$$
where $u(\cdot)$ denotes either sine or cosine.\\
Hence, the corresponding strip of analyticity is independent of $j$:
$$
\sigma_{\widehat m} :=  \frac{2\pi}{4K(k_{\widehat m})} \  
K(\sqrt{1-k_{\widehat m}^2}).
$$
One can prove similarly that the corresponding functions $v_j$'s for $0<j\leq 4$ 
have different strips of analyticity from the ones corresponding to the case $j\geq 5$.\\
\end{remark}

\begin{corollary} \label{letshopeitiscorrect}
Any non-trivial linear combination of 
the functions  $\{v_j\}_{j\leq 2q_0}$ is non-zero, {\it i.e.}, 
they are linearly independent. 
\end{corollary} 

\medskip 

\begin{proof}
The claim easily follows from the fact that we are considering finite linear combinations of analytic functions, with different maximal strips of analyticity.
\end{proof}

\medskip

Finally, we can conclude the proof of Proposition \ref{newpropindep}.

\medskip

\begin{proof}{(Proposition \ref{newpropindep})}
If  we had
$$
\sum_{j=1}^{2q_0} \alpha_j \bbe_j \in \overline{\langle \{\bbe_j\}_{ j> 2q_0}\rangle},
$$
then  
$$
\sum_{j=1}^{2q_0} \alpha_j v_j =0.
$$
It follows from Corollary \ref{letshopeitiscorrect} that $\alpha_1=\ldots = \alpha_{2q_0} =0$, which completes the proof.
\end{proof}

\medskip

%%%%%%%%%%%%%%%%%%%%%%%%%%%%%%%%%%%%%%%%%

\section{Some technical lemmata} \label{AppendixTechnical}

Let us recall the expression of the angles of the action-angle coordinates, see \eqref{philambda}; for the sake of simplicity, as before, we denote  by $k_{q}$ the eccentricity of the caustic of rotation number $1/q$ (with $q\geq 3)$:
$$\f_q(\xi) := \am \left( \frac{4K(k_q)}{2\pi}\,\xi ; k_q\right)$$
and its inverse
$$\xi_q(\f) := \frac{2\pi}{ 4K(k_q)}F \left( \f ; k_q\right). 
$$
Similarly, we denote the corresponding functions corresponding to boundary and rotation number $0$ ({\it i.e.}, in the limit as $q\rightarrow +\infty$) :
$$\f_\infty(\xi) := \am \left( \frac{4K(e_0)}{2\pi}\,\xi ; e_0\right)$$
and its inverse
$$\xi_\infty(\f) := \frac{2\pi}{ 4K(e_0)}F \left( \f ; e_0\right),
$$
where we have used that $k_q\rightarrow e_0^+$ in the limit as $q\rightarrow +\infty$. \\

\begin{lemma}\label{lem:deviation}
For  each $q\ge 1$ 
$$\xi_q(\xi_\infty) - \xi_\infty = O_{e_0,c}(1/q^2)$$
and
$$
{k_q-e_0 = O_{e_0,c}(1/q^2).}
$$
\end{lemma}

\medskip

\begin{proof}
Observe that
\begin{eqnarray} 
\label{eq:t_infty}
\xi_q(\xi_{\infty}) &=& 
\frac{2\pi}{4K(k_{{q}})}\,  F\left(\am\left(\frac{4K(e_0)}{2\pi}\, \xi_{\infty} ; {e_0}\right); k_{q}\right) \nonumber \\
&=& \xi_\infty + \left[
\frac{2\pi}{4K(k_{{q}})}\,  F\left(\am\left(\frac{4K(e_0)}{2\pi}\, \xi_{\infty} ; {e_0}\right); k_{q}\right) - \xi_{\infty}\right] \nonumber\\
&=& \xi_\infty + 
\frac{\pi}{2} \int_{e_0}^{k_q}  
\underbrace{
			\frac{\partial }{\partial k} \left(
				\frac{F\left(\am\left(\frac{4K(e_0)}{2\pi}\, \xi_{\infty} ; {e_0}\right); k\right)}{K(k)}\right)}_{=: \;\alpha(\xi_\infty, k)} \, dk.  
\end{eqnarray}

Hence:
\begin{eqnarray}\label{ximinusxiinf}
|\xi_q(\xi_{\infty}) - \xi_{\infty}| &\leq&  \frac{\pi}{2} \left( \max_{e_0\leq k\leq k_3(e_0)} \max_{\xi_{\infty} \in [0,2\pi)} |\alpha(\xi_\infty, k)| \right) (k_q-e_0) \nonumber\\
&\leq& C(e_0,a) (k_q-e_0).
\end{eqnarray}

\bigskip

In order to conclude the proof,  we need to estimate 
$k_q-e_0.$

By definition of $k_q=k_{\l_q}$ (see Proposition \ref{prop1})
we have  
\[
k_q^2=\frac{a^2-b^2}{a^2-\l_q^2} =
\frac{a^2 e_0^2}{a^2-\l_q^2},
\]
from which we deduce that
\begin{eqnarray}\label{kminuse}
k_q - e_0 &=& \frac{a e_0}{\sqrt{a^2-\l_q^2}} - e_0 \nonumber\\
&=&  \frac{e_0 \l_q^2}{\sqrt{a^2-\l_q^2} \left( a+ \sqrt{a^2-\l_q^2}  \right)}.
\end{eqnarray}

\medskip

Using definition \eqref{defrotnumber} we obtain
\be \label{eq:k_q}
\frac 2q = \frac{F(\arcsin \frac{\l_q}{b};k_q)}{K(k_q)} \quad \Longleftrightarrow \quad
\frac 2q  \, K(k_q) = {F(\arcsin \frac{\l_q}{b};k_q)}.
\ee
Rewrite using the definition of both $F$ and 
$K$, and the fact that $b=a \sqrt{1-e_0^2}$,
we obtain an implicit equation for $\l_q$ (observe that  $k_q=k_q(\l_q)$):
\be\label{e:lb}
\frac{2}{q}\int_0^{\pi/2}
\frac{d\f}{\sqrt{1-k_q^2\,\sin^2\f}}
=\int_0^{\arcsin \l_q/ (a\sqrt{1-e_0^2})} 
\frac{d\f}{\sqrt{1-k_q^2\,\sin^2\f}}.\\ \nonumber
\ee

Since $k_q \in [e_0,k_3]$ for all $q\geq 3$, then
$$
1\leq \frac{1}{\sqrt{1-k_q^2\,\sin^2\f}} \leq \frac{1}{\sqrt{1-k_3^2}},
$$
hence if we substitute in \eqref{e:lb} we deduce
$$\frac{\pi}{q} \frac{1}{\sqrt{1-k_3^2}} 
\geq 
\arcsin \left( \frac{ \l_q}{ a\sqrt{1-e_0^2}}\right).
$$
In particular, if $q\geq 2/\sqrt{1-k_3^2}=:q_0(e_0)$ we have
\begin{equation}\label{estimatelambdaq}
\l_q\leq
a\sqrt{1-e_0^2} \, \sin \left(  \frac{\pi}{q} \frac{1}{\sqrt{1-k_3^2}} \right),
\end{equation}
namely $\l_q=O_{e_0,a}(1/q)$.

Substituting in \eqref{kminuse} and \eqref{ximinusxiinf}, and observing that $c=a\sqrt{1-e_0^2}$,
we conclude that 
$$
\xi_q(\xi_{\infty}) - \xi_{\infty} = O_{e_0,c}(1/q^2),
\quad \text{ and }\quad 
{k_q-e_0 = O_{e_0,c}(1/q^2).}$$
\end{proof}

\medskip

\begin{lemma}\label{decayfourier}
Let $f:[0,2\pi) \longrightarrow \R$ a $C^1$ function. Then, there exists $C=C(e_0,c)$ such that for each $q\geq 3$:
$$
\left|   
\int_0^{2\pi} f(\f) c_q(\f)\,d\f
\right| \leq \frac{C \,\|f\|_{C^1}}{q}
$$
and
$$
\left|   
\int_0^{2\pi} f(\f) s_q(\f)\,d\f
\right| \leq \frac{C\, \|f\|_{C^1}}{q}.
$$
\end{lemma}

\medskip

\begin{proof}
If follows from the definition of $c_q$ (see \eqref{eq:c_q-function}), $\xi_q$, $\f_{\infty}$ and $\xi_\infty$, that 
\begin{eqnarray*}
\int_0^{2\pi} f(\f) c_q(\f)\,d\f &=&
\frac{4K(k_q)}{2\pi} \int_0^{2\pi} f( \f ) \cos(q\, \xi_q(\f)) \,\xi_q'(\f)\,d\f \\
&=&
\frac{4K(k_q)}{2\pi} \int_0^{2\pi} f( \xi_{\infty} ) \cos(q\, \xi_q(\xi_{\infty})) \,\xi_q'(\xi_\infty)\,  \f_{\infty}'(\xi_{\infty}) d\xi_{\infty} \\
&=&
\frac{4K(k_q)}{2\pi} \int_0^{2\pi} f( \xi_{\infty} ) \cos(q\, \xi_q(\xi_{\infty})) \, \frac{d}{d\xi_\infty}\left(\xi_q(\xi_\infty)\right) \, d\xi_{\infty}. 
\end{eqnarray*}
Using Lemma \ref{lem:deviation}:
\begin{eqnarray*}
\int_0^{2\pi} f(\f) c_q(\f)\,d\f &=&  \frac{4K(k_q)}{2\pi} 
\int_0^{2\pi} \left( f( \xi_{\infty} ) \cos(q \xi_{\infty}) + O_{e_0,c}(1/q)\right) \, d\xi_{\infty}\\
&=&  \frac{4K(k_q)}{2\pi} 
\int_0^{2\pi}  f( \xi_{\infty} ) \cos(q \xi_{\infty}) \, d\xi_{\infty}  + O_{e_0,c}\left(\frac{\|f\|_{C^0}}{q}\right).\\
\end{eqnarray*}

Observe that $\f_{\infty}=\f_{\infty}(\xi_{\infty})$ is an analytic function, so $f( \xi_{\infty} )$ is $C^1$ and its $q$-th Fourier coefficient are $O_{e_0,c}(\|f\|_{C^1}/q)$; hence we conclude
$$\int_0^{2\pi} f(\f) c_q(\f)\,d\f  = O_{e_0,c}\left(\frac{\|f\|_{C^1}}{q}\right),$$
which proves the first relation. In the same way, one proves the other one involving $s_q$.
\end{proof}

\bigskip

For $q\in \N$ and $j\geq 3$, let us consider the elements of the (infinite) correlation matrix $\widetilde{A}=(\ta_{i,h})_{i,h=0}^{\infty}$,  
{introduced in (\ref{matr:corr-coeff}), Section \ref{sec:basis-prop}.}

\medskip

\begin{lemma} 
There exists $\rho=\rho(e_0,c)>0$ such that for all $q\in\N$ and $j\geq 6$:
\begin{eqnarray*}
\ta_{q,j}  = 2K(k_{[j/2]}) \,\delta_{q,j} + 
O_{e_0,c}\left(  j^{-1} \,e^{-\rho\,{|q-j|}}\right),
\end{eqnarray*}
where $[\cdot]$ denotes the integer part and $\delta_{q,j}$ the Dirac's delta.\\
\end{lemma}

\begin{proof}
We proceed similarly to what done in Lemma \ref{decayfourier}. 
In particular,  recall formula \eqref{eq:t_infty}
\[
\xi_q(\xi_\infty)=\xi_\infty + \frac \pi 2 \int_{e_0}^{k_q}
\alpha(\xi_\infty,k)dk =: \xi_{\infty} + \Delta_q (\xi_\infty).
\]
Observe that $\Delta_q$ is analytic in a complex strip of width at least $\rho=\rho(e_0,c)>0$ 
(independent of $q$) and that there exists $C=C(e_0,c)$ such that $q^2 \|\Delta_q\|_{\rho} \leq C$ 
for all $q\geq3$, where $\|\cdot\|_\rho$ denotes the analytic norm of the function in the strip 
\textb{$\{|{\rm Im }z| \leq \rho\}$ (namely, the sup-norm on this closed strip of the modulus of its complex extention)}. This follows from the second part of Lemma \ref{lem:deviation}, 
namely the fact that $q^2(k_q-e_0)$ is uniformly bounded.

Recalling the definition of $c_q$, $s_q$, $\xi_q$, $\f_{\infty}$ and $\xi_\infty$, 
 we obtain the following (we prove it only in one case, the proofs of the others are identical):
{\footnotesize
\begin{eqnarray*} 
\ta_{2q,2j} &=& \int_0^{2\pi} \cos (q \f) \,c_{j}(\f)\, d\f \\
&=& \frac{4K(k_j)}{2\pi } \int_0^{2\pi} \cos( q\f ) \cos(j\, \xi_j(\f)) \,\xi_j'(\f)\,d\f \\
&=&
 \frac{4K(k_j)}{2\pi } \int_0^{2\pi} \cos( q\xi_{\infty} ) \cos(j\, \xi_j(\xi_{\infty})) \,\xi_j'(\xi_\infty)\,  \f_{\infty}'(\xi_{\infty}) d\xi_{\infty} \\
 &=&
  \frac{4K(k_j)}{2\pi }\int_0^{2\pi} \cos( q \xi_{\infty} ) \cos(j\, \xi_j(\xi_{\infty})) \, \frac{d}{d\xi_\infty}\left(\xi_j(\xi_\infty)\right) \, d\xi_{\infty}\\
&=&
 \frac{4K(k_j)}{2\pi }\int_0^{2\pi} \cos( q \xi_{\infty} ) \cos(j\, \xi_\infty + j \Delta_j(\xi_{\infty})) \, \left( 1+ \frac{d}{d\xi_\infty} \Delta_j(\xi_\infty)   \right) \, d\xi_{\infty}\\
 &=&
 \frac{4K(k_j)}{2\pi }\int_0^{2\pi} \cos( q \xi_{\infty} )  \big[
 \cos (j\xi_\infty) \cos (j \Delta_j(\xi_{\infty})) -  \sin (j\xi_\infty) \sin (j \Delta_j(\xi_{\infty}))
  \big]\\
  && \qquad \qquad \qquad \cdot
  \left( 1+ \frac{d}{d\xi_\infty} \Delta_j(\xi_\infty)   \right) \, d\xi_{\infty}\\
    &=&
 \frac{K(k_j)}{\pi }\int_0^{2\pi} \Big[ \big[
 \cos ((q+j)\xi_\infty) + 
\cos((q-j)\xi_\infty) \big] \cos (j \Delta_j(\xi_{\infty})) \\
&& \qquad \qquad \qquad -
 \big[
 \sin ((q+j)\xi_\infty) - 
\sin((q-j)\xi_\infty) \big] \sin (j \Delta_j(\xi_{\infty})) \Big] \\
  && \qquad \qquad \qquad \cdot
  \left( 1+ \frac{d}{d\xi_\infty} \Delta_j(\xi_\infty)   \right) \, d\xi_{\infty}\\
  &=& 2K(k_j) \delta_{q,j} +
 \frac{K(k_j)}{\pi }\int_0^{2\pi} \Big[ \big[
 \cos ((q+j)\xi_\infty) + 
\cos((q-j)\xi_\infty) \big] \big(\cos (j \Delta_j(\xi_{\infty}))-1\big) \\
&& \qquad \qquad \qquad -
 \big[
 \sin ((q+j)\xi_\infty) - 
\sin((q-j)\xi_\infty) \big] \sin (j \Delta_j(\xi_{\infty})) \Big] \\
  && \qquad \qquad \qquad \cdot
  \left( 1+ \frac{d}{d\xi_\infty} \Delta_j(\xi_\infty)   \right) \, d\xi_{\infty}.
\end{eqnarray*}
}

Since $\Delta_j$ is analytic in the strip of width $\rho$, then also 
$\frac{d}{d\xi_\infty} \Delta_j$,  $\sin (j{\Delta_j})$ and $\cos (j{\Delta_j})-1$
are also analytic in the same strip and their analytic norm in the strip of width $\rho$ is at most $O_{e_0,c}(1/j)$; hence, their Fourier coefficients decay exponentially.
It suffices to notice that the above integral consists of a combination of their Fourier coefficients.
Therefore:
\newline 
$. \qquad \qquad \qquad \ta_{2q,2j}=2K(k_{j}) \,\delta_{q,j} + O_{e_0,c}
\left(  \frac1j e^{-\rho\, {|q-j|}}\right).$
\end{proof}

\section{From local to global Birkhoff conjecture} \label{affineflowideas}

In this appendix we want to outline some ideas on how to use our local results to prove the global Birkhoff conjecture. {Roughly 
speaking, we would like to use the Affine Length Shortening (ALS) 
PDE flow, which evolves any convex domain into an ellipse \cite{SaTa}, 
in order to extend our result from a small neighborhood of ellipses 
to all  strictly convex domains. The idea we outline here  is to find a {\it Lyapunov 
function} for the flow, which measures the {\it non-integrability} of a domain. 
{Moreover, we propose to reduce the analysis to glancing periodic orbits, 
which stay in a nearly {integrable} zone during the whole ALS evolution}.  \\

\subsection{Affine length and affine curvature of a plane curve}
Let us first recall some definitions (see for instance \cite{SaTa}). Let $C:\T \to \R^2$ be an embedded closed curve with curve parameter 
$p$. A reparametri\-zation  $s$ can be chosen so that in the new 
parameter $s$ one has {(hereafter we will use the shorthand to use subscripts to denote derivatives)}
\[
[C_s,C_{ss}]=1,
\]
where $[X,Y]$ stands of the determinant of the $2\times 2$ matrix 
whose columns are given by vectors $X,Y\in\R^2$. Notice that the relation 
is invariant under the $SL_2(\R)$-transformations.  Call 
the parameter $s$ the {\it affine arc-length}; in particular, if 
\[
g(p)=[C_p,C_{pp}]^{1/3}
\]
then the parameter $s$ is explicitly given by 
\[
s(p)=\int_0^p g(\xi)\,d\xi.
\]
Assume $g(\xi)$ is non-vanishing, which is automatically 
satisfied for strictly convex curves. 
 
Call {\it the affine curvature} $\nu(s)$ the function given by 
\[
\nu=[C_{ss},C_{sss}].
\]
 
{\it The affine perimeter} for the closed curve $C$ is then defined by 
\[
L:=\oint g(p)\,dp. 
 \] 
 
{
 \begin{remark}
 In analogy with what happens for the Euclidean curvature,  the curves of 
 constant affine curvature $\nu$ are precisely all non-singular  conics.  
 More specifically,  those with $\nu>0$ are ellipses, those with $\nu=0$ 
 are parabolas, and those with $\nu<0$ are hyperbolas.
 \end{remark}
 }
 
 {
To conclude this subsection, 
let us point out  the relation between the (constant) affine curvature of 
an ellipse $\nu_0$ and its instant eccentricity $\mu_0$ (in elliptic coordinates).
One can easily show that
\[
\mu_0 = \arsinh\, (2 \nu_0^{-3/2}/c^2)/2. 
\]
}

{Moreover, if we consider a domain $\Omega$ which is $\e$-close to 
an ellipse $\E$ (of instant eccentricity $\mu_0$ and affine curvature $\nu_0$), 
and we denote by $\nu(s)$ the affine curvature of $\partial \Omega$ and by 
$\mu(s)=\mu_0 + \e \mu_1(s)$ the instant eccentricity in the elliptic coordinate 
frame associated to $\E$, as in \eqref{formulaperturbellipse}, then:  
\[
\mu(s)= f(\nu(s))=f(\nu_0+\e \nu_1)=
f(\nu_0)+\e f'(\nu_0) \nu_1+O(\e^2),
\]
where $f(a)=\arsinh\,(2a^{-3/2}/c^2)/2$. 
Thus, Fourier expansion of $\mu_1$ coincides with 
Fourier expansion of $\nu_1$ up to $O(\e^2)$-error.
}
 }

 \medskip
 
 \subsection{Affine Length Shortening (ALS) flow} \label{ALSflow}
The study of evolution of plane curves in the direction of 
the Euclidean normal with speed proportional to the Euclidean 
curvature (also known as {\it curve-shortening flow}) has been intensively 
studied, see for example  {\cite{Grayson} and references therein}.
The classical result says that the Euclidean curvature evolution is 
a ``{\it Euclidean curve shortening}'' and flows every convex domain 
toward a circle\footnote{Actually it shrinks every curve to a point. 
However, rescaling of either the perimeter (or the area) the curve will 
converge to a circle}. More specifically, for any closed convex curve
the isoperimetric ratio (\ie the ratio between the squared curve length 
and the area) decreases monotonically (and {in finite time}) to  $4\pi$, \ie the value of this ratio for circles.
  
 Adapting this idea, Sapiro and Tannenbaum \cite{SaTa} developed an analogous flow describing 
the evolution of plane curves in the direction of the {\it affine} normal, with speed proportional to 
the {\it affine} curvature; this flow is generally called  the {\it affine length shortening (ALS) flow} 
(or {\it affine curvature flow})  and, analogously to the Euclidean one, it is  ``affine length shortening''. 
Similarly to the Euclidean curvature evolution, in fact, this flow evolves every convex domain to an ellipse. 
More specifically, the isoperimetric ratio (\ie the ratio between the squared affine 
curve length and the area) decreases monotonically to  $8\pi^2$, 
which is the ratio for ellipses\footnote{Also in this case 
the flow shrinks every curve to a point. However, under rescaling 
of either the perimeter (or the area) the curve will converge to an ellipse.}.

\medskip 

\subsection{Application to billiards}

Our idea is to apply  the above geometric flow to deduce the non-integrability 
of a domain, by means of a suitable {\it Lyapunov function}. Let us describe 
this construction more specifically.

Let $\Om\subset \R^2$ be a strictly convex domain with a sufficiently 
smooth boundary $\partial \Omega$. Let $s$ be the arc-length parameter of 
the boundary and let us denote by $|\partial \Omega|$ its Euclidean perimeter.

For each {$q>2$} and for every point $s$ on the boundary, let us 
denote by $L_{1/q}(s)$ the maximal perimeter of a $q$-gon starting at this point. 
For each {$q>2p>1$} and for every point $s$ on the boundary, let us denote by 
$L_{p/q}(s)$ the maximal perimeter of a star shape $q$-gon starting at this point 
whose rotation number is $p/q$ and the points are ordered on the boundary in 
the same cyclical order as the rotation by $p/q$. Notice that if there exists 
an integrable rational caustic of rotation number $p/q$, then   $L_{p/q}(s)$ is 
constant or, equivalently  

\be \label{eq:Delta_q} 
\Delta_{p/q} := \int_0^{|\partial \Omega|} (L_{p/q}(s) - \langle L_{p/q} \rangle)^2 \,ds 
=0, \ 
\text{ where }\langle L_{p/q} \rangle = \int_0^{|\partial \Omega|} L_{p/q}(s) ds.
\ee  

\medskip

Suppose now that for any strictly convex domain $\Omega$ which is 
sufficiently close to an ellipse, but not an ellipse, the billiard map in 
$\Omega$ satisfies one of the two conditions:

\begin{enumerate}
\item either it has caustics for all rotation numbers in $(0,1/q_0]$ 
for some $q_0>2$; 
\item or it has a sequence $q_k\to \infty$, as $k \to \infty$, and a sequence $p_k$ 
such that $\frac{p_k}{q_k}\to 0$ and there is no integrable rational caustic of rotation 
number $\frac{p_k}{q_k}$, or, equivalently, $\Delta_{p_k/q_k}\ne 0$.  
\end{enumerate}

{This situation corresponds to a stronger version of the local 
Birkhoff conjecture  than the one proved in the present article. 
So far, this picture has been proven to  hold true  only for ellipses of 
small eccentricities (see  \cite{HKS}).}
 
\medskip  
 
Recall that, as we explained in subsection \ref{ALSflow}, for 
any convex domain $\Omega$, {different from an ellipse}, 
its evolution $\Omega_t$ under the  ALS flow brings it into 
a neighborhood of the ellipses. Thus, for some $T>0$ we have that 
\begin{itemize}
\item $\Omega_T$  belongs to a neighborhood of ellipses, 
\item there is a sequence $q_k\to \infty$ as $k \to \infty$ and a sequence 
$p_k$ such that $p_k/q_k\to 0$ and the billiard map associated to 
$\Omega_T$ has no integrable rational caustics of rotation number $p_k/q_k$ or, 
equivalently, $\Delta_{p_k/q_k}\ne 0$. 
\end{itemize}

We conjecture the following. \medskip

\noindent{\bf Conjecture} 
{\it Let $\Omega_t$ the evolution of the domain $\Omega$ under 
the {normalized} affine curvature flow  (\ie we keep the perimeter, or the area, 
of the domain fixed along the flow) and let $\Delta_{p/q}^t$ be 
the $\Delta_{p/q}$--function associated $\Omega_t$. Then, there exists 
$q_0=q_0(\Om)>2$  such that for some rational $0<p/q<1/q_0$  
we have that $\Delta_{p/q}^t$ is monotone in $t$}. 

\medskip 

Hereafter we verify a local version of this conjecture when
$\Om$ is the unit circle. See Lemma \ref{lem:monotone}.

\subsection{ALS flow evolution}

Let us first describe some results on the ALS flow.

In  \cite[(32) page 96]{SaTa} 
the formula for the evolution of the affine curvature $\nu$ is derived 
\be \label{eq:ALS-flow}
\dfrac{\partial \nu}{\partial t}= 
\frac{4}{3} \nu^2 + \frac{1}{3} \nu_{ss}.
\ee

Let us describe what happens in the case of ellipses, \ie 
$\nu\equiv \nu_0$; in particular, we  want to point out a subtlety of 
this flow, namely, certain blow up in a finite time. Then  
\[
\dfrac{\partial \nu}{\partial t}= \frac{4}{3} \nu^2
\]
becomes an ODE. If we make a substitution 
$\nu = \chi^{-1}$, then 
 \[
\frac{4}{3} \nu^2 = \dfrac{\partial \nu}{\partial t}= -  
\dfrac{1}{\chi^2}\,\dfrac{\partial \chi}{\partial t}= 
-\nu^2\,\,\dfrac{\partial \chi}{\partial t}. 
\] 
Thus, $\chi(t)=\chi_0-\frac{4t}{3}$ and $\nu_0(t)=
\frac{3}{3\chi_0-4t}.$ Notice that in finite time $\nu_0(t)$ 
blows up. It corresponds to the area of the corresponding 
curve converging to zero. See discussions in \cite[Section 7.1]{SaTa}. 
In \cite[Section 8.1]{SaTa}  bounds on the time of blow up
are presented in terms of minimal and maximal 
affine curvature $\nu$. 

Denote the above solution $\nu_0(t)$. Notice that one needs 
to rescale $\nu$, {\it e.g.}, to keep the area inside the domain 
fixed. If no rescaling is done, then the domain collapses to a point.  
Indeed, let  $\nu(s,t)=\nu_0(t)+\e \Delta \nu(s,t)$ for small 
$\e$, then we get 
\begin{eqnarray*}
\dfrac{\partial \nu}{\partial t}& =& \dfrac{\partial \nu_0}{\partial t}+
\e \dfrac{\partial \Delta \nu}{\partial t} \\
&=&
\frac{4}{3} (\nu_0^2 +2 \e \nu_0 \Delta \nu+ \e^2 \Delta \nu^2)
+ \frac{1}{3}   (\nu_0)_{ss}
 +\frac{1}{3}  \e \Delta \nu_{ss}.
\end{eqnarray*}
Simplifying
\[
\dfrac{\partial \Delta \nu}{\partial t}= 
\frac{4}{3} (2\nu_0 \Delta \nu +\e \Delta \nu^2)
+\frac{1}{3} \Delta \nu_{ss}.
\] 
Rewriting as a Fourier expansion 
\[
\Delta \nu (s,t) = \sum_{k\in \Z} \Delta \nu_k(t) e^{ iks}, 
\]
we obtain 
\[
\dfrac{\partial \Delta \nu_k}{\partial t}= 
\left(\frac{8}{3} \nu_0(t)- \frac{k^2}{3}\right) \Delta \nu_k 
+O(\e(\Delta \nu^2)_k),
\] 
where $(\Delta \nu^2)_k$ is the $k$th Fourier coefficient of 
$\Delta \nu^2$. This shows that for each {$|k|> \sqrt{8\, \nu_0(t)}$} 
such that $\Delta \nu_k\ne 0$, for $\e$ small enough 
this Fourier coefficient decays along the ALS flow.  
We will use this fact to prove that locally in time 
the functional $\Delta_q $ decays monotonically. 
See Lemma \ref{lem:monotone}.\\

\subsection{Preservation of rational caustics}
In this section we relate the presence of an integrable rational 
caustic of rotation number $1/q$  to properties of
resonant Fourier coefficients, \ie those with index divisible by $q$.\\

Let us first recall the following facts.
In  \cite[Section 4]{PR} they study small perturbation of ellipses. 
Following notations of \cite[Section 4]{PR} we have that  the
perimeter  
$$
L_\e = L_0+\e L_1+O(\e^2),
$$ 
is given by \cite[Formula (5)]{PR} (here we drop subindex $q$). 
Then by \cite[Proposition 4.1]{PR} 
(see also Proposition \ref{prop:caustic-preserve}) the linear term in 
$\e$ has the form 
\be \label{eq:res-sum}
L_1 (\f)=2\l \sum_{k=0}^{q-1} \mu_1 (\f_q^k),
\ee
where {$\l$ is the parameter associated to a given caustic 
(see also subsection \ref{subsec:ellipticdynamics})} and $\mu_1$ represents the first-order perturbation (in $\e$) of the boundary (see Section \ref{sec:preserva-rational}).\\

Let us now consider the usual polar coordinates and 
let $\Om=\{(\rho,\f):\, \rho=\rho_0\}$ be the circle centered 
at the origin and radius $\rho_0$. We are interested in studying small 
perturbations  given by 
\be \label{eq:family}
\Om_\e=\{(\rho_\e,\f): \ \rho_\e=\rho_0+\e \rho_1(\f)+O(\e^2)\},
\ee
where $\rho_1$ is a $C^r$ smooth function for $r\ge 2$.  
Assume by rescaling that $\rho_1(0)=1$. Expand the perturbation 
in Fourier series: 
\[
\rho_1(\f)=\sum_{j\in \Z} 
\rho_1^{(j)} e^{i j \f}.
\]
We show that for perturbations of the circle and for an appropriate 
choice of $q$, the existence of integrable rational caustics depends on resonant 
Fourier coefficients, \ie those  divisible by $q$. In fact, plug  
the rigid rotation $\f\mapsto \f+\frac{2\pi}{q}$ into 
(\ref{eq:res-sum}). Denote by $\Delta_q(\rho_1,\e)$ the value 
of the function $\Delta_{1/q}$  associated to the domain $\Om_\e$, 
as defined in \eqref{eq:Delta_q}. 
Using \eqref{eq:res-sum} we have that for some $c>0$ independent of $\e$ 
\begin{eqnarray*}  
L_\e &=& L_0 + c \e \sum_{k=0}^{q-1}\rho_1 \left(\f+\frac{2\pi k}{q}\right)  + O(\e^2)\\
&=&
L_0 + c \e \sum_{k=0}^{q-1} \sum_{j\in \Z} 
\rho_1^{(j)} e^{i j \left(\f+\frac{2\pi k}{q}\right)}
  + O(\e^2).
\end{eqnarray*}
Thus, 
\begin{eqnarray} \label{eq:res-sums}   
\Delta_q(\rho_1,\e)&:=& {c^2} \e^2 \, q^2\ \sum_{j\in \Z\setminus \{0\}}   (\rho^{(jq)}_1)^2 
+O(\e^3).  
\end{eqnarray}

\medskip

Consider the domain $\Omega_\e$ defined in \eqref{eq:family}. 
The vanishing of the  function $\Delta_q(\rho_1,\e)$ detects the existence of an integrable rational caustic of rotation number $1/q$. 
According to our computations this function has an asymptotic expansion 
\eqref{eq:res-sums}. Denote by $\Om^t_\e$ the image of $\Om_\e$ 
under the ALS flow \eqref{eq:ALS-flow}. 
 
\blm \label{lem:monotone}
Let $\rho_\e(\f),\ \e \ge 0$ be the family  of domains in \eqref{eq:family}. Assume 
that $q>2$ and that $\rho_1^{(q)}\ne 0$. Then, for $\e$ sufficiently small,  
the family of domains $\Om^t_\e$ for $0\le t\le \e$ satisfies 
\[
\frac{\partial \Delta^t_{q}(\rho_1,\e)}{\partial t}<0.
\]
\elm 

\begin{proof}
 Notice that up to a {affine-length} parametrization, $s$    and polar angle $\f$ are the same. 
Consider derivative with respect to 
the affine length shortening flow (\ref{eq:ALS-flow}) of 
$\Delta^t_{q}(\rho_1,\e)$. According to \eqref{eq:res-sums}, this leads 
to derivative of the resonant Fourier coefficients.  For each  $j>0$ we have that  
\[
\frac{\partial \rho_1^{(jq)}}{\partial t}= \e 
\left[\left(\frac 83 \rho_1^{(0)} -\frac{j^2q^2}{3}\right) \rho_1^{(jq)}+
\e \sum_{p\in \Z\setminus \{0\}}\rho_1^{(jq-p)}\rho_1^{(p)} +O(\e^2)\right]. 
\]
It follows from \eqref{eq:res-sum} and \eqref{eq:res-sums} that 
\[
\Delta^t_{q,\e}=c^2 \e^2 q^2 \sum_{j\in\Z\setminus\{ 0\}} \left(\rho_1^{(jq)}\right)^2+O(\e^3).
\]
Consider 
\[
\dfrac{\partial}{\partial t} \sum_{j\in\Z\setminus \{0\}} \left(\rho_1^{(jq)}\right)^2=
\]
\[
= \e \sum_{j\in\Z\setminus 0}
\left[\left(\frac 83 \rho_1^{(0)} -\frac{j^2q^2}{3}\right) \left(\rho_1^{(jq)}\right)^2+
\e \sum_{p\in \Z\setminus \{0\}}\rho_1^{(jq-p)}\rho_1^{(p)}\rho_1^{(jq)} +O(\e^2)\right]
\]

Since $\rho^{(q)}_1\ne 0$, for $\e$ small enough the last expression is negative.  
\end{proof}  

  \medskip
  
\noindent {\bf Concluding remarks:} 
This Lemma is certainly only an example of, what we believe, is
a much more general phenomenon. More specifically, we conjecture:  \medskip  
\begin{center}
{\it Monotonicity of 
the functional $\Delta_q$ along the ALS flow \eqref{eq:ALS-flow}
\footnote{As we pointed out if $\Delta_{q_k}\ne 0$ for some 
sequence $q_k\to \infty$ as $k\to \infty$, then the billiard
in non-integrable.}.}
\end{center} 
 \medskip 
The next step would be a local analysis of the ALS flow in 
a neighborhood of ellipses. It would be more challenging 
to extend this local analysis to the space of 
strictly convex domains and this will be an important step  to prove the global Birkhoff Conjecture.

\medskip

%%%%%%%%%%%%%%%%%%%%%%%%%%%%%%%%

\end{document}